\documentclass[11pt,reqno]{amsart}
\usepackage{amssymb,mathrsfs,color,mathtools,empheq, verbatim, epstopdf}
\usepackage{pinlabel}
\mathtoolsset{showonlyrefs}
\usepackage{hyperref} 
\usepackage{enumerate, bbm}
\usepackage{tensor}
\usepackage{graphicx}
\usepackage{xcolor} 
\usepackage{tensor}
\usepackage{slashed}
\usepackage{cite}
\usepackage[shortlabels]{enumitem}
\usepackage{geometry}
\numberwithin{equation}{section}

\newtheorem{thm}{Theorem}[section]
\newtheorem{cor}[thm]{Corollary}
\newtheorem{lem}[thm]{Lemma}
\newtheorem{prop}[thm]{Proposition}

\theoremstyle{definition} 
\newtheorem{rem}[thm]{Remark}

\theoremstyle{remark}
\newtheorem{claim}[thm]{Claim}

\def\bS {\mathbb{S}}
\def\bT {\mathbb{T}}

\def\cE {\mathcal{E}}

\def\cH {\mathcal{H}}

\def\cM {\mathcal{M}}

\def\cR {\mathcal{R}}

\newcommand{\la}{\langle}
\newcommand{\ra}{\rangle}
\newcommand{\La}{\big\langle}
\newcommand{\Ra}{\big\rangle}

\newcommand{\wto}{\rightharpoonup}

\newcommand{\Id}{\operatorname{Id}}

\newcommand{\vd}{\mathrm{d}}
\newcommand{\udr}{\,r\vd r}
\newcommand{\vD}{\mathrm{D}}

\definecolor{deepgreen}{cmyk}{1,0,1,0.5}

\newcommand{\E}{\mathcal{E}}

\newcommand{\C}{\mathbb{C}}

\newcommand{\R}{\mathbb{R}}

\newcommand{\Z}{\mathbb{Z}}

\newcommand{\al}{\alpha}
\newcommand{\be}{\beta}

\newcommand{\de}{\delta}

\newcommand{\om}{\omega}
\newcommand{\lam}{\lambda}
\newcommand{\te}{\theta}

\newcommand{\ta}{\tau}
\newcommand{\ka}{\kappa}

\newcommand{\De}{\Delta}

\newcommand{\Ga}{\Gamma}
\newcommand{\Lam}{\Lambda}

\newcommand{\p}{\partial}
\newcommand{\na}{\nabla}

\makeatletter

\newcommand{\Rmnum}[1]{\expandafter\@slowromancap\romannumeral #1@}
\makeatother

\newcommand{\ti}{\widetilde}

\newcommand{\U}{\underline}

\newcommand{\Del}[1]{}

\newcommand{\mif}{{\ \ \text{if} \ \ }}

\newcommand{\mas}{{\ \ \text{as} \ \ }}

\renewcommand\Re{\mathrm{Re}\,}
\renewcommand\Im{\mathrm{Im}\,}

\newcommand{\dr}{\, \mathrm{d}r}

\definecolor{green}{rgb}{0,0.8,0}

\renewcommand{\Re}{\mathrm{Re}}
\renewcommand{\Im}{\mathrm{Im}}

\newcommand{\ud}{\mathrm{d}}

\newcommand{\0}{\emptyset}

\newcommand{\veps}{\varepsilon}

\newcommand{\calE}{\mathcal E}

\newcommand{\calH}{\mathcal H}
\newcommand{\calI}{\mathcal I}

\newcommand{\calL}{\mathcal L}
\newcommand{\calM}{\mathcal M}

\newcommand{\calQ}{\mathcal Q}
\newcommand{\calR}{\mathcal R}

\newcommand{\calU}{\mathcal U}
\newcommand{\calV}{\mathcal V}

\newcommand{\ULam}{\U{\Lam}}

\allowdisplaybreaks[4]

\begin{document}
	\parindent=12pt
	\title[ ]{Construction of infinite time bubble tower solutions for energy critical equivariant Schr\"odinger maps}

\author[Y. Yin]{Yuchen Yin}
\address{School of Mathematical Sciences,
University of Science and Technology of China, Hefei 230026, Anhui, China.}
\email{yuchenyin@mail.ustc.edu.cn}
    
\begin{abstract}
   We prove the existence of infinite time bubble tower solutions for energy critical $k$-equivariant Schr\"odinger maps $\mathbb{R}^2\to \mathbb{S}^2$ with $k\geq 3$. More precisely, for any integers $k\geq 3$ and $N\geq 1$, we construct a $k$-equivariant solution that is global in one time direction and decomposes asymptotically into a superposition of $N$ bubbles with scales $\lambda_j(t)\sim|t|^{-\beta_j}$, where $\beta_j=\frac{1}{2}\left(\left(\frac{k}{k-2}\right)^{N-j}-1\right)$.  Our approach is a backward construction combined with modulation analysis and high order energy estimates. We use Darboux coordinates adapted to the multi bubble profile to express the remainder in a complex scalar form. A key ingredient is the  construction of a correction that cancels the leading effect of the nonlinear defect.
\end{abstract}
\maketitle
\tableofcontents
\section{Introduction}\label{Sec:1}

\subsection{Setting of the problem}

We consider the Schr\"odinger map
\begin{align}
    \label{SM}
    \begin{cases}
        \p_t u=u\wedge \De u,\\
        u(0)=u_0\in \dot H^1,
    \end{cases}\ (t,x)\in \R\times \R^2,\ u(t,x)\in \bS^2
\end{align}

The equation ~\eqref{SM} is  related to the Landau-Lifshitz equation in ferromagnetism and is a special case of the Schr\"odinger flow for maps from a Riemannian manifold to a K\"ahler manifold. 

The energy \footnote{We follow the convention in \cite{MRR-11},  where the energy is defined without the factor $\frac{1}{2}$ .  With this convention, the stationary solution $Q_k$ introduced below has energy $E(Q_k)=8\pi |k|$ . }
\begin{align}\label{eq:energy-def}
    E(u)=\int_{\R^2}|\na u(x)|^2\,\ud x
\end{align}
is conserved along the flow.

The equation ~\eqref{SM} can be rephrased as a Hamiltonian system. Following \cite{DW-98-SM} and denoting by $J_{u}$ the map defined by $J_{u}v=u\wedge v$ for $v\in \R^3$, we have
\begin{align}\label{SM-H}
    \p_t u=-\frac{1}{2}J_{u}\circ\vD E(u),
\end{align}
and the associated symplectic form is 
\begin{align}
    \om_{u}(\xi,\eta)=\int_{\R^2}(J_{u}\xi)\cdot \eta\ud x=\int_{\R^2}u\cdot (\xi\wedge \eta)\ud x
\end{align}
for functions $\xi,\eta$ taking values in $T_{u}\bS^2$.

The equation ~\eqref{SM} is energy-critical in the sense that ~\eqref{SM} and ~\eqref{eq:energy-def} are invariant under the scaling transform
\begin{align}
    u_{\lam}(t,x)=u\left(\frac{t}{\lam^2},\frac{x}{\lam}\right),\ \lam\in\R_{>0}.
\end{align}

In this paper, we are concerned with equivariant solutions. For an integer $k$,  the $k$-equivariant solution has the form
\begin{align}
    u(t,x)=e^{k\theta R}(u_1(t,r),u_2(t,r),u_3(t,r))^{\top},
\end{align}
where $R=\begin{pmatrix}
    0&-1&0\\1 &0&0\\0&0&0
\end{pmatrix}$, and  $(r,\te)$ are the polar coordinates on the plane. The homotopy degree of the map  is
\begin{align}
    \rm{deg}(u)\coloneq\frac{1}{4\pi}\int_{\R^2}u\cdot\left(\p_{x_1}u\wedge\p_{x_2}u\right)\ud x.
\end{align}
The spatial reflection $u(x_1,x_2)\to u(x_1,-x_2)$ maps $k$-equivariant solutions to $(-k)$-equivariant solutions. Therefore we restrict to $k\geq 0$.

The equation ~\eqref{SM} admits stationary solutions for each $k\in\Z$:
\begin{align}
    Q_k(\te,r)=
e^{k\te R}
\begin{pmatrix}
\sin\phi_k\\
0\\
\cos\phi_k
\end{pmatrix},
\end{align}
where $\phi_k(r)=2\arctan r^k$.  The solution $Q_k$ has degree $k$. From \cite{Bogomolnyi},  the harmonic maps $Q_k$ minimize the Dirichlet energy ~\eqref{eq:energy-def} with $E(Q_k)=8\pi|k|$.

\subsection{Background}

 Schr\"odinger maps have been studied extensively over the past decades. Local existence for smooth initial data goes back to Sulem, Sulem and Bardos \cite{SSB-SM-86}; see also the works of Ding and Wang \cite{DW-01-SM},  McGahagan \cite{Mc07-Cauchy}. The global theory for small data at the critical regularity  was developed by Bejenaru, Ionescu, Kenig, and Tataru \cite{BIKT-SM-11} for target manifold being $\bS^2$,  and was extended by Li \cite{Li-21-SM},\cite{Li-23-SM} to general compact K\"ahler targets. For general large data, Rodnianski, Rubinstein  and Staffilani \cite{RRS-09-SM}  proved the gloabl wellposedness when the domain manifold is $\R^1$ or $\bS^1$.

Under equivariant symmetry with target manifold being $\bS^2$, more results have been established.  For $1$-equivariant Schr\"odinger maps, Bejenaru, Ionescu, Kenig, and Tataru \cite{BIKT-13} proved that solutions with energy less than $E(Q_1)$ are global in $\dot H^1\cap\dot H^3$ and scatter. For $k\geq 3$, $Q_k$ is shown to be asymptotically stable under $k$-equivariant perturbation by Gustafson, Kang, Nakanishi and Tsai \cite{GKT-07-SM}, \cite{GKT-08-SM}, \cite{GNT-10-SM}.      The stability property of $Q_k$  was studied by  Bejenaru  and  Tataru \cite{BT-14-SM} for $k=1$ where $Q_1$ was shown to be unstable in $\dot H^1$ but to be stable under $1$-equivariant perturbation in a space $Y$ such that $H^1\subset Y\subset \dot H^1$, and by Bejenaru,  Pillai and Tataru \cite{BPT} for $k=2$  where $Q_2$ was shown to be stable under $2$-equivariant perturbation in $\dot B_{2,1}^1$.

\bigskip

A guiding principle for the description of long time dynamics of Schr\"odinger maps is the soliton resolution conjecture,  which predicts that every solution with finite energy resolves into a superposition of asymptotically decoupling harmonic maps and free radiation. Such results have been established in several models, including equivariant energy critical wave maps and radial focusing energy critical wave equations, see \cite{JL-WM-SRC}, \cite{DKM-NLW-SRC-ODD}, \cite{JL-NLW-SRC}, \cite{DJKM-NLW-nonradial} and the references therein. For Schr\"odinger maps, however, the soliton resolution conjecture remains open.

From the quantitative perspective, the interaction between bubbles and radiation has been extensively studied. For $1$-equivariant Schr\"odinger map, Merle, Rapha\"el and Rodnianski \cite{MRR-11} proved the existence of a $1$-codimensional set of smooth well-localized initial data, which generates type II blowup solutions of blow rate of order $ \frac{|T-t|}{|\log(T-t)|^2}$, and Perelman \cite{Galina-SM} constructed the solutions with prescribed blow up rate $t^{\frac{1}{2}+\nu}$ with arbitrary $\nu>1$  following the method in \cite{KST-08}.

However, the dynamics involving the interaction among bubbles have been less studied. In particular, prior to the present work, no multi bubble construction has been established for Schr\"odinger maps to the author's knowledge. As for focusing energy critical Schr\"odinger equations, two bubble solutions have been constructed by Jendrej \cite{J-17-NLS} and by Jendrej, Li and Xu \cite{JLX-NLS-26}. 

\bigskip

There is also a physical motivation for the study of  bubble tower solutions. They can be viewed as an analogue of an  $n\pi$-skyrmion  appearing in the study of topological magnetic textures. In the micromagnetic interpretation, a single harmonic map corresponds to a Belavin-Polyakov skyrmion, whereas an alternating bubble tower  produces a sequence of concentric magnetization reversals at different spatial scales.  See \cite{Phy-1}, \cite{Phy-2}, \cite{Phy-3} for physics results. We emphasize, however, that this analogy is only qualitative.

\subsection{Main result}

Our goal is to provide the first construction of bubble tower solutions to ~\eqref{SM} with arbitrary number of bubbles.

Before we state our main result, we introduce some notation. Since $k$ is fixed in the rest of the paper, we will suppress the subscript $k$ of $Q_k$, $\phi_k$ and write $Q=Q_k$, $\phi=\phi_k$ for convenience.

We use the convention $e_1=(1,0,0)^{\top},e_2=(0,1,0)^{\top},e_3=(0,0,1)^{\top}$.  Define $Q_{\lam}=Q(\te,\lam^{-1}r)$. For $\vec p=(\lam_1,\dots,\lam_N,\te_1,\dots,\te_N)$ we define
\begin{align}
    \calQ_{\vec p}=\calQ_{\vec \lam,\vec \te}=\calQ_{(\lam_1,\dots,\lam_N,\te_1,\dots,\te_N)}=\sum_{j=1}^N(-1)^{j-1}\left(e^{\te_j R}Q_{\lam_j}-e_3\right)+e_3,\ \calU_{\vec p}=\frac{\calQ_{\vec p}}{|\calQ_{\vec p}|}\ \mif\ |\calQ_{\vec p}|\neq 0.
\end{align}
In the parameter regime considered in this paper, $|\calQ_{\vec p}|$ is uniformly bounded away from zero, so that $\calU_{\vec p}$ is well defined. The parameters $\vec p$  are usually taken to depend on $t$, and we slightly abuse the notation $\calQ(t,\te,r)=\calQ_{\vec p(t)}(\te,r)$, $\calU(t,\te,r)=\calU_{\vec p(t)}(\te,r)$ when the parameters are clear from the text. 

To measure the difference between the solution and the $\calQ_{\vec p}$, for a vector-valued function $u(\te,r)=e^{k\te R}v(r)$, we introduce the space
\begin{align}
        \|u\|_{\dot \calH^1}^2=2\pi\int_0^{\infty}\left(|v_r|^2+\frac{k^2}{r^2}|v|^2\right)\udr.
    \end{align}
For definition of other spaces, see Section~\ref{subsec:notation}.

\begin{thm}\label{thm:main}
    For any integers $k\geq 3$ and $N\geq 1$, there exists a time $T_0\ll 1$ such that there exists a solution $u\in C_t\left((-\infty,T_0];\dot H_x^1\cap\dot H_x^3\right)$ to \eqref{SM} satisfying 
    \begin{align}\label{eq:main}
    \left\|u(t)-\calQ_{\vec p(t)}\right\|_{\dot \calH^1}\to 0\ \mas \ t\to -\infty,
\end{align}
where 
\begin{align}
    \lam_j(t)=q_j|t|^{-\be_j},\te_j=-\frac{1+(-1)^{j}}{2}\frac{\pi}{2}.
\end{align}
The constants $\be_j=\frac{1}{2}\left(\left(\frac{k}{k-2}\right)^{N-j}-1\right)$, $q_N=1$,
\begin{align}\label{eq:def-beta-q}
    q_j=\left(\frac{\be_j}{\ka_k}q_{j+1}^k\right)^{\frac{1}{k-2}},\ \ka_k=\frac{4k^2}{\pi}\sin \frac{\pi}{k}.
\end{align}

\end{thm}

\begin{rem}
    Using Hardy's inequality, ~\eqref{eq:main} also holds in $L^{\infty}$.
\end{rem}
\begin{rem}
    To the author's  knowledge, this is the first construction of bubble tower solutions for Schr\"odinger maps. Our result can be viewed as an analogue to solutions constructed in \cite{HK-26-WM} for wave maps. The author expects that the other results in the study of wave type equations such as \cite{JK-25-WM}, \cite{KP-26-WM}  can be adapted to Schr\"odinger maps.
\end{rem}
\begin{rem}
    For $k=2$, we expect infinite time bubble solutions exist in view of the formal modulation system ~\eqref{eq:modulation-formal}. Note that \cite{JLX-NLS-26} constructed two-bubble solutions for Schr\"odinger equations in dimension $6$, which corresponds to the case $k=2$. See also \cite{JK-classification-6d} for the study of harmonic map heat flows when $k=2$. For $k=1$, we refer to Shen \cite{Shen-26-NLW}, Kim \cite{Kim-HMHF-26-no-bubble-tree}, where they separately proved the nonexistence of solutions with more than one bubble  for radial focusing energy critical wave equations in dimension $3$ and for $1$-equivariant harmonic map heat flows and  radial focusing energy critical heat equations in dimension $3,4,5$. The author expects that the same conclusion holds for $1$-equivariant Schr\"odinger map.
\end{rem}
\begin{rem}
    Multi bubble solutions have been extensively studied in other dispersive models. We mention some related works. For energy-critical equivariant wave maps,  in \cite{JK-25-WM} Jendrej and Krieger constructed a $2$-equivariant solution that blows up in finite time with two concentrating concentric bubbles plus a radiation, which was extended to arbitrary number of bubbles by Krieger and Palacios in \cite{KP-26-WM}.  In \cite{HK-26-WM} Hwang and Kim constructed $k$-equivariant pure multi bubble solutions that are global in one time direction when $k\geq 3$. For energy critical wave equations, see Jendrej  \cite{J-17-NLW},  Martel and Merle \cite{MM-NLW-16}, \cite{MM-NLW-26}, Jendrej and Martel\cite{JM-NLW-20},  Gu, Jendrej and Zhao \cite{GJZ-26-construction-4d-NLW}. For classification results, see \cite{JL-18-threshold-two-bubble-WM}, \cite{JZZ-26-rigidity-5d-NLW}, \cite{MM-26-classification-5d-NLW}.  We also mention \cite{CZ-13-NLW},  \cite{MR-NLS-18-mass}, \cite{Merle-NLS-90}.     As for harmonic map heat flows as the parabolic counterpart of ~\eqref{SM}, see \cite{DPW-20}, \cite{KM-classification}, \cite{TaegyuKim-26}, \cite{JK-classification-6d} for some results.
\end{rem}

\subsection{Strategy of the proof}

Our overall framework is inspired by \cite{Merle-NLS-90}, \cite{Martel-05-KdV},  \cite{J-17-NLS}, \cite{MR-NLS-18-mass},  while the technical implementation is closer in spirit to that of \cite{HK-26-WM}. There are, however, several differences in the present setting. In particular, the $k$-equivariant Schr\"odinger map equation is not a scalar equation, and the underlying geometric structure gives rise to a number of new issues that have no direct analogue in the scalar setting. Nevertheless, we describe the full proof strategy below, emphasizing along the way the new ingredients required by ~\eqref{SM}.

\bigskip

In general, we want to construct a sequence of solutions $u_n$ such that the lifespan of $u_n$ contains $[T_n,T_0]$ for a sequence $T_n\to-\infty$ and a uniform $T_0$. Each of the solutions stays close to the given multi bubble configuration. Finally, the desired solution is obtained by compactness argument.

To construct $u_n$, one has to choose suitable initial data at $T_n$. Since $\calQ$ does not lie in $\bS^2$, it is natural to consider $\calU=\frac{\calQ}{|\calQ|}$.  We use modulation technique to decompose the solution into a superposition of rescaled and rotated bubbles plus an infinite dimensional remainder along the flow, and the problem is then reduced to describing the dynamics of the parameters of the bubbles and control the energy of the remainder. Precisely we use the decomposition of the form
\begin{align}
    u(t)=\calU_{\vec p(t)}+g(t),
\end{align}
where $g$ satisfies suitable orthogonal conditions, see ~\eqref{eq:g-ortho}, where we use localized version due to the slow decay of the kernels when $k=3,4$. Note that we decompose the solution as $u=\calU+g$ rather than as $u=\calQ+g$, which is more convenient since the initial data at $T_n$ of $g$ is zero, and the identity $g\cdot \calU=-\frac{1}{2}|g|^2$  simplifies our analysis. We also note that we choose a logarithm cutoff function in the orthogonal conditions. This is due to the slow decay of the ground state $Q$ when $k=3$. The usual cutoff function of the form $\chi(r/L)$ is not sufficient for the proof of Lemma~\ref{lem:b-lam-close}, and the logarithm cutoff function provides an extra decay of the derivative of $\chi_L$. Such logarithm cutoff function was also used to construct suitable Morawetz corrector in \cite{KK-CSS}.  

Differentiating the orthogonal conditions leads to a system of $\lam'_j(t)$ and $\te_j'(t)$. Ignoring the cutoff function and other error terms we obtain a formal ODE system:
\begin{equation}\label{eq:modulation-formal}
    \begin{aligned}
        \begin{cases}
        \lam_j\lam_j'=\ka_k(-1)^{j-1}\left[ \left(\frac{\lam_j}{\lam_{j+1}}\right)^k\sin\left(\te_j-\te_{j+1}\right)-\left(\frac{\lam_{j-1}}{\lam_{j}}\right)^k\sin\left(\te_{j-1}-\te_{j}\right) \right],\\
        \lam_j^2\te_j'=k\ka_k(-1)^{j-1}\left[ \left(\frac{\lam_j}{\lam_{j+1}}\right)^k\cos\left(\te_j-\te_{j+1}\right)-\left(\frac{\lam_{j-1}}{\lam_{j}}\right)^k\cos\left(\te_{j-1}-\te_{j}\right) \right]
    \end{cases}
    \end{aligned}
\end{equation}
where we adopt the convention  $\lam_0=0,\lam_{N+1}=\infty$  .This suggests the asymptotic behaviors of the parameters as $\lam_j(t)\sim q_j|t|^{-\be_j}$, $\te_j(t)\sim-\frac{1+(-1)^{j}}{2}\frac{\pi}{2}$, where $q_j$ and $\be_j$ are defined in ~\eqref{eq:def-beta-q}.

The remainder $g$ enters the actual modulation equations for the parameters at quadratic and higher orders, and the quadratic part is hard to control only in the critical energy norm. More results addressing this issue have been obtained for semilinear energy critical focusing wave equations or energy critical wave maps. One approach is to combine a virial correction with energy conservation to control the modulation equations. See \cite{J-17-NLW}, \cite{JL-18-threshold-two-bubble-WM}, \cite{JL-WM-SRC}, \cite{JL-NLW-SRC}, \cite{JZZ-26-rigidity-5d-NLW}, \cite{GJZ-26-construction-4d-NLW} for examples. From the construction perspective, another natural strategy is to work in a higher regularity space, which provides more room for closing the bootstrap argument. Such strategy was already used in \cite{MRR-11}, \cite{RR-WM-12}, \cite{RS-HMHF-14}, \cite{Jeong-WM-25} for some examples, and was recently applied to construct bubble tower solutions for energy critical equivariant wave maps in \cite{HK-26-WM}. This is also the strategy adopted in the present paper. 

\bigskip

We consider the bootstrap regime where the Sobolev regularity of $g$ is controlled up to order $3$, which is natural due to the scaling relation $\p_t\sim \De_x$. The difficulty is that the equation satisfied by $g$ is actually quasilinear, and we have to choose a suitable energy functional to avoid loss of derivative. Although such a functional can be constructed directly from the equation for $g$ from a purely analytic point of view,  we choose to derive it instead from the underlying geometry. 

More precisely,  we work in a local Darboux coordinate chart adapted to $\calU$, in which the deviation between $u$ and $\calU$ can be naturally understood as a tangent vector. Using the notation in Section ~\ref{subsec:Darboux} with $v,\calV$ replaced by $u,\calU$ and suppressing the superscript, we can derive the evolution equation for $z=\Phi^{-1}(u)$. Using the Hamiltonian form ~\eqref{SM-H}, we immediately obtain $iG_z\De z$ as the principal part of pullback of $-\frac{1}{2}J_{u}\circ\vD E(u)$. On the other hand, since the pullback of the Hessian of energy at $z=0$ is $\calH_{\calU}$ up to constant, we use $-iG_z\calH_{\calU}z$ as the principal part. Since in complex plane we have the canonical isometry $T_w\C \cong \C$ for any $w\in \C$, we shall also regard $z$ and its derivatives as the tangent vectors at the origin, which leads naturally to the definition of $\calE_{3,0}$ in ~\eqref{eq:def-energy-functional-E3}.  

However, the control of $\calE_{3,0}$ cannot be closed by itself due to two main difficulties.  The first is that the time derivative of $\calE_{3,0}$ contains the terms involving modulation parameters which are not integrable in time. We adapt the Morawetz functional for wave maps introduced in \cite{HK-26-WM} to our setting to solve this. Here again, the Morawetz functional should be viewed as a symplectic form $\om(\xi,\na_{X}\eta)$ with a special vector field $X$. This interpretation explain the motivation of the early introduction of Darboux coordinates, in which the expression of the symplectic form is simplified, as shown in ~\eqref{eq:Darboux}.  The Morawetz functional we used is defined in  ~\eqref{eq:def-Morawetz-functional}. We  note that logarithm cutoff is also used in the proof of Morawetz estimates near one bubble, see Lemma~\ref{lem:Morawetz-one}.

The second comes to the nonlinearity of the equation ~\eqref{SM} which generates the error from $\p_t\calU-\calU\wedge\De\calU$. This is a new obstruction in our setting, which can be neglected in wave maps as in \cite{HK-26-WM} due to the different scaling relation for $t,x$. We will show how to overcome this problem heuristically. Ignoring the quasilinear nature and other things, we consider the toy model
\begin{align}
    \p_tz=-iHz+F-F_1,
\end{align}
where $z$, $H$, $F$, $F_1$ is the expression of $g$, the linearized operator $-\De-|\na\calU|^2$ , $\calU\wedge\De\calU$, $\p_t\calU$  under certain coordinate chart. The operator $H$ is self-adjoint and $\R$-linear.  Consider the energy functional $\La Hz,H^2 z \Ra $, taking times derivatives and ignore the term containing $\p_t H$ yields that
\begin{align}
    \frac{\ud  }{\ud t}\La Hz,H^2z \Ra=& 2\La H^2z,H(-iHz+F-F_1)  \Ra\\
    =&2\La H^2z,[H,-i]Hz \Ra+2\La H^2 z,H(F-F_1) \Ra.
\end{align}
The term containing commutator is harmless since the bad terms in $H$ cancel in $[H,i]$. Since $\p_t\calU=\sum_{j}\left(\lam_j'\p_{\lam_j}\calU+\te_j'\p_{\te_j}\calU\right)$, and $\p_{\lam_j}\calU, \p_{\te_j}\calU$ is the approximate kernel of $H$ due to the symmetry of ~\eqref{SM}, the only bad term is $ 2\La H^2 z,HF\Ra=2\La iHz,iH^2F   \Ra$, which leads to the correction term
$ \La iH^2F,z \Ra $, as defined in ~\eqref{eq:def-energy-functional-E3}.  The term $\p_t\calU-\calU\wedge\De\calU$ also generates an error in estimating the time derivative of Morawetz functional, where $F_1$ cannot be canceled. It seems difficult to find a suitable correction on the functional, and the method to overcome this is to add a correction $\eta$ to $z$ and use $z+\eta$ instead of $z$ in the Morawetz correction. Formally we want to find $\eta$ such that
\begin{align}
    \p_t(z+\eta)=-iH(z+\eta),
\end{align}
 which compared with the equation satisfied by $z$ yields
\begin{align}
    \p_t\eta=-iH\eta-F+F_1.
\end{align} 
We prefer to solve the stationary equation
\begin{align}
    -iH\eta=F-F_1
\end{align}
at each time $t$ and estimate $\p_t \eta$ using $\eta=(-iH)^{-1}(F-F_1)$. However, even if we manage to solve this stationary equation, we will have to estimate $\lam_j''$ and $\te_j''$ when estimating $\p_t\eta$ due to the expression of $F_1$. We solve this difficulty by introducing $2N$ parameters $b_j^{\lam}$, $b_j^{\te}$ to replace $ \sum_{j}\left(\lam_j'\p_{\lam_j}\calU+\te_j'\p_{\te_j}\calU\right)$ by $\sum_{j}\left(\frac{b_j^{\lam}}{\lam_j}\p_{\lam_j}\calU+\frac{b_j^{\te}}{\lam_j^2}\p_{\te_j}\calU\right) $, and solve the new equation
\begin{align}
    -iH\eta=F-F_{1,b}.
\end{align}
The auxiliary parameters $b_j^{\lam}$ and $b_j^{\te}$ can be uniquely determined once we impose suitable orthogonal conditions. By taking time derivatives of these conditions we obtain that $b_j^{\lam}$ and $b_j^{\te}$ have controllable time derivatives. Inserting the equation satisfied by $\eta$, $b_j^{\lam}$ and $b_j^{\te}$ into modulation equations ~\eqref{eq:rough-modulation-lambda-raw} and ~\eqref{eq:rough-modulation-theta-raw} yields that $b_j^{\lam}$ and $b_j^{\te}$ are close to $\lam_j\lam_j'$ and $\lam_j^2\te_j'$ respectively. For details, see Lemma ~\ref{lem:correction-z-eta} and Lemma~\ref{lem:b-lam-close}. 

The  parameters $b_j^{\lam}$ and $b_j^{\te}$ play another important role in closing the high regularity estimates. Since the $(\dot\calH^3)^*$ norm of the localized kernel in the orthogonal conditions tends to $\infty$ as $L\to \infty$ when $k=3,4$,  the implicit constants may depend on the size of the cutoff function, see ~\eqref{eq:g-ortho}, ~\eqref{eq:HU-coercove-H3}, ~\eqref{eq:E-30-coercive}. As a consequence, rough modulation estimates ~\eqref{eq:lam'-rough-mod} and ~\eqref{eq:te'-rough-mod} are not sufficient by themselves to control some critical terms of order $|t|^{-2-2\veps_0}$ appearing in the  time derivative of the energy functional, since there is no quantitative relation between the implicit constants in  ~\eqref{eq:lam'-rough-mod} and ~\eqref{eq:te'-rough-mod} and $c_L$ in ~\eqref{eq:E-30-coercive}. However, the parameters $b_j^{\lam}$ $b_j^{\te}$ have more precise estimates, and the difference $b_j^{\lam}-\lam_j\lam_j'$, $b_j^{\te}-\lam_j^2\te_j'$ can be controlled by some terms that can be absorbed by the Morawetz functional and other terms that is sufficiently small and can absorb the constants about $L$, which makes it possible to control the terms containing $|\lam_j'|$ or $|\te_j'|$ when rough modulation estimates ~\eqref{eq:lam'-rough-mod} and ~\eqref{eq:te'-rough-mod} are not sufficient. See the estimates for $(\rm{A})$ and $(\rm{B})$ in Section ~\ref{subsec:I3} for example.

\bigskip

After improving the estimates for $g$, the estimates for parameters $\lam_j$ and $\te_j$ can be improved easily. A formal calculation on the time derivatives of $\frac{\lam_j-\lam_j^{\infty}}{\lam_j^{\infty}}$ and $\te_j-\te_j^{\infty}$ yields that for $1\leq j\leq N$, $\te_j$ is stable and can be improved directly using the modulation equation. The same thing holds for $\lam_N$. For $1\leq j\leq N-1$, $\lam_j$ is unstable and can be determined using a classical topological shooting argument. However, the parameters $\lam_j(t)$ and $\te_j(t)$ are imprecise proxies for the dynamics since the orthogonal conditions are imposed related to the localized kernels.  We introduce refined parameters $\hat \lam_j$ and $\hat\te_j$ in ~\eqref{eq:def-refined-parameter} in spirit of (4.46) in \cite{JL-WM-SRC} and of (5.49) in \cite{JL-NLW-SRC}. The dynamics of refined parameters can be described explicitly and we use topological shooting method with respect to $\hat\lam_j$. Then we get the solution $u_n$ defined on $[T_n,T_0]$, and the desired solution is obtained by compactness arguments.

\subsection{Notation}\label{subsec:notation}
    For $\be_j$ and $q_j$ defined in Theorem ~\ref{thm:main},  we define
    \begin{align}
        \lam_j^{\infty}(t)=q_j|t|^{-\be_j},\ \te_j^{\infty}=-\frac{1+(-1)^{j}}{2}\frac{\pi}{2}.
    \end{align}

    Choose a smooth cutoff function $\chi(r)$ such that $\chi=1$ for $-\infty\leq r\leq 0$, $\chi=0$ for $r\geq 1$ and $0\leq \chi\leq 1$. Define
    \begin{align}
        \chi_{L}(r)=\begin{cases}
            \chi(\frac{\log r-\log L}{\log L}),\ r\neq 0,\\
            1,\ r=0.
        \end{cases},\quad   \chi_{L,j}=\chi_L(r/\lam_j)=\chi_{L}(|x|/\lam_j).
    \end{align}

    For $u(\te,r)=e^{k\te R}v(r)$, $\dot H^1$ and $\dot H^3$ denote the standard Sobolev spaces.
    \begin{align}
        \|u\|_{\dot H^1}^2&=2\pi\int_0^{\infty}\left(|v_r|^2+\frac{k^2}{r^2}(v_1^2+v_2^2)\right)\udr
        =2\pi\int_0^{\infty}\left(|\p_r u|^2+\frac{k^2}{r^2}|Ru|^2\right)\udr.
    \end{align}
    \begin{align}
        \|u\|_{\dot H^3}^2=\|\De_{\R^2} u\|_{\dot H^1}^2.
    \end{align}
    We define
    \begin{align}
        \|u\|_{\dot \calH^3}^2=2\pi\int_0^{\infty}\left(|\p_r^3 v|^2+\frac{|\p_r^2 v|^2}{r^2}+\frac{|\p_r v|^2}{r^4}+\frac{|v|^2}{r^6}\right)\udr.
    \end{align}

    For $z(r)=z_1(r)+iz_2(r)$  with $z_1,z_2\in \R$, we define
    \begin{align}
        |z|_{-j}=\left(\sum_{l=0}^j \frac{|\p_r^{l}z|^2}{r^{2(j-l)} }\right)^{\frac{1}{2}},\ j=1,2,3,
    \end{align}
    
    \begin{align}
        \|z\|_{\calH^j}^2=2\pi\int_0^{\infty}|z|_{-j}^2\udr,
    \end{align}
    \begin{align}
        \|z\|_{\rm{Mor}(\lam)}^2=2\pi\int_0^{\infty}\frac{\lam}{(\lam+r)^2}|z|_{-3}^2\udr,\ \|z\|_{\rm{Mor_2}(\lam)}^2=2\pi\int_0^{\infty}\frac{\lam}{(\lam+r)^2}|z|_{-2}^2\udr.
    \end{align}
    
    For $a=(a^1,a^2,a^3)$, $b=(b^1,b^2,b^3)$, we define
    \begin{align}
    \na a\cdot\na b:=\sum_{j=1}^{2}\sum_{\alpha=1}^{3}\partial_j a^{\alpha}\,\partial_j b^{\alpha}.
    \end{align}

    For vector-valued functions $u,v$, we define
    \begin{align}
        \La u,v\Ra=\int_{\R^2}u\cdot v\,\ud x.
    \end{align}

    For complex-valued functions $z,w$, we define
    \begin{align}
        \La z,w\Ra=\Re\int_{\R^2}\bar z w\,\ud x.
    \end{align}

    For vector- or complex-valued function $f$, we define
    \begin{align}
        \Lam f= x\cdot\na f,\ \underline{\Lam}f=(x\cdot\nabla+1)f,
    \end{align}
    and for $\lam>0$, we define
    \begin{align}
        f_{\lam}(x)=f\left(\frac{x}{\lam}\right),\ f_{\underline{\lam}}(x)=\frac{1}{\lam}f\left(\frac{x}{\lam}\right).
    \end{align}

    For parameters $0<\lam_1<\cdots<\lam_N$, we will always adopt the convention $\lam_0=0$, $\lam_{N+1}=\infty$. For $a>0$ we also adopt the convention $a\cdot\infty=\infty$, $\frac{a}{\infty}=0$, $\sqrt{\infty}=\infty$.
    
    For $a,b>0$, we write $a\lesssim b$ (resp. $a\gtrsim b$), if there exists a constant $C>0$ depending only on $k,N$. such that $a\leq Cb$  (resp. $a\geq Cb$). We write $a\sim b$ if $a\lesssim b$, $b\lesssim a$. We write $a\lesssim_P b$ or $a\gtrsim_P b$ if the implicit constant also depends on other parameters $P$. We use $C(P)$ to denote general constants that only depend on $k,N,P$.
    
\bigskip

\noindent
\textbf{Acknowledgments.}  The author would like to thank his advisor, Prof. Lifeng Zhao, for his valuable advice  on this work. This work was supported by NSFC Grant of China No. 12271497.

\bigskip

\noindent
\textbf{AI use disclosure.} GPT-5.6 Sol  was used to assist with language editing and to check calculations.  All calculations, estimates and arguments in the manuscript were derived and verified by the author. The author takes full responsibility for the mathematical content.

\section{Preliminaries}

\subsection{Local Cauchy theory}
In this section we recall the local Cauchy theory of ~\eqref{SM}. 

\begin{lem}[\rm{\cite[Theorem 3.3,  Corollary 4.2 ]{Mc07-Cauchy}}]
    \label{lem:Cauchy}
    Let $u_0\in\dot H_x^1\cap\dot H_x^3$, then there exists a time \\
    $T=T\left(\|u_0\|_{\dot H_x^1},\|u_0\|_{\dot H_x^3}\right)>0$ such that ~\eqref{SM} has a unique solution $u$ in $C_t([0,T];\dot H^1_x\cap\dot H^3_x)$. 

    Moreover, given a time interval $[T_1,T_2]$ and a constant $M>0$.  For any $\veps>0$, there exists $\de=\de(\veps,M,T_1,T_2)>0$ such that if $u_1,u_2\in L^{\infty}_t([T_1,T_2];\dot H^1_x\cap\dot H^3_x)$ are solutions to ~\eqref{SM} such that
    \begin{align}
        \|u_1\|_{L^{\infty}_t([T_1,T_2];\dot H^1_x\cap\dot H^3_x)}+\|u_2\|_{L^{\infty}_t([T_1,T_2];\dot H^1_x\cap\dot H^3_x)}\leq M, \ \|u_{1}(T_1)-u_{2}(T_1)\|_{L_x^2\cap\dot H_x^1}\leq \de,
    \end{align}
    then
    \begin{align}
        \|u_{1}-u_{2}\|_{L_t^{\infty}([T_1,T_2];L^2_x\cap\dot H_x^1)}\leq \veps.
    \end{align}
\end{lem}
\begin{rem}
    In fact, \cite{Mc07-Cauchy}  only proved that $u\in L^{\infty}_t([0,T];\dot H^1_x\cap\dot H^3_x)$, which, combined with ~\eqref{SM}, only implies $u\in C_t([0,T];\dot H^1_x)$ and thus $u(s)\wto u(t)$ in $\dot H_x^3$ as $s\to t$ for $t\in [0,T]$. However, for each $t\in [0,T)$ we can approximate $u$ at time $t$ with data $(u(t),0)$ using the same hyperbolic equation defined in \cite{Mc07-Cauchy}. Applying the same energy estimates in the proof of Theorem 3.3 in \cite{Mc07-Cauchy}, we obtain that the $\dot H_{\rm{co}}^3$ norm of $u$ is continuous in time, where $\dot H_{\rm{co}}^3$  is defined with covariant derivatives, which combined with the weak convergence of $u(t)$ in $\dot H_x^3$ and the continuity of $u(t)$ in $\dot H_x^1$ yields $u\in C_t([0,T);\dot H_x^3)$.  We note that the  local Cauchy theory stated in Theorem 1.1 \cite{RRS-09-SM} also says the solution is continuous  in time in the highest order Sobolev space, with no proof provided.
\end{rem}
\begin{rem}
    The continuous dependence on initial data stated above is also not presented directly in \cite{Mc07-Cauchy}, but follows by the same argument as in the proof of Theorem 4.1 in \cite{Mc07-Cauchy}.
\end{rem}

\begin{cor}
    For any initial data $u_0\in\dot H_x^1\cap\dot H_x^3$, there exists $T_+>0$ such that the solution  $u\in C_{t,\rm{loc}}([0,T_+);\dot H_x^1\cap\dot H_x^3)$ and satisfies $\limsup_{t\to T_+}\|u(t)\|_{\dot H_x^1\cap\dot H_x^3}=\infty$ if $T_+<\infty$. 
\end{cor}

\subsection{Basic estimates for stationary solutions}

In this section we collect some basic calculations and estimates for the stationary solution $Q$. For estimates of other objects with coherent structures appearing in the rest of the paper, see Appendix \ref{Appendix:estimates}.

By definition we have
\begin{align}
    Q=e^{k\te R}\begin{pmatrix}
        \frac{2r^k}{1+r^{2k}}\\ 0 \\ \frac{1-r^{2k}}{1+r^{2k}}
    \end{pmatrix},\ RQ=e^{k\te R}\begin{pmatrix}
        0 \\ \frac{2r^k}{1+r^{2k}} \\ 0
    \end{pmatrix},\ \Lam Q=e^{k\te R}\begin{pmatrix}
        \frac{2kr^k(1-r^{2k})}{(1+r^{2k})^2}\\ 0 \\ -\frac{4kr^{2k}}{(1+r^{2k})^2}
    \end{pmatrix},
\end{align}
and we have the following identities:
\begin{align}
    Q\cdot RQ=Q\cdot\Lam Q=RQ\cdot \Lam Q=0,
\end{align}
\begin{align}
    Q\wedge RQ=-\frac{1}{k}\Lam Q,\ Q\wedge \Lam Q= kR Q,\ RQ\wedge \Lam Q=-k|RQ|^2 Q.
\end{align}
\begin{align}
     \|RQ \|_{L^2}^2=\frac{4\pi^2}{k^2\sin\left(\frac{\pi}{k}\right)},\|\Lam Q\|_{L^2}^2=\frac{4\pi^2}{\sin\left(\frac{\pi}{k}\right)},
\end{align}
\begin{align}
    \De Q=-|\na Q|^2 Q=-2r^{-2}|\Lam Q|^2 Q.
\end{align}
For $\lam>0$, $\al\in \R$, we have
\begin{align}
    e^{\al R}Q_{\lam}(\te,r)=e^{k\te R}\begin{pmatrix}
        \frac{2(r/\lam)^k}{1+(r/\lam)^{2k}}\cos \al\\ \frac{2(r/\lam)^k}{1+(r/\lam)^{2k}}\sin \al \\ \frac{1-(r/\lam)^{2k}}{1+(r/\lam)^{2k}}
    \end{pmatrix},
\end{align}
which yields
\begin{equation}\label{eq:Q-rescaled-expansion}
    \begin{aligned}
        e^{\al R}Q_{\lam}(\te,r)=\begin{cases}
        2\left(\frac{r}{\lam}\right)^ke^{k\te R}\begin{pmatrix}
            \cos \al\\ \sin \al \\ 0
        \end{pmatrix}+e_3+O\left(\left(\frac{r}{\lam}\right)^{2k}\right), &r\leq\lam, \\
        2\left(\frac{\lam}{r}\right)^ke^{k\te R}\begin{pmatrix}
            \cos \al\\ \sin \al \\ 0
        \end{pmatrix}-e_3+O\left(\left(\frac{\lam}{r}\right)^{2k}\right), &r\geq\lam,
    \end{cases}
    \end{aligned}
\end{equation}
and the remainder keeps the same order when we apply $(r\p_r)^a$ for $1\leq a\leq 4$.\\
For $\lam>0$ we define
\begin{align}
    F_{\lam}(r)=\min\left\{\left(\frac{r}{\lam}\right)^k, \left(\frac{\lam}{r}\right)^k  \right\}.
\end{align}
For later use, we note that the scaling parameters $\lam_j$ appearing from Section 3 onward will depend on time, and we will write
\begin{align}\label{eq:def-Fj}
    F_{j}(t,r)=F_{\lam_j(t)}(r), \ 1\leq j\leq N
\end{align}
for convenience, while in this section, the parameters are still regarded as fixed.
Then for $a\geq 1$, we have
\begin{align}\label{eq:RQ-LamQ-F}
    |(r\p_r)^{a} e^{\al R}RQ_{\lam} |+|(r\p_r)^{a} e^{\al R}\Lam Q_{\lam} |\lesssim F_{\lam} ,
\end{align}
For $0<\lam_1\ll\cdots\ll\lam_N$ and $\vec \te=(\te_1,\dots,\te_N)\in \bT^N$, by definition 
\begin{align}
    \calQ_{\vec \lam, \vec \te}(\te,r)=e^{k\te R}\begin{pmatrix}
        \mathop{\sum} \limits_{j=1}^N (-1)^{j-1}\frac{2(r/\lam_j)^k}{1+(r/\lam_j)^{2k}}\cos \te_j \\
        \mathop{\sum} \limits_{j=1}^N (-1)^{j-1}\frac{2(r/\lam_j)^k}{1+(r/\lam_j)^{2k}}\sin \te_j \\
        1-2 \mathop{\sum} \limits_{j=1}^N (-1)^{j-1}\frac{(r/\lam_j)^{2k}}{1+(r/\lam_j)^{2k}}
    \end{pmatrix},
\end{align}
In the region $\{\sqrt{\lam_{j-1}\lam_j}\leq r\leq \sqrt{\lam_j\lam_{j+1}}\}$, we have
\begin{align}
    \calQ&=(-1)^{j-1}e^{\te_j R}Q_{\lam_j}+\sum_{\ell<j}(-1)^{\ell-1}\left(e^{\te_{\ell}R}Q_{\lam_{\ell}}+e_3\right)+\sum_{\ell>j}(-1)^{\ell-1}\left(e^{\te_{\ell}R}Q_{\lam_{\ell}}-e_3\right),
\end{align}
which implies 
\begin{equation}\label{eq:Q-cal-expansion-local}
    \begin{aligned}
        \calQ-(-1)^{j-1}e^{\te_j R}Q_{\lam_j}&=2\sum_{\ell<j}(-1)^{\ell-1}\left(\frac{\lam_{\ell}}{r}\right)^ke^{k\te R}\begin{pmatrix}
            \cos \te_{\ell}\\ \sin \te_{\ell} \\ 0
        \end{pmatrix}+2\sum_{\ell>j}(-1)^{\ell-1}
 \left(\frac{r}{\lam_{\ell}}\right)^ke^{k\te R}\begin{pmatrix}
            \cos \te_{\ell}\\ \sin \te_{\ell} \\ 0
        \end{pmatrix}\\
        &\quad +O\left(\sum_{\ell\neq j}F_{\ell}^2\right),
    \end{aligned}
\end{equation}
where the remainder keeps the same order when we apply $(r\p_r)^a$ for $1\leq a\leq 4$.

We also have
\begin{align}
    |\calQ|^2-1=2(-1)^{j-1}e^{\te_j R}Q_{\lam_j}\cdot \left(\calQ-(-1)^{j-1}e^{\te_j R}Q_{\lam_j}\right)+\left|\calQ-(-1)^{j-1}e^{\te_j R}Q_{\lam_j}\right|^2,
\end{align}
which implies for $(r\p_r)^a$ with $0\leq a\leq 4$, 
\begin{align}\label{eq:Q-cal-norm}
    \left|(r\p_r)^a\left(|\calQ|^2-1\right)\right|\lesssim\sum_{1\leq \ell<m\leq N}F_{\ell}F_{m}.
\end{align}

\subsection{Static decomposition}
The notation we use in this section is independent of other sections.    
\begin{lem}\label{lem:static-mod}
    There exists $\veps_0,C_0>0$ such that for any $k$-equivariant $u\in \dot H^1(\R^2;\bS^2)$ satisfying
    \begin{align}
        \inf_{(\vec \mu,\vec \zeta)\in \R_+^N\times\R^N}\left(\left\|u-\calU_{\vec \mu,\vec \zeta}\right\|_{\dot \calH^1}+\sum_{j=1}^{N-1}\frac{\mu_j}{\mu_{j+1}}+\sum_{j=1}^N\left|\zeta_j+\frac{1+(-1)^{j}}{2}\frac{\pi}{2}\right|\right)<\veps_0,
    \end{align}
    there exist unique $g\in \dot \calH^1$,  $\vec \lam,\vec \te$ satisfying
    \begin{align}
        \|g\|_{\dot \calH^1}+\sum_{j=1}^{N-1}\frac{\lam_j}{\lam_{j+1}}+\sum_{j=1}^N\left|\te_j+\frac{1+(-1)^{j}}{2}\frac{\pi}{2}\right|<C_0 \veps_0
    \end{align}
    such that $u$ admits a decomposition 
    \begin{align}
        u=\calU_{\vec \lam,\vec \te}+g
    \end{align}
    with the orthogonality conditions
    \begin{align}\label{eq:g-ortho}
    \left\langle g, \chi_{L,j}e^{\te_j R}\Lam Q_{\lam_j}\right\rangle=\left\langle g,\chi_{L,j}e^{\te_j R}R Q_{\lam_j}\right\rangle=0,\ j=1,\dots,N.
    \end{align}
\end{lem}

\begin{proof}
    The proof is a standard application of the implicit function theorem, see Lemma 2.17 in \cite{JL-NLW-SRC} for example.  
\end{proof}

\subsection{Darboux coordinates}\label{subsec:Darboux}

In this section we introduce Darboux coordinates for $k$-equivariant mappings near a given $k$-equivariant mapping. The notation we use in this section is independent of other sections.

Given a $k$-equivariant mapping $\calV(\te,r)=e^{k\te R}V(r)\in \dot H^1$. We first introduce an orthogonal system $(h_1^V(r),h_2^V(r),V(r))$ determined through Coulomb gauge:
\begin{equation}\label{eq:frame-def}
    \begin{aligned}
    \begin{cases}
        \p_r h_1^V=-(h_1^V\cdot \p_r V)V,\\
        h_1^V(0)=(1,0,0)^{\top},\\
        h_2^V=V\wedge h_1^V.
    \end{cases}
    \end{aligned}
\end{equation}
We define $f_j^V(\te,r)=e^{k\te R}h_j^V(r)$ for $j=1,2$ such that $(f_1^V(r),f_2^V(r),\calV(r))$ forms an orthogonal system.
Then for $k$-equivariant mapping $v\in \dot H^1$ satisfying $\|v-\calV\|_{\dot  \cH^1}<1$, there exists a unique complex-valued function $z^V(r)=z_1^V(r)+iz_2^V(r)\in \calH^1$  with $z_1^V,z_2^V\in \R$ such that
\begin{align}\label{eq:v-z}
    v(\te,r)&=\sqrt{1-\frac{|z^V|^2}{4}}(z_1^Vf_1^V+z_2^Vf_2^V)+\left(1-\frac{|z^V|^2}{2}\right)\calV\\
    &=e^{k\te R}\left(\sqrt{1-\frac{|z^V|^2}{4}}(z_1^V h_1^V+z_2^V h_2^V)+\left(1-\frac{|z^V|^2}{2}\right)V\right).
\end{align}
Conversely,  for a complex-valued function $z^V(r)=z_1^V(r)+iz_2^V(r)\in \calH^1$  with $z_1^V,z_2^V\in \R$ and $\|z^V\|_{\calH^1}<1$, we define
\begin{align}
    \Phi_V (z^V)(\te,r)=\sqrt{1-\frac{|z^V|^2}{4}}(z_1^Vf_1^V+z_2^Vf_2^V)+\left(1-\frac{|z^V|^2}{2}\right)\calV.
\end{align}For a complex-valued function $\xi(r)=\xi_1(r)+i\xi_2(r)\in L^2$ with $\xi_1,\xi_2\in \R$, we define
\begin{align}
    \xi_{z^V,V}(\te,r)&= \vD \Phi_{V}(z^V)(\xi)   \\
    &=\sqrt{1-\frac{|z^V|^2}{4}}(\xi_1 f_1^V+\xi_2 f_2^V)-\frac{\Re\left(\overline {z^V}\xi\right)}{4\sqrt{1-\frac{|z^V|^2}{4}}}(z_1^Vf_1^V+z_2^Vf_2^V)-\Re \left(\overline {z^V}\xi\right)\calV.
\end{align}
We have the following inequality:
\begin{align}
    \label{eq:Darboux}
    \int_{\R^2}v\cdot\left(\xi_{z^V,V}\wedge\eta_{z^V,V}\right)\ud x=2\pi \Im\int_{0}^{\infty}\bar{\xi}\eta\, \udr.
\end{align}
Let
\begin{align}\label{eq:Gz-def}
    G_{z^V,V}\xi:=(\vD\Phi_V(z^V) )^{*}\vD \Phi_{V}(z^V) \xi=\left(1-\frac{|z^V|^2}{4}\right)\xi+\frac{8-|z^V|^2}{4(4-|z^V|^2)}\Re(\overline{z^V} \xi)z^V,
\end{align}
then we have
\begin{align}
    \La G_{z^V,V} \xi,\eta \Ra=\La  \xi, G_{z^V,V}\eta \Ra =\La \xi_{z^V,V},\eta_{z^V,V}  \Ra.
\end{align}
We also have
\begin{align}
    (\vD G_{z^V,V}(h))\xi&=-\frac{1}{2}\Re(\overline{z^V}h)\xi+\frac{2\Re(\overline{z^V} h)}{(4-|z^V|^2)^2}\Re(\overline{z^V}\xi)z^V\\
    &\quad +\frac{8-|z^V|^2}{4(4-|z^V|^2)}\Re(\bar h\xi )z^V+\frac{8-|z^V|^2}{4(4-|z^V|^2)}\Re(\overline{z^V}\xi )h.
\end{align}
An immediate result is that for $|z^V|$ sufficiently small, we have
\begin{align}\label{eq:DGz-bounded}
    |(\vD G_{z^V,V}(h))\xi|\leq 2 |z^V||h||\xi|.
\end{align}
We denote by
\begin{align}\label{eq:Kz-def}
    K_{z^V,V}=G_{z^V,V}-\rm{Id}.
\end{align}
Then for $|z^V|$ sufficiently small, there holds
\begin{align}
    &|K_{z^V,V} \xi|\leq 2|z^V|^2|\xi|,\ |(\p_rK_{z^V,V})\xi|\leq 2|z^V||\p_r z^V||\xi|.\label{eq:Kz-bounded}\\
    &|(\p_r^2K_{z^V,V})\xi|\leq 2\left[|\p_rz^V|^2+|z^V||\p_r^2 z^V|\right]|\xi|,\label{eq:Kz-rr-bounded}\\
    &|(\p_r^3K_{z^V,V})\xi|\leq 6\left[|z^V||\p_rz^V|^3+|\p_r z^V||\p_r^2 z^V|+|z^V||\p_r^3 z^V|\right]|\xi|.\label{eq:Kz-rrr-bounded}
\end{align}
Here for an $\R$-linear operator $K$,  we identify it with four real-valued functions $K_j$, $1\leq j\leq 4$ such that for $z\in \C$, we have $Kz=K_1\Re z+K_2\Im z+i\left(K_3\Re z+K_4 \Im z\right)$. The $\R$-linear operator $\p_r K$ is defined naturally via $(\p_rK)z=(\p_rK_1)\Re z+(\p_rK_2)\Im z+i\left((\p_rK_3)\Re z+(\p_rK_4) \Im z\right)$.

\subsection{Variational estimates}

Define
\begin{align}
    H_{\lam}=-\p_r^2-\frac{1}{r}\p_r+k^2\frac{\cos 2\phi(\frac{r}{\lam})}{r^2}, A_{\lam}=-\p_r+k\frac{\cos \phi(\frac{r}{\lam})}{r}, \ A_{\lam}^*=\p_r+\frac{1}{r}+k\frac{\cos \phi(\frac{r}{\lam})}{r}.
\end{align}
Then we have
\begin{align}
    H_{\lam}=A^*_{\lam}A_{\lam}.
\end{align}
Define 
\begin{align}
    \ti H_{\lam}=A_{\lam}A_{\lam}^*.
\end{align}

\begin{lem}\label{lem:Alam-coercive}
    For $L>1$ sufficiently large,   there exists a constant $\ti C(L)>0$ such that for any complex-valued function $z$ satisfying that the term in the left-hand side is finite, there holds
    \begin{align}
        \|z\|_{\rm{Mor}(\lam)}^2\lesssim \ti C(L)\|A_{\lam}z\|_{\rm{Mor}_2(\lam)}^2+\lam^{-9}\left|\La z,\chi_{L}(r/\lam)\sin\phi(r/\lam) \Ra\right|^2+\lam^{-9}\left|\La z,i\chi_{L}(r/\lam)\sin\phi(r/\lam) \Ra\right|^2.
    \end{align}
\end{lem}
\begin{proof}
    By scaling, it suffices to consider the case $\lam=1$. The proof follows the same argument as in  Lemma B.1 in \cite{HK-26-WM}, and we omit the details.   
\end{proof}

\section{Bootstrap framework}\label{sec:bootstrap}

In this section we state the bootstrap proposition which is the core of our analysis. Strictly speaking, the control for $\lam_j(t)$ with $1\leq j\leq N-1$ is a result of topological shooting argument instead of a direct bootstrap improvement. Nevertheless, we present it within the bootstrap framework for a clearer organization of the paper.

\begin{prop}\label{prop:bootstrap}
    Take $0<\veps_1<\veps_2<\dots<\veps_{N}<\veps_0<\min\{\frac{1}{2},\frac{1}{k-2} \}$. There exists $L>0$ such that,
    there exists $T_0<0$ with the following properties: Let $T<T_1<T_0$ and let $\lam_j^0$ such that for $j=1,\dots,N-1$,
    \begin{align}
        \left|\frac{\lam_j^0}{\lam_{j}^{\infty}(T)}-1\right|<\frac{1}{2}|T|^{-\veps_j}
    \end{align}
    Consider the solution $u(t)$ of ~\eqref{SM} with the initial data $\calU_{(\lam_1^0,\dots,\lam_{N-1}^0,1,\vec \te^{\infty})}$ at time $T$. Suppose that $u(t)$ exists on the time interval $[T,T_1]$ and can be decomposed according to Lemma ~\ref{lem:static-mod}
    with the following properties hold for $t\in [T,T_1]$:
    \begin{align}
        \left|\frac{\lam_j(t)}{\lam_{j}^{\infty}(t)}-1\right|&\leq|t|^{-\veps_j},\ j=1,\dots, N,\label{eq:lamN-bootstrap}\\
        \left|\te_j(t)-\te_j^{\infty}\right|&\leq |t|^{-\veps_j} ,\ j=1,\dots, N,\label{eq:te-bootstrap}\\
        \|g\|_{\dot \calH^1}&\leq |t|^{-\frac{1}{2}} \label{eq:g-H1-bootstrap}\\
        \|g\|_{\dot \calH^3}&\leq |t|^{-\frac{1}{2}-\veps_0}.\label{eq:g-H3-bootstrap}
    \end{align}
   Then for $t\in[T,T_1]$ there holds
    \begin{align}
        \left|\frac{\lam_N(t)}{\lam_{N}^{\infty}(t)}-1\right|&\leq \frac{1}{2}|t|^{-\veps_N},\label{eq:lamN-improved}\\
        \left|\te_j(t)-\te_j^{\infty}\right|&\leq \frac{1}{2}|t|^{-\veps_j} ,\ j=1,\dots, N,\label{eq:te-improved}\\
        \|g\|_{\dot \calH^1}&\leq \frac{1}{2}|t|^{-\frac{1}{2}},\label{eq:g-H1-improved}\\
        \|g\|_{\dot \calH^3}&\leq \frac{1}{2}|t|^{-\frac{1}{2}-\veps_0}.\label{eq:g-H3-improved}
    \end{align}
    Moreover, we can choose $\lam_j^0$ such that the related solution exists on $[T,T_0]$ with ~\eqref{eq:lamN-bootstrap}, ~\eqref{eq:te-bootstrap}, ~\eqref{eq:g-H1-bootstrap}, ~\eqref{eq:g-H3-bootstrap} and
    \begin{align}\label{eq:lam-improved}
         \left|\frac{\lam_j(t)}{\lam_{j}^{\infty}(t)}-1\right|&\leq\frac{1}{2}|t|^{-\veps_j},\ j=1,\dots, N-1
    \end{align}
    hold for $t\in [T,T_0]$.
   
\end{prop}

The proof of Proposition~\ref{prop:bootstrap} will be presented in Section~\ref{sec:energy-estimate}, ~\ref{sec:mod}, ~\ref{sec:high-Sobolev} ~\ref{sec:completing-bootstrap}.

\section{Control of critical Sobolev norm}\label{sec:energy-estimate}

In this section we study the energy expansion and derive some basic estimates to improve the estimates of $\|g\|_{\dot\calH^1}$.
Sometimes we drop the subscript $\vec p$ of $\calU_{\vec p}$ for convenience.

Our starting point is the following identity:
\begin{align}
    E(\calU+g)&=E(\calU)+2\int\na\calU\cdot\na g \,\ud x+\int|\na g|^2\,\ud x\\
    &=E(\calU)-2\La\ta(\calU),g \Ra+q_{\calU}(g),
\end{align}
where
\begin{align}
    \ta(\calU)=\De\calU+|\na\calU|^2\calU
\end{align}
and
\begin{align}
    q_{\calU}(g)=\int \left( |\nabla g|^2 - |\nabla \mathcal{U}|^2 |g|^2 \right)\,\ud x.
\end{align}

\begin{prop}\label{prop:energy-control}
    There exists $C,c>0$ depending only on $k,N$ such that for $k$-equivariant function $u\in \dot H^1(\R^2;\bS^2)$ of the form $u=\calU_{\vec p}+g$, with 
    \begin{align}
    \sum_{j=1}^{N-1}\frac{\lam_j}{\lam_{j+1}}+\sum_{j=1}^{N-1}\left|\te_j+\frac{1+(-1)^j}{2}\frac{\pi}{2}\right|+\|g\|_{\dot \calH^1}\leq c,
\end{align}
    and $g$ satisfying
\begin{align}
    \La g,\chi_{L,j}e^{\te_j R}\Lam Q_{\lam_j}\Ra= \La g,\chi_{L,j}e^{\te_j R}R Q_{\lam_j}\Ra=0,\ j=1,2,\dots,N,
\end{align}
we have
\begin{align}\label{eq:energy-control}
         C\|g\|_{\dot \calH^1}^2\leq &  \left|E(u)-NE(Q)\right|\\
         &+\sum_{j=1}^{N-1}\left(\frac{\lam_j}{\lam_{j+1}}\right)^k\left(|\cos(\te_j-\te_{j+1})|+\left(\frac{\lam_j}{\lam_{j+1}}\right)^k\left|\log \frac{\lam_j}{\lam_{j+1}}\right|\right).
     \end{align}
\end{prop}

\begin{cor}\label{cor:g-improved}
    Under the bootstrap assumption, for $t\in [T,T_1]$ we have
    \begin{align}\label{eq:g-H1-improved-cor}
        \|g(t)\|_{\dot\calH^1}\lesssim |t|^{-\frac{1}{2}-\frac{1}{k-2}-\frac{\veps_1}{2}}.
    \end{align}
\end{cor}
Proposition ~\ref{prop:energy-control} is an immediate result of the remaining lemmas of this section.
Before proving Proposition ~\ref{prop:energy-control}, we show the proof of Corollary~\ref{cor:g-improved}.
\begin{proof}[Proof of Corollary ~\ref{cor:g-improved}]
    Since $E(u)$ is conserved along the flow, inserting the bootstrap assumption
    \begin{align}
        C\|g\|_{\dot \calH^1}^2\leq &  \left|E(u(T))-NE(Q)\right|\\
         &+\sum_{j=1}^{N-1}\left(\frac{\lam_j}{\lam_{j+1}}\right)^k\left(|\cos(\te_j-\te_{j+1})|+\left(\frac{\lam_j}{\lam_{j+1}}\right)^k\left|\log \frac{\lam_j}{\lam_{j+1}}\right|\right)\\
         =& \left|E(\calU_{\vec \lam(T),\vec \te^{\infty}})-NE(Q)\right|\\
         &+\sum_{j=1}^{N-1}\left(\frac{\lam_j}{\lam_{j+1}}\right)^k\left(|\cos(\te_j-\te_{j+1})|+\left(\frac{\lam_j}{\lam_{j+1}}\right)^k\left|\log \frac{\lam_j}{\lam_{j+1}}\right|\right)\\
         \lesssim & |t|^{-1-\frac{2}{k-2}-\veps_1}   
    \end{align}
\end{proof}

\begin{lem}\label{lem:zeroth-energy}
    Under the assumption of Proposition~\ref{prop:energy-control}, we have
\begin{align}\label{eq:zeroth-energy}
    \left|E(\calU)-NE(Q)-32 k \pi \sum_{j=1}^{N-1} \cos(\te_j-\te_{j+1}) \left(\frac{\lam_j}{\lam_{j+1}}\right)^k \right|\lesssim \sum_{j=1}^{N-1} \left(\frac{\lam_j}{\lam_{j+1}}\right)^{2k}\left|\log \frac{\lam_j}{\lam_{j+1}}\right|.
\end{align}
\end{lem}
\begin{proof}
    The proof is similar to that of Lemma 2.5 in \cite{J-17-NLS}. To calculate the asymptotic behavior of $E(u)$, we will divide the integral domain  into several parts: $\{\sqrt{\lam_{j-1}\lam_j}\leq r\leq\sqrt{\lam_{j}\lam_{j+1}}\}$, $j=1,\dots,N$. In each interval the dominant component of $|\na \calU|^2$ can be extracted explicitly due to the coherent structure of $\calU$. Since $\calU=\frac{\calQ}{|\calQ|}$, we have
    \begin{align}
        \p_r \calU=\frac{\p_r \calQ}{|\calQ|}-\frac{\calQ\cdot \p_r \calQ}{|\calQ|^3}\calQ.
    \end{align}
    Therefore
    \begin{align}
        |\p_r \calU|^2=\frac{|\p_r \calQ|^2}{|\calQ|^2}+\frac{|\calQ\cdot \p_r \calQ|^2}{|\calQ|^4}-2\frac{|\calQ\cdot \p_r \calQ|^2}{|\calQ|^4}=\frac{|\p_r \calQ|^2}{|\calQ|^2}-\frac{|\calQ\cdot \p_r \calQ|^2}{|\calQ|^4}.
    \end{align}
    Then we have
    \begin{align}
        E(\calU)&=2\pi\int_0^{\infty}\left(|\p_r\calU|^2+\frac{k^2}{r^2}|R\calU|^2\right)\udr\\
        &=2\pi\int_0^{\infty}\left(\frac{|\p_r \calQ|^2}{|\calQ|^2}-\frac{|\calQ\cdot \p_r \calQ|^2}{|\calQ|^4}+\frac{k^2}{r^2}\frac{|R\calQ|^2}{|\calQ|^2} \right)\udr.
    \end{align}
In the region $\{\sqrt{\lam_{j-1}\lam_j}\leq r\leq\sqrt{\lam_{j}\lam_{j+1}}\}$, using the expansion ~\eqref{eq:Q-rescaled-expansion}, ~\eqref{eq:Q-cal-expansion-local} and ~\eqref{eq:Q-cal-norm},  we have
    \begin{align}
        &\frac{|\p_r \calQ|^2}{|\calQ|^2}-\frac{|\calQ\cdot \p_r \calQ|^2}{|\calQ|^4}+\frac{k^2}{r^2}\frac{|R\calQ|^2}{|\calQ|^2} \\
        =&\sum_{\ell=1}^N \frac{8k^2}{r^2}\frac{(r/\lam_{\ell})^{2k}}{(1+ (r/\lam_{\ell})^{2k})^2}\\
        &+\frac{16k^2}{r^2}\sum_{\ell<j} (-1)^{j+\ell}\left(\frac{\lam_{\ell}}{r}\right)^k\cos(\te_{\ell}-\te_j)\frac{(r/\lam_j)^{3k}((r/\lam_j)^{2k}-3)}{(1+(r/\lam_j)^{2k})^3}\\
        &+\frac{16k^2}{r^2}\sum_{\ell>j} (-1)^{j+\ell}\left(\frac{r}{\lam_{\ell}}\right)^k\cos(\te_{\ell}-\te_j)\frac{(\lam_j/r)^{3k}((\lam_j/r)^{2k}-3)}{(1+(\lam_j/r)^{2k})^3}\\
        &+O\left(\frac{1}{r^2}\left(\sum_{\ell<m,\ell,m\neq j}F_{\ell}F_m+\left(\sum_{\ell<m}F_{\ell}F_m\right)^2\right)\right)
    \end{align}
For the leading terms, we have
    \begin{align}
        2\pi\sum_{j=1}^N\int_{\sqrt{\lam_{j-1}\lam_j}}^{\sqrt{\lam_j\lam_{j+1}}}\sum_{\ell=1}^N \frac{8k^2}{r^2}\frac{(r/\lam_{\ell})^{2k}}{(1+ (r/\lam_{\ell})^{2k})^2}\udr=NE(Q).
    \end{align}
 For the interaction terms, we have
   \begin{align}
       &2\pi\int_{\sqrt{\lam_{j-1}\lam_j}}^{\sqrt{\lam_j\lam_{j+1}}}\frac{16k^2}{r^2}\sum_{\ell<j} (-1)^{j+\ell}\left(\frac{\lam_{\ell}}{r}\right)^k\cos(\te_{\ell}-\te_j)\frac{(r/\lam_j)^{3k}((r/\lam_j)^{2k}-3)}{(1+(r/\lam_j)^{2k})^3}\udr\\
       =&-16k \pi \sum_{\ell<j} (-1)^{j+\ell}\cos(\te_{\ell}-\te_j)\left(\frac{\lam_{\ell}}{\lam_j}\right)^k+O\left(\left(\frac{\lam_{j-1}}{\lam_j}\right)^k\left[ \left(\frac{\lam_{j-1}}{\lam_j}\right)^k+\left(\frac{\lam_{j}}{\lam_{j+1}}\right)^k \right]\right)\\
       =&\ 16k \pi \cos(\te_{j-1}-\te_j)\left(\frac{\lam_{j-1}}{\lam_j}\right)^k+O\left(\sum_{\ell=1}^{N-1}\left(\frac{\lam_{\ell}}{\lam_{\ell+1}}\right)^{2k}\right),
   \end{align}
    \begin{align}
        &2\pi\int_{\sqrt{\lam_{j-1}\lam_j}}^{\sqrt{\lam_j\lam_{j+1}}}\frac{16k^2}{r^2}\sum_{\ell>j} (-1)^{j+\ell}\left(\frac{r}{\lam_{\ell}}\right)^k\cos(\te_{\ell}-\te_j)\frac{(\lam_j/r)^{3k}((\lam_j/r)^{2k}-3)}{(1+(\lam_j/r)^{2k})^3}\udr\\
        =&-16k \pi \sum_{\ell>j}(-1)^{j+\ell}\cos(\te_{\ell}-\te_j)\left(\frac{\lam_{j}}{\lam_{\ell}}\right)^k+O\left(\left(\frac{\lam_{j}}{\lam_{j+1}}\right)^k\left[ \left(\frac{\lam_{j-1}}{\lam_j}\right)^k+\left(\frac{\lam_{j}}{\lam_{j+1}}\right)^k \right]\right)\\
        =&\ 16k \pi \cos(\te_{j+1}-\te_j)\left(\frac{\lam_{j}}{\lam_{j+1}}\right)^k+O\left(\sum_{\ell=1}^{N-1}\left(\frac{\lam_{\ell}}{\lam_{\ell+1}}\right)^{2k}\right)
    \end{align}
 For the remainder terms, we have
    \begin{align}
        \sum_{j=1}^N \int_{\sqrt{\lam_{j-1}\lam_j}}^{\sqrt{\lam_j\lam_{j+1}}}\frac{1}{r^2}\left(\sum_{\ell<m,\ell,m\neq j}F_{\ell}F_m+\left(\sum_{\ell<m}F_{\ell}F_m\right)^2\right)\udr\lesssim \sum_{j=1}^{N-1} \left(\frac{\lam_j}{\lam_{j+1}}\right)^{2k}\left|\log \frac{\lam_j}{\lam_{j+1}}\right|.
    \end{align}
Combining these estimates yields ~\eqref{eq:zeroth-energy}.

\end{proof}

\begin{lem}
    Under the assumption of Proposition~\ref{prop:energy-control}, we have
    \begin{align}
        \left|
\La\ta(\calU),g \Ra
\right|
\lesssim  \left(\sum_{j=1}^{N-1} \left(\frac{\lam_j}{\lam_{j+1}}\right)^{2k}\left|\log \frac{\lam_j}{\lam_{j+1}}\right|\right)^{\frac{1}{2}}
\,\|g\|_{\dot{\calH}^1}.
    \end{align}
\end{lem}
\begin{proof}
    Cauchy inequality and Hardy inequality yields
    \begin{align}
         \left|
\La\ta(\calU),g \Ra
\right|\leq \|r\ta(\calU)\|_{L^2}\left\|\frac{g}{r}\right\|_{L^2}
\lesssim \|r\ta(\calU)\|_{L^2}\|g\|_{\dot{\calH}^1}.
    \end{align}
    For $\ta(\calU)$, we have
    \begin{align}
        |\ta(\calU)|=|\calU\wedge\De\calU|,
    \end{align}
    \begin{align}
        \calU\wedge\De\calU=\frac{1}{|\calQ|^2}\calQ\wedge\De\calQ+2\frac{1}{|\calQ|}\p_r(|\calQ|^{-1})\calQ\wedge\p_r\calQ
    \end{align}

    In the region $\{\sqrt{\lam_{j-1}\lam_j}\leq r\leq\sqrt{\lam_{j}\lam_{j+1}}\}$, we calculate the expansion of $\calU\wedge\De\calU$ using ~\eqref{eq:Q-rescaled-expansion}, ~\eqref{eq:Q-cal-expansion-local} and  \eqref{eq:Q-cal-norm} to obtain
    \begin{equation}\label{eq:U-wedge-De-U-expansion-local}
        \begin{aligned}
             &\calU\wedge\De\calU\\
        =&\frac{8k}{r^2}\frac{(r/\lam_j)^k}{1+(r/\lam_j)^{2k}}\sum_{\ell<j}(-1)^{j+\ell}\left(\frac{\lam_{\ell}}{r}\right)^k\left[k\cos(\te_{\ell}-\te_j)e^{\te_j R}RQ_{\lam_j}-\sin(\te_{\ell}-\te_j)e^{\te_j R}\Lam Q_{\lam_j}\right]\\
        &-\frac{8k}{r^2}\frac{(\lam_j/r)^k}{1+(\lam_j/r)^{2k}}\sum_{\ell>j}(-1)^{j+\ell}\left(\frac{r}{\lam_{\ell}}\right)^k\left[k\cos(\te_{\ell}-\te_j)e^{\te_j R}RQ_{\lam_j}+\sin(\te_{\ell}-\te_j)e^{\te_j R}\Lam Q_{\lam_j}\right]\\
        &+O\left(\frac{1}{r^2}\sum_{\ell\neq j}F_{\ell}^2\right),
        \end{aligned}
    \end{equation}
    which yields
    \begin{align}
         \|r\ta(\calU)\|_{L^2}^2\lesssim \sum_{j=1}^{N-1} \left(\frac{\lam_j}{\lam_{j+1}}\right)^{2k}\left|\log \frac{\lam_j}{\lam_{j+1}}\right|.
    \end{align}
\end{proof}

\begin{lem}
    Under the assumption of Proposition~\ref{prop:energy-control}, we have
    \begin{align}\label{eq:qU-coercive}
        \|g\|_{\dot \calH^1}^{2} \lesssim q_{\calU}(g).
        \end{align}
\end{lem}
\begin{proof}
    Since $g\cdot\calU=-\frac{1}{2}|g|^2$, under the bootstrap assumption we only have to prove ~\eqref{eq:qU-coercive} with $g$ replaced by $g+\frac{1}{2}|g|^2\calU$, which is tangent to $\calU$. The proof is the same as that of Lemma 2.18 in \cite{JL-WM-SRC}. 
\end{proof}

\section{Control of parameters}\label{sec:mod}

In this section we derive the modulation estimates for solutions to ~\eqref{SM} under the bootstrap assumption.  We define
\begin{align}\label{eq:def-refined-parameter}
    \hat {\lam}_j=\lam_j+\frac{(-1)^j}{\|\Lam Q\|_{L^2}^2 \lam_j}\La g,e^{\te_j R}\Lam Q_{\lam_j}\Ra,\ \hat \te_j=\te_j-\frac{(-1)^j}{\|R Q\|_{L^2}^2 \lam_j^2}\La g,e^{\te_j R}R Q_{\lam_j}\Ra.
\end{align}

\begin{prop}
    \label{prop:mod}
    We have the following estimates for $\lam_j',\te_j'$:
    \begin{align}
        \lam_j|\lam_j'|&\lesssim \left(o_{L\to\infty}(1)+o_{T_0\to-\infty}(1)\right)\lam_j^2\|g\|_{\dot \calH^3}+\left(\frac{\lam_{j-1}}{\lam_{j}}\right)^k+\left(\frac{\lam_{j}}{\lam_{j+1}}\right)^k, \label{eq:lam'-rough-mod}\\
        \lam_j^2|\te_j'|&\lesssim \left(o_{L\to\infty}(1)+o_{T_0\to-\infty}(1)\right)\lam_j^2\|g\|_{\dot \calH^3}+\left(\frac{\lam_{j-1}}{\lam_{j}}\right)^k+\left(\frac{\lam_{j}}{\lam_{j+1}}\right)^k.\label{eq:te'-rough-mod}
    \end{align}
 For the refined parameters, we have
    \begin{align}\label{eq:lam-hatlam-te-hatte-close}
        \left|\frac{\hat\lam_j(t)}{\lam_j(t)}-1\right|+\left|\hat\te_j(t)- \te_j(t)\right|\lesssim \|g\|_{\dot \calH^1}.
    \end{align}
    The dynamical behavior of the refined parameters can be estimated explicitly:
    \begin{align}\label{eq:lam-hat-modulation}
        \lam_j \hat \lam_j' - \ka_k (-1)^j \left[ \left( \frac{\lam_{j-1}}{\lam_j} \right)^k \sin(\te_{j-1} - \te_j) -\left( \frac{\lam_j}{\lam_{j+1}} \right)^k \sin(\te_j - \te_{j+1})  \right]  =o\left(|t|^{-\veps_j}\right)\frac{\lam_j^2}{|t|},
    \end{align}

    \begin{align}\label{eq:te-hat-modulation}
         \lam_j^2 \hat \te_j' - k\ka_k (-1)^j \left[ \left( \frac{\lam_{j-1}}{\lam_j} \right)^k \cos(\te_{j-1} - \te_j) -\left( \frac{\lam_j}{\lam_{j+1}} \right)^k \cos(\te_j - \te_{j+1}) \right] =o\left(|t|^{-\veps_j}\right)\frac{\lam_j^2}{|t|}.
    \end{align}

\end{prop}
\begin{rem}
    ~\eqref{eq:lam-hatlam-te-hatte-close}, ~\eqref{eq:lam-hat-modulation}, ~\eqref{eq:te-hat-modulation} are only used in Section~\ref{subsec:stable} and Section~\ref{subsec:unstable}.
\end{rem}

\begin{proof}
    The proof is very similar to Lemma 4.11 and Lemma 4.16 in \cite{JL-WM-SRC} and  Lemma 5.12 and Lemma 5.17 in \cite{JL-NLW-SRC}. 
    For the proofs of  ~\eqref{eq:lam'-rough-mod} and ~\eqref{eq:te'-rough-mod}, differentiating the orthogonal conditions ~\eqref{eq:g-ortho} yields a linear system of $\lam_j\lam_j'$ and $\lam_j^2\te_j'$. A direct computation implies the elements of the coefficient matrix is diagonally dominant and hence invertible, which implies the desired rough estimate. We conduct the calculation as follows.
    Differentiating the orthogonal conditions, we obtain
    \begin{equation}\label{eq:rough-modulation-lambda-raw}
        \begin{aligned}
            &\lam_j\lam_j' \frac{1}{\lam_j}\left[ \La \p_{\lam_j}\calU , \chi_{L,j}e^{\te_j R}\Lam Q_{\lam_j}\Ra-\La g , \p_{\lam_j}\left( \chi_{L,j}e^{\te_j R}\Lam Q_{\lam_j}\right)\Ra \right]\\
        &+\lam_j^2\te_j'\frac{1}{\lam_j^2}\left[ \La \p_{\te_j}\calU , \chi_{L,j}e^{\te_j R}\Lam Q_{\lam_j}\Ra-\La g , \p_{\te_j}\left( \chi_{L,j}e^{\te_j R}\Lam Q_{\lam_j}\right)\Ra \right]\\
        &+\sum_{\ell\neq j}\lam_\ell\lam_\ell' \frac{1}{\lam_\ell}\La \p_{\lam_\ell}\calU , \chi_{L,j}e^{\te_j R}\Lam Q_{\lam_j}\Ra+\sum_{\ell\neq j} \lam_\ell^2\te_\ell'\frac{1}{\lam_\ell^2} \La \p_{\te_\ell}\calU , \chi_{L,j}e^{\te_j R}\Lam Q_{\lam_j}\Ra \\
        =& \La \calU\wedge\De \calU,  \chi_{L,j}e^{\te_j R}\Lam Q_{\lam_j}\Ra+ \La \calU\wedge\De g,  \chi_{L,j}e^{\te_j R}\Lam Q_{\lam_j}\Ra\\
        &+ \La g\wedge\De \calU,  \chi_{L,j}e^{\te_j R}\Lam Q_{\lam_j}\Ra+ \La g\wedge\De g,  \chi_{L,j}e^{\te_j R}\Lam Q_{\lam_j}\Ra.
        \end{aligned}
    \end{equation}

    \begin{equation}\label{eq:rough-modulation-theta-raw}
        \begin{aligned}
            &\lam_j\lam_j' \frac{1}{\lam_j}\left[ \La \p_{\lam_j}\calU , \chi_{L,j}e^{\te_j R}R Q_{\lam_j}\Ra-\La g , \p_{\lam_j}\left( \chi_{L,j}e^{\te_j R}R Q_{\lam_j}\right)\Ra \right]\\
        &+\lam_j^2\te_j'\frac{1}{\lam_j^2}\left[ \La \p_{\te_j}\calU , \chi_{L,j}e^{\te_j R}R Q_{\lam_j}\Ra-\La g , \p_{\te_j}\left( \chi_{L,j}e^{\te_j R}R Q_{\lam_j}\right)\Ra \right]\\
        &+\sum_{\ell\neq j}\lam_\ell\lam_\ell' \frac{1}{\lam_\ell}\La \p_{\lam_\ell}\calU , \chi_{L,j}e^{\te_j R}R Q_{\lam_j}\Ra+\sum_{\ell\neq j} \lam_\ell^2\te_\ell'\frac{1}{\lam_\ell^2} \La \p_{\te_\ell}\calU , \chi_{L,j}e^{\te_j R}R Q_{\lam_j}\Ra \\
        =& \La \calU\wedge\De \calU,  \chi_{L,j}e^{\te_j R}R Q_{\lam_j}\Ra+ \La \calU\wedge\De g,  \chi_{L,j}e^{\te_j R}R Q_{\lam_j}\Ra\\
        &+ \La g\wedge\De \calU,  \chi_{L,j}e^{\te_j R}R Q_{\lam_j}\Ra+ \La g\wedge\De g,  \chi_{L,j}e^{\te_j R}R Q_{\lam_j}\Ra.
        \end{aligned}
    \end{equation}
By definition, for $\ell=1,\dots, N$ we have 
    \begin{align}\label{eq:partialU-expression}
        \p_{\lam_{\ell}}\calU=\frac{\p_{\lam_{\ell}}\calQ-(\calU\cdot\p_{\lam_{\ell}}\calQ)\calU}{|\calQ|},\ \p_{\te_{\ell}}\calU=\frac{\p_{\te_{\ell}}\calQ-(\calU\cdot\p_{\te_{\ell}}\calQ)\calU}{|\calQ|},
    \end{align}
    and
    \begin{align}
        \p_{\lam_\ell}\calQ=(-1)^{\ell}\frac{1}{\lam_{\ell}}e^{\te_{\ell}R}\Lam Q_{\lam_{\ell}}, \p_{\te_\ell}\calQ=(-1)^{\ell-1}e^{\te_{\ell}R}R Q_{\lam_{\ell}},
    \end{align}

    First of all, we estimate the coefficients of $\lam_j\lam_j'$ and $\lam_j^2\te_j'$ in ~\eqref{eq:rough-modulation-lambda-raw} and ~\eqref{eq:rough-modulation-theta-raw}.

    We first calculate the coefficients of $\lam_j\lam_j'$ in ~\eqref{eq:rough-modulation-lambda-raw} and of $\lam_j^2\te_j'$ in ~\eqref{eq:rough-modulation-theta-raw}.
    Since $Q\cdot \Lam Q=Q\cdot RQ=0$, we have
    \begin{equation}\label{eq:coefficient-diagonal}
        \begin{aligned}
            \frac{1}{\lam_j}\La \p_{\lam_j}\calU , \chi_{L,j}e^{\te_j R}\Lam Q_{\lam_j}\Ra&=(-1)^j\|\Lam Q\|_{L^2}^2+o_{L\to \infty}(1)+o_{T_0\to-\infty}(1),\\
        \frac{1}{\lam_j^2}\La \p_{\te_j}\calU , \chi_{L,j}e^{\te_j R}R Q_{\lam_j}\Ra&=(-1)^{j-1}\|R Q\|_{L^2}^2+o_{L\to \infty}(1)+o_{T_0\to-\infty}(1),
        \end{aligned}
    \end{equation}
   We also have
    \begin{align}
        \left|\frac{1}{\lam_j}\La g , \p_{\lam_j}\left( \chi_{L,j}e^{\te_j R}\Lam Q_{\lam_j}\right)\Ra\right|+\left|\frac{1}{\lam_j^2}\La g , \p_{\te_j}\left( \chi_{L,j}e^{\te_j R}R Q_{\lam_j}\right)\Ra\right|\lesssim\|g\|_{\dot \calH^1}=o_{T_0\to-\infty}(1).
    \end{align}
Then we estimate the coefficients of $\lam_j^2\te_j'$ in ~\eqref{eq:rough-modulation-lambda-raw} and of $\lam_j\lam_j'$ in  ~\eqref{eq:rough-modulation-theta-raw}.  The components  containing $g$ can be bounded by $\|g\|_{\dot \calH^1}$. For remaining components, we have
    \begin{equation}\label{eq:coefficient-off-diagonal-1}
        \begin{aligned}
            &  \left|\frac{1}{\lam_j^2} \La \p_{\te_j}\calU , \chi_{L,j}e^{\te_j R}\Lam Q_{\lam_j}\Ra\right|        
        =\left|\frac{1}{\lam_j} \La \p_{\lam_j}\calU , \chi_{L,j}e^{\te_j R}RQ_{\lam_j}\Ra\right|\\
        =&\left|\frac{1}{\lam_j^2}\int\frac{\chi_{L,j}}{|\calQ|^3}\left(\calQ\cdot e^{\te_j R}R Q_{\lam_j}\right)\left(\calQ\cdot e^{\te_j R}\Lam Q_{\lam_j}\right)\ud x\right|\\
        \lesssim&\frac{1}{\lam_j^2}\|\chi_{L,j}^{\frac{1}{2}}\calQ\cdot e^{\te_j R}R Q_{\lam_j}\|_{L^2}\|\chi_{L,j}^{\frac{1}{2}}\calQ\cdot e^{\te_j R}\Lam Q_{\lam_j}\|_{L^2}\\
        \lesssim & o_{L\to \infty}(1)+o_{T_0\to-\infty}(1).
        \end{aligned}
    \end{equation}
   The term $\La \p_{\lam_j}\calU , \chi_{L,j}e^{\te_j R}R Q_{\lam_j}\Ra$ in the expression of coefficient of $\lam_j\lam_j'$ in ~\eqref{eq:rough-modulation-theta-raw} is actually  $- \frac{1}{\lam_j}\La \p_{\te_j}\calU , \chi_{L,j}e^{\te_j R}\Lam Q_{\lam_j}\Ra$.\\
     Last we estimate the coefficients of $\lam_\ell\lam_\ell'$ and $\lam_\ell^2\te_\ell'$ for $\ell\neq j$ in ~\eqref{eq:rough-modulation-lambda-raw} and ~\eqref{eq:rough-modulation-theta-raw}. We estimate directly using
    \begin{align}
    |\lam_{\ell}\p_{\lam_{\ell}}\calU|+|\p_{\te_{\ell}}\calU|\lesssim F_{\ell},\ |\Lam Q_{\lam_j}|+|RQ_{\lam_j}|\lesssim F_j,\ |\chi_{L,j}|\leq 1
    \end{align}
    to obtain
    \begin{equation}\label{eq:coefficient-off-diagonal-2}
        \begin{aligned}
             &\left|\frac{1}{\lam_\ell}\La \p_{\lam_\ell}\calU , \chi_{L,j}e^{\te_j R}\Lam Q_{\lam_j}\Ra\right|+\left|\frac{1}{\lam_\ell^2} \La \p_{\te_\ell}\calU , \chi_{L,j}e^{\te_j R}\Lam Q_{\lam_j}\Ra\right|\\
        &+\left|\frac{1}{\lam_\ell}\La \p_{\lam_\ell}\calU , \chi_{L,j}e^{\te_j R}R Q_{\lam_j}\Ra\right|+\left|\frac{1}{\lam_\ell^2} \La \p_{\te_\ell}\calU , \chi_{L,j}e^{\te_j R}RQ_{\lam_j}\Ra\right|
        \lesssim  \begin{cases}
            \left(\frac{\lam_{\ell}}{\lam_j}\right)^{k-2},\ \ell<j,\\
            \left(\frac{\lam_{j}}{\lam_{\ell}}\right)^{k},\ \ell>j,\\
        \end{cases}
        \end{aligned}
    \end{equation}

     Secondly, we estimate the forcing terms in ~\eqref{eq:rough-modulation-lambda-raw} and ~\eqref{eq:rough-modulation-theta-raw}. In ~\eqref{eq:rough-modulation-lambda-raw} the forcing terms are
     \begin{align}
        & \La \calU\wedge\De \calU,  \chi_{L,j}e^{\te_j R}\Lam Q_{\lam_j}\Ra+ \La \calU\wedge\De g,  \chi_{L,j}e^{\te_j R}\Lam Q_{\lam_j}\Ra\\
        &+ \La g\wedge\De \calU,  \chi_{L,j}e^{\te_j R}\Lam Q_{\lam_j}\Ra+ \La g\wedge\De g,  \chi_{L,j}e^{\te_j R}\Lam Q_{\lam_j}\Ra.
    \end{align}
Using ~\eqref{eq:F-pointwise} we have
    \begin{align}
        &\left|\La \calU\wedge\De \calU,  \chi_{L,j}e^{\te_j R}\Lam Q_{\lam_j}\Ra\right|\\
        \lesssim& \int_0^{\infty} r^{-2}\left(\sum_{1\leq m<\ell\leq N}F_m F_{\ell}\right)F_j \udr\lesssim\left(\frac{\lam_{j-1}}{\lam_j}\right)^k+ \left(\frac{\lam_{j}}{\lam_{j+1}}\right)^k.
    \end{align}
Note that since $Q_{\lam_j}$ satisfies $\De Q_{\lam_j}+|\na Q_{\lam_j}|^2Q_{\lam_j}=0$, we have
    \begin{equation}\label{eq:Q-cancellation-linear-g}
        \begin{aligned}
             &\La Q_{\lam_j}\wedge\De g, \Lam Q_{\lam_j}\Ra + \La g\wedge\De Q_{\lam_j},  \Lam Q_{\lam_j}\Ra\\
         =&\La Q_{\lam_j}\wedge \left(\De g+|\na Q_{\lam_j}|^2 g\right)  ,\Lam Q_{\lam_j} \Ra\\
         =&\La  \De g+|\na Q_{\lam_j}|^2 g, \Lam Q_{\lam_j}\wedge Q_{\lam_j}   \Ra\\
         =&\La  \De g+|\na Q_{\lam_j}|^2 g, -kRQ_{\lam_j}   \Ra\\
         =&\La g, -kR\left(\De Q_{\lam_j}+|\na Q_{\lam_j}|^2Q_{\lam_j}\right) \Ra=0.
        \end{aligned}
    \end{equation}
    Then we write
    \begin{equation}\label{eq:g-linear-decomposition}
        \begin{aligned}
            &\La \calU\wedge\De g,  \chi_{L,j}e^{\te_j R}\Lam Q_{\lam_j}\Ra+ \La g\wedge\De \calU,  \chi_{L,j}e^{\te_j R}\Lam Q_{\lam_j}\Ra\\
        =&\La \left(\calU-(-1)^{j-1}e^{\te_j R}Q_{\lam_j}\right)\wedge\De g,  \chi_{L,j}e^{\te_j R}\Lam Q_{\lam_j}\Ra+ \La g\wedge\De \left(\calU-(-1)^{j-1}e^{\te_j R}Q_{\lam_j}\right),  \chi_{L,j}e^{\te_j R}\Lam Q_{\lam_j}\Ra\\
        &+(-1)^{j-1}\La e^{\te_j R}Q_{\lam_j}\wedge\De g, \left(\chi_{L,j}-1\right)e^{\te_j R}\Lam Q_{\lam_j}\Ra + (-1)^{j-1}\La g\wedge\De(e^{\te_j R} Q_{\lam_j}),  \left(\chi_{L,j}-1\right)e^{\te_j R}\Lam Q_{\lam_j}\Ra.
        \end{aligned}
    \end{equation}
Combining ~\eqref{eq:Q-cal-expansion-local}, ~\eqref{eq:Q-cal-norm} and Hardy's inequality yields
    \begin{align}
        &\left|\La \left(\calU-(-1)^{j-1}e^{\te_j R}Q_{\lam_j}\right)\wedge\De g,  \chi_{L,j}e^{\te_j R}\Lam Q_{\lam_j}\Ra \right|\\
        \lesssim & \left\|\frac{\De g}{r}\right\|_{L^2}\left\|r\chi_{L,j}e^{\te_j R}\Lam Q_{\lam_j}\wedge \left(\calU-(-1)^{j-1}e^{\te_j R}Q_{\lam_j}\right) \right\|_{L^2}\\
        \lesssim &\ (o_{T_0\to-\infty}(1)+o_{L\to\infty}(1)    )\lam_j^2\|g\|_{\dot \calH^3}.
    \end{align}
     The remaining three terms in ~\eqref{eq:g-linear-decomposition} can be estimated similarly to obtain the same bound  $(o_{T_0\to-\infty}(1)+o_{L\to\infty}(1)    )\lam_j^2\|g\|_{\dot \calH^3}$ .\\
    Since $\sup_{r>0}r^2F_j(t,r)\lesssim\lam_j^2(t)$, we have
    \begin{align}
        \left|\La g\wedge\De g,  \chi_{L,j}e^{\te_j R}\Lam Q_{\lam_j}\Ra\right|
        \lesssim  \left\|\frac{g}{r}\right\|_{L^2}\left\|r^2F_j\frac{\De g}{r}\right\|_{L^2}
        \lesssim \lam_j^2\|g\|_{\dot \calH^1}\|g\|_{\dot \calH^3}=o_{T_0\to-\infty}(1)\lam_j^2\|g\|_{\dot \calH^3}.
    \end{align}
The forcing terms in ~\eqref{eq:rough-modulation-theta-raw} can be also estimated in the same way to obtain the bound 
    
    $(o_{T_0\to-\infty}(1)+o_{L\to\infty}(1)    )\lam_j^2\|g\|_{\dot \calH^3}+\left(\frac{\lam_{j-1}}{\lam_j}\right)^k+ \left(\frac{\lam_{j}}{\lam_{j+1}}\right)^k$ .
    
    Combining the estimates of coefficient matrix and forcing terms above, we get
    \begin{align}
        &\ \lam_j|\lam_j'|+\lam_j^2|\te_j'|\\
        \lesssim &   \left(o_{L\to\infty}(1)+o_{T_0\to-\infty}(1)\right)\lam_j^2\|g\|_{\dot \calH^3}+\left(\frac{\lam_{j-1}}{\lam_{j}}\right)^k+\left(\frac{\lam_{j}}{\lam_{j+1}}\right)^k\\
        &+\sum_{\ell< j}\left[\left(o_{L\to\infty}(1)+o_{T_0\to-\infty}(1)\right)\lam_{\ell}^2\|g\|_{\dot \calH^3}+\left(\frac{\lam_{\ell-1}}{\lam_{\ell}}\right)^k+\left(\frac{\lam_{\ell}}{\lam_{\ell+1}}\right)^k\right]\cdot\left(\frac{\lam_{\ell}}{\lam_j}\right)^{k-2}\\
        &+\sum_{\ell> j}\left[\left(o_{L\to\infty}(1)+o_{T_0\to-\infty}(1)\right)\lam_{\ell}^2\|g\|_{\dot \calH^3}+\left(\frac{\lam_{\ell-1}}{\lam_{{\ell}}}\right)^k+\left(\frac{\lam_{{\ell}}}{\lam_{{\ell}+1}}\right)^k\right]\cdot\left(\frac{\lam_{j}}{\lam_{\ell}}\right)^{k}\\
        \lesssim & \left(o_{L\to\infty}(1)+o_{T_0\to-\infty}(1)\right)\lam_j^2\|g\|_{\dot \calH^3}+\left(\frac{\lam_{j-1}}{\lam_{j}}\right)^k+\left(\frac{\lam_{j}}{\lam_{j+1}}\right)^k,
    \end{align}
    Then ~\eqref{eq:lam'-rough-mod} and ~\eqref{eq:te'-rough-mod} are proven.

    The estimate ~\eqref{eq:lam-hatlam-te-hatte-close} follows directly by Cauchy's inequality.

    We turn to prove ~\eqref{eq:lam-hat-modulation} and ~\eqref{eq:te-hat-modulation}.
    For ~\eqref{eq:lam-hat-modulation}, we calculate
    \begin{equation}\label{eq:lam'-hat-expression}
        \begin{aligned}
            \hat\lam_j'&=\lam_j'-\frac{(-1)^j}{\|\Lam Q\|_{L^2}^2}\La \lam_{j}'\p_{\lam_{j}}\calU,e^{\te_j R}  \Lam Q_{\underline{\lam_j}}  \Ra\\
        &\quad +\frac{(-1)^j}{\|\Lam Q\|_{L^2}^2}\La \calU\wedge\De\calU+\calU\wedge\De g+g\wedge\De\calU+g\wedge \De g,e^{\te_j R}  \Lam Q_{\underline{\lam_j}}  \Ra\\
        &\quad -\frac{(-1)^j}{\|\Lam Q\|_{L^2}^2}\La \sum_{\ell\neq j}\lam_{\ell}'\p_{\lam_{\ell}}\calU+\sum_{1\leq\ell\leq N}\te_{\ell}'\p_{\te_{\ell}}\calU,e^{\te_j R}  \Lam Q_{\underline{\lam_j}}  \Ra\\
        &\quad + \frac{(-1)^j}{\|\Lam Q\|_{L^2}^2}\La g, \te_j'e^{\te_j R}R\Lam   Q_{\underline{\lam_j}}-\frac{\lam_j'}{\lam_j}e^{\te_j R}  \ULam \Lam Q_{\underline{\lam_j}} \Ra.
        \end{aligned}
    \end{equation}
     For terms in the first line, we have
   \begin{align}
      &1-\frac{(-1)^j}{\|\Lam Q\|_{L^2}^2}\La \p_{\lam_{j}}\calU,e^{\te_j R}  \Lam Q_{\underline{\lam_j}}  \Ra\\
      =& \frac{1}{\|\Lam Q\|_{L^2}^2\lam_j^2}\int\left(1-\frac{1}{|\calQ|}\right)|\Lam Q_{\lam_j}|^2\ud x+\frac{1}{\|\Lam Q\|_{L^2}^2\lam_j^2}\int\frac{\left(\calU\cdot e^{\te_j R}\Lam Q_{\lam_j}\right)^2}{|\calQ|}\ud x,
   \end{align}
    which combined with ~\eqref{eq:Q-cal-expansion-local}, ~\eqref{eq:Q-cal-norm}, bootstrap assumption and ~\eqref{eq:lam'-rough-mod} yields
    \begin{equation}\label{eq:estimate-111111}
        \begin{aligned}
             &\lam_j\left|\lam_j'-\frac{(-1)^j}{\|\Lam Q\|_{L^2}^2}\La \lam_{j}'\p_{\lam_{j}}\calU,e^{\te_j R}  \Lam Q_{\underline{\lam_j}}  \Ra\right|\\
        \lesssim& \lam_j|\lam_j'|\frac{1}{\lam_j^2}\int F_j^2\left(\left(\sum_{1\leq\ell<m\leq N}F_{\ell}F_{m}\right)+\left(\sum_{\ell\neq j}F_{\ell}\right)^2\right)\udr\\
        &+  \lam_j|\lam_j'|\frac{1}{\lam_j^2}\left(\int_0^{\sqrt{\lam_{j-1}\lam_j}}+\int_{\sqrt{\lam_j\lam_{j+1}}}^{\infty}\right) F_j^2\udr   \\
        \lesssim& \lam_j|\lam_j'| \left[\left(\frac{\lam_{j-1}}{\lam_j}\right)^{k-1}+\left(\frac{\lam_{j}}{\lam_{j+1}}\right)^{k-1}\right]
        \lesssim  o(|t|^{-\veps_j}) \frac{\lam_j^2}{|t|}.
        \end{aligned}
    \end{equation}
     The term containing $\calU\wedge\De\calU$ dominates the dynamics. It follows from ~\eqref{eq:U-wedge-De-U-expansion-local} that
    \begin{align}
        &\lam_j\frac{(-1)^j}{\|\Lam Q\|_{L^2}^2}\La \calU\wedge\De\calU,e^{\te_j R}  \Lam Q_{\underline{\lam_j}}  \Ra\\ 
        =&\frac{(-1)^j}{\|\Lam Q\|_{L^2}^2}16\pi k^2\left[\left( \frac{\lam_{j-1}}{\lam_j} \right)^k \sin(\te_{j-1} - \te_j) -\left( \frac{\lam_j}{\lam_{j+1}} \right)^k \sin(\te_j - \te_{j+1})  \right]+ o(|t|^{-\veps_j}) \frac{\lam_j^2}{|t|}.
    \end{align}
The term $\frac{(-1)^j}{\|\Lam Q\|_{L^2}^2}\La \calU\wedge\De g+g\wedge\De\calU+g\wedge \De g,e^{\te_j R}  \Lam Q_{\underline{\lam_j}}  \Ra$ can be estimated in the same way as $\La \calU\wedge\De g+g\wedge\De\calU+g\wedge \De g,  \chi_{L,j}e^{\te_j R}\Lam Q_{\lam_j}\Ra$ using slightly sharper estimates on $\calU$. And the remaining terms in ~\eqref{eq:lam'-hat-expression} can be estimated easily using $Q\cdot RQ=Q\cdot\Lam Q=RQ\cdot \Lam Q=0$ and ~\eqref{eq:lam'-rough-mod},~\eqref{eq:te'-rough-mod}. We omit the details.

    The estimate ~\eqref{eq:te-hat-modulation} follows in the same way as in the proof of ~\eqref{eq:lam-hat-modulation}. This completes the proof.
      
\end{proof}

\section{Control of high Sobolev norm}\label{sec:high-Sobolev}
In this section we will improve the estimate of $\|g\|_{\dot\calH^3}$. We will introduce a Morawetz correction in the spirit of \cite{HK-26-WM}. We use Darboux coordinates to make the form of Morawetz correction clean. We will consider the orthogonal frame defined in Section~\ref{subsec:Darboux} with $\calV$ replaced by $\calU$. We will omit the superscript and denote the orthogonal frame by $(f_1(t,r),f_2(t,r),\calU_{\vec p(t)}(r))$. We will abuse the notation  $\calU(t,r)=\calU_{\vec p(t)}(r) $ for convenience.   The estimates for the frame is presented in Lemma~\ref{lem:frame-estimate}.

\subsection{Evolution equation in Darboux coordinates}
In this section we consider the Darboux coordinates $z\in \calH^1$ defined in ~\eqref{eq:v-z} with $v,\calV$ replaced by $u,\calU$. We will omit the superscript of related quantities defined in Section~\ref{subsec:Darboux}. 
We will derive the evolution equation for $z$ and analyze the leading component.

It is easy to see that there exists a constant $\hat T_0<0$ such that for $T_0\leq \hat T_0$, there holds uniformly in $T_0$ and $t\in[T,T_0]$
\begin{align}\label{eq:g-z-equivalent}
    \|g\|_{\dot\calH^1}\sim\|z\|_{\calH^1},\ \|g\|_{\dot\calH^3}\sim\|z\|_{\calH^3}.
\end{align}
We denote by $\hat C$ the constant depending only on $k,N$ such that 
\begin{align}\label{eq:C-hat-def}
    \|g\|_{\dot \calH^3}^2\leq \hat C\|z\|_{\calH^3}^2.
\end{align}
We will always assume $T_0\leq \hat T_0$.

For vector-valued function $g$ and complex-valued function $z=z_1+iz_2$, define
\begin{align}\label{eq:notations-TPLTE}
    T_{\calU}g=g\cdot f_1+ig\cdot f_2,\ P_{\calU}g=g-(g\cdot \calU)\calU ,\ \calL_{\calU}g=\calU\wedge \De g+g\wedge\De\calU,\ \ti T_{\calU} z= z_1f_1+z_2 f_2.
\end{align}
Define
\begin{align}
    \calH_{\calU} z= T_{\calU} \left(-\De (\ti T_{\calU}z)-|\na \calU|^2 \ti T_{\calU}z\right),
\end{align}
then for complex-valued functions $z,w$ we have
\begin{align}
    \La \calH_{\calU}z,w  \Ra &=\int_{\R^2}\left( \na \ti T_{\calU}z\cdot   \na \ti T_{\calU}w-|\na\calU|^2\ti T_{\calU}z\cdot\ti T_{\calU}w \right)\ud x,\\
    &= 2\pi\Re\int_0^{\infty}\left(\overline{\p_r z}\p_r w+\bar z V_{\calU}w\right)\udr.  \label{eq:HU}
\end{align}
Here
\begin{align}
    V_{\calU}z&=\left[ \left(
-(f_2\cdot\p_r\calU)^2
+\frac{k^2}{r^2}\left(|Rf_1|^2-|R\calU|^2\right)
\right)z_1+ \left(
(f_1\cdot\p_r\calU)(f_2\cdot\p_r\calU)
+\frac{k^2}{r^2}Rf_1\cdot Rf_2
\right)z_2\right] \\
&\quad +i\left[\left(
(f_1\cdot\p_r\calU)(f_2\cdot\p_r\calU)
+\frac{k^2}{r^2}Rf_1\cdot Rf_2
\right)z_1+\left(
-(f_1\cdot\p_r\calU)^2
+\frac{k^2}{r^2}\left(|Rf_2|^2-|R\calU|^2\right)
\right)z_2\right].
\end{align}
Equivalently, we have
\begin{align}\label{eq:HU-expression}
    \calH_{\calU}=-\p_r^2-\frac{\p_r}{r}+V_{\calU}.
\end{align}
Note that, if we replace $\calU$ by $e^{\te R}Q_{\lam}$ for some $(\te,\lam)\in\bT\times\R_{>0}$, then the operator $\calH_{\calU}$ becomes $H_{\lam}$. Let
\begin{align}
    H_0=-\p_r^2-\frac{1}{r}\p_r+\frac{k^2}{r^2}, \ V_j=H_{\lam_j}-H_0=-\frac{2k^2}{r^2}\sin^2\phi(\lam_j^{-1}r).
\end{align}
Denote by
\begin{align}
    \E_{\calU}=\calH_{\calU}-H_0-\sum_{j=1}^N V_j.
\end{align}
The estimates concerning $\E_{\calU}$ are shown in Lemma~\ref{lem:operator-estimate}.
Since $|\calU|=1$, we have $\calU\cdot\De\calU+|\na \calU|^2=\frac{1}{2}\De|\calU|^2=0 $, which combined with $\ti T_{\calU}z\cdot \calU=0$ yields
    \begin{align}
        T_{\calU}\left(\ti T_{\calU}z\wedge\De\calU\right)&=T_{\calU}\left(\ti T_{\calU}z\wedge\left(-|\na\calU|^2\calU\right)\right),
    \end{align}
which implies
\begin{align}\label{eq:eq:HU-2}
    \calH_{\calU}z&=iT_{\calU}\left(\calL_{\calU}\ti T_{\calU}z\right).
\end{align}
For $\calH_{\calU}$ we have the following coercive properties:
\begin{lem}\label{lem:HU-coercive}
    There exists $L_0>0$ such that, for $L>L_0$ and for $w\in \calH^1$ satisfying
    \begin{align}
        \left\langle \ti T_{\calU}w, \chi_{L,j}e^{\te_j R}\Lam Q_{\lam_j}\right\rangle=\left\langle \ti T_{\calU}w,\chi_{L,j}e^{\te_j R}R Q_{\lam_j}\right\rangle=0,\ j=1,\dots,N,
    \end{align}
    there holds
    \begin{align}\label{eq:HU-coercove-H1}
        \La \calH_{\calU}w,w \Ra\gtrsim\|w\|_{\calH^1}^2,
    \end{align}
    Moreover, if we also have $w\in \calH^3$, then we have
    \begin{align}\label{eq:HU-coercove-H3}
        \|\calH_{\calU}w\|_{\calH^1}\geq C_{L}\|w\|_{\calH^3}
    \end{align}
    for some constant $C_{L}$ depending on $L$.
\end{lem}
\begin{proof}
    The proof is similar to that of Lemma 2.19 in \cite{JL-WM-SRC} and that of Lemma 2.7 in \cite{KM-classification}. We omit the details.
\end{proof}

Explicit estimates about $\calU$ are collected in Appendix~\ref{Appendix:estimates}, which will be frequently used in the remaining parts of Section~\ref{sec:high-Sobolev}. 

\bigskip

Now we derive the evolution equation for $z$.

Since $u=\calU_{\vec p}+g$ and $g\cdot \calU_{\vec p}=-\frac{1}{2}|g|^2$, we have
\begin{align}
    z(t,r)=\frac{g\cdot f_1+ig\cdot f_2}{\sqrt{1-\frac{|g|^2}{4}}},\ |z(t,r)|=|g(t,\te,r)|.
\end{align}
Differentiating by $t$, we have
\begin{align}
    \p_tz=\frac{\p_tg\cdot f_1+i\p_tg\cdot f_2}{\sqrt{1-\frac{|g|^2}{4}}}
    +\frac{g\cdot \p_t f_1+ig\cdot \p_t f_2}{\sqrt{1-\frac{|g|^2}{4}}}
     +\frac{1}{4}\frac{g\cdot f_1+ig\cdot f_2}{\left(1-\frac{|g|^2}{4}\right)\sqrt{1-\frac{|g|^2}{4}}}g\cdot g_t.
\end{align}
Denote by
\begin{align}
    N_{\calU}g=\calL_{\calU}\left(-\frac{1}{2}|g|^2\calU\right)+g\wedge \De g.
\end{align}
Since 
\begin{align}
    \p_t g &= \calU\wedge \De \calU+\calL_{\calU}g+g\wedge \De g-\p_t \calU\\
    &=\calU\wedge \De \calU+\calL_{\calU}\left(g+\frac{1}{2}|g|^2\calU\right)+\calL_{\calU}\left(-\frac{1}{2}|g|^2\calU\right)+g\wedge \De g-\p_t \calU\\
    &=\calU\wedge \De \calU+\calL_{\calU}\left(P_{\calU}g\right)+N_{\calU}g-\p_t \calU,
\end{align}
we have 
\begin{align}\label{eq:zt-raw}
    \p_tz&=\frac{1}{\sqrt{1-\frac{|g|^2}{4}}}T_{\calU}\left(\calU\wedge \De \calU+\calL_{\calU}\left(P_{\calU}g\right)+N_{\calU}g-\p_t \calU\right)\\
    &\quad +\frac{g\cdot \p_t f_1+ig\cdot \p_t f_2}{\sqrt{1-\frac{|g|^2}{4}}}\\
    &\quad +\frac{1}{4}\frac{g\cdot f_1+ig\cdot f_2}{\left(1-\frac{|g|^2}{4}\right)\sqrt{1-\frac{|g|^2}{4}}}g\cdot \left(\calU\wedge \De \calU+\calL_{\calU}\left(P_{\calU}g\right)+N_{\calU}g-\p_t \calU\right).
\end{align}
Combining $u=\calU+g$ with ~\eqref{eq:v-z} where $v,\calV$ replaced by $u,\calU$ and the fact that $|z|=|g|$, we have
\begin{align}\label{eq:PU-g-z}
    P_{\calU}g=\sqrt{1-\frac{|z|^2}{4}}\left(z_1f_1+z_2f_2\right)=\sqrt{1-\frac{|z|^2}{4}}\ti T_{\calU}z.
\end{align}
Denote by
\begin{align}
    \calR_{\rm{mod}}&:= -\left(\frac{1}{\sqrt{1-\frac{|g|^2}{4}}}-1\right)T_{\calU}\left(\p_t \calU\right)+\frac{g\cdot \p_t f_1+ig\cdot \p_t f_2}{\sqrt{1-\frac{|z|^2}{4}}}-\frac{1}{4}\frac{g\cdot f_1+ig\cdot f_2}{\left(1-\frac{|g|^2}{4}\right)\sqrt{1-\frac{|g|^2}{4}}}g\cdot \left(\p_t \calU\right).
\end{align}
Then combining with ~\eqref{eq:HU} we have
\begin{align}
    \p_tz=-iG_{z}\cH_{\calU}z+T_{\calU}\left(\calU\wedge\De\calU-\p_t\calU\right)+\cR_{\rm{NL}}+\calR_{\rm{mod}}.
\end{align}
Here
\begin{align}
    \calR_{\rm{NL}}&=iG_z\calH_{\calU}z+\left(\frac{1}{\sqrt{1-\frac{|z|^2}{4}}}-1\right)T_{\calU}\left(\calU\wedge \De \calU\right)\\
    &\quad +\frac{1}{\sqrt{1-\frac{|z|^2}{4}}}T_{\calU}\left(\calL_{\calU}\left(P_{\calU}g\right)+N_{\calU}g\right)\\
    &\quad +\frac{z}{4-|z|^2}g\cdot \left(\calU\wedge \De \calU+\calL_{\calU}\left(P_{\calU}g\right)+N_{\calU}g\right).
\end{align}

\subsection{Construction of the correction}

The term $T_{\calU}\left(\calU\wedge\De\calU-\p_t\calU\right)$ appearing in the expression of $\p_t z$ cannot be ignored, and in this section we seek for a correction. 

\begin{lem}
    \label{lem:correction-z-eta}
    For $L>1$ sufficiently large, there exists $T_L<0$ such that for $T_0<T_L$, there exist unique $\eta(t,r)\in \C$ and $b_j^{\lam}(t),b_j^{\te}(t)\in \R$ such that for $t\in[T,T_1]$, there holds
    \begin{align}
        i\calH_{\calU}\eta=-T_{\calU}(\calU\wedge\De\calU)+\sum_{j=1}^N\frac{b_j^{\lam}}{\lam_j}T_{\calU}(\p_{\lam_j}\calU)+\sum_{j=1}^N\frac{b_j^{\te}}{\lam_j^2}T_{\calU}(\p_{\te_j}\calU),\label{eq:eta-b}\\
        \left\langle \ti T_{\calU}\eta, \chi_{L,j}e^{\te_j R}\Lam Q_{\lam_j}\right\rangle=\left\langle \ti T_{\calU}\eta,\chi_{L,j}e^{\te_j R}R Q_{\lam_j}\right\rangle=0,\ j=1,\dots,N. \label{eq:eta-ortho}
    \end{align}
    Moreover, we have
    \begin{align}
        \|\eta\|_{\calH^1}+\sum_{j=1}^N\left(|b_j^{\lam}|+|b_j^{\te}|\right)&\lesssim |t|^{-1-\frac{2}{k-2}},\label{eq:correction-H1}\\
        \|\eta\|_{\calH^3}+\sum_{j=1}^N\frac{1}{\lam_j^2}\left(|b_j^{\lam}|+|b_j^{\te}|\right)&\lesssim_L |t|^{-1},\label{eq:correction-H3}\\
        \|\p_t\eta\|_{\calH^1}+\sum_{j=1}^N\left(|\p_tb_j^{\lam}|+|\p_tb_j^{\te}|\right)&\lesssim \max_{1\leq j\leq N}\left(\frac{|\lam_j'|}{\lam_j}+|\te_j'|\right)|t|^{-1-\frac{2}{k-2}},\label{eq:correction-t-H1}\\
        \|\p_t\eta\|_{\calH^3}+\sum_{j=1}^N\frac{1}{\lam_j^2}\left(|\p_t b_j^{\lam}|+|\p_t b_j^{\te}|\right)&\lesssim_L \max_{1\leq j\leq N}\left(\frac{|\lam_j'|}{\lam_j}+|\te_j'|\right) |t|^{-1},\label{eq:correction-t-H3}\\
        \|r^2\p_r^2 (\calH_{\calU}\eta)\|_{L^2(\udr)}&\lesssim |t|^{-1-\frac{1}{k-2}}.\label{eq:correction-H4}
    \end{align}
\end{lem}
\begin{proof}
    We will use the inequalities of $T_{\calU}\left(\calU\wedge\De\calU\right)$ proven in Lemma~\ref{lem:F-estimate}. 
    Define 
    \begin{align}
        X=\{w\in \calH^1: \left\langle \ti T_{\calU}w, \chi_{L,j}e^{\te_j R}\Lam Q_{\lam_j}\right\rangle=\left\langle \ti T_{\calU}w,\chi_{L,j}e^{\te_j R}R Q_{\lam_j}\right\rangle=0,\ j=1,\dots,N\}
    \end{align}
    We view it as a real Hilbert space with the inner product inherit from $\calH^1$.\\
    \textbf{Step 1.}  For any fixed $\ti b_{j}^{\lam},\ti b_{j}^{\te}$,  we claim that, there exists a unique $\eta_{b}\in X$, such that
    \begin{align}\label{eq:claim-1}
        \La \calH_{\calU} \eta_b , v  \Ra=\La iT_{\calU}(\calU\wedge\De\calU)-\sum_{j=1}^N i \frac{\ti b_j^{\lam}}{\lam_j}T_{\calU}(\p_{\lam_j}\calU)-\sum_{j=1}^N i\frac{\ti b_j^{\te}}{\lam_j^2}T_{\calU}(\p_{\te_j}\calU)  ,v \Ra
    \end{align}
    holds for any $v\in X$.
    The proof of the claim is a standard application of Lax-Milgram Theorem. Now we examine the required conditions. First of all, using ~\eqref{eq:HU} and the fact that
    \begin{align}
        |\na\ti T_{\calU}z|\lesssim |\p_r z|+\frac{|z|}{r},\ |\na \calU|\lesssim\frac{1}{r},
    \end{align}
    we have
    \begin{align}
        \La\calH_{\calU} z,w\Ra\lesssim \| z \|_{\calH^1}\|w\|_{\calH^1}.
    \end{align}
     The coercivity of $\La\calH_{\calU} \cdot,\cdot\Ra$ is a consequence of Lemma~\ref{lem:HU-coercive}.
    Using Cauchy's inequality we have
     \begin{align}
         &\left|\La iT_{\calU}(\calU\wedge\De\calU)-\sum_{j=1}^N i \frac{\ti b_j^{\lam}}{\lam_j}T_{\calU}(\p_{\lam_j}\calU)-\sum_{j=1}^N i\frac{\ti b_j^{\te}}{\lam_j^2}T_{\calU}(\p_{\te_j}\calU)  ,v \Ra\right|\\
         \lesssim &\left(\|T_{\calU}(\calU\wedge\De\calU) \|_{(\calH^1)^*}+\sum_{j=1}^N\left(|\ti b_j^{\lam}|+|\ti b_j^{\te}|\right)\right)\|v\|_{\calH^1}.
     \end{align}
    Then Lax-Milgram Theorem yields the existence of $\eta_b\in X$ with
    \begin{align}\label{eq:etab-H1-raw}
        \|\eta_b\|_{\calH^1}\lesssim  \|T_{\calU}(\calU\wedge\De\calU) \|_{(\calH^1)^*}+\sum_{j=1}^N\left(|\ti b_j^{\lam}|+|\ti b_j^{\te}|\right).
    \end{align}
    \textbf{Step 2.} We claim that, there exist unique $ b_{j}^{\lam}, b_{j}^{\te}$ such that, for related $\eta$, there holds
    \begin{align}\label{eq:claim-2}
        \La \calH_{\calU} \eta , v  \Ra=\La iT_{\calU}(\calU\wedge\De\calU)-\sum_{j=1}^N i \frac{ b_j^{\lam}}{\lam_j}T_{\calU}(\p_{\lam_j}\calU)-\sum_{j=1}^N i\frac{ b_j^{\te}}{\lam_j^2}T_{\calU}(\p_{\te_j}\calU)  ,v \Ra
    \end{align}
    for any $v\in \calH^1$. Let
    \begin{align}
        R_b=\calH_{\calU} \eta-iT_{\calU}(\calU\wedge\De\calU)+\sum_{j=1}^N i \frac{ b_j^{\lam}}{\lam_j}T_{\calU}(\p_{\lam_j}\calU)+\sum_{j=1}^N i\frac{ b_j^{\te}}{\lam_j^2}T_{\calU}(\p_{\te_j}\calU).  
    \end{align}
    Combining with the result in Step 1, proving claim ~\eqref{eq:claim-2} is equivalent to finding $ b_{j}^{\lam}, b_{j}^{\te}$ such that, for related $\eta$, there holds
    \begin{align}\label{eq:claim-2-ortho}
        \left\langle R_b, T_{\calU}\left( \chi_{L,j}e^{\te_j R}\Lam Q_{\lam_j}\right)\right\rangle=\left\langle R_b, T_{\calU}\left(\chi_{L,j}e^{\te_j R}R Q_{\lam_j}\right)\right\rangle=0,\ j=1,\dots,N.
    \end{align}
Denote by $\eta_0$ the function obtained in Step 1 with $\ti b_j^{\lam}=\ti b_j^{\te}=0$ for all $j$. Since the right-hand side in ~\eqref{eq:claim-1} is linear in each $\ti b_j^{\lam}$ and $\ti b_j^{\te}$, taking $\ti b_j^{\lam}=1$ ( or $\ti b_j^{\te}=1$) with other parameters being zero, we obtain $2N$ functions $\eta_j^{\lam},\eta_j^{\te}$ independent of $\ti b_j^{\lam},\ti b_j^{\te}$ such that
    \begin{align}
        \eta_b=\eta_0+\sum_{j=1}^N\ti b_j^{\lam} (\eta_j^{\lam}-\eta_0)+\sum_{j=1}^N\ti b_j^{\te} (\eta_j^{\te}-\eta_0).
    \end{align}
    Then ~\eqref{eq:claim-2-ortho} is actually a linear system of $ b_j^{\lam}, b_j^{\te}$ . 
    Combining the argument as in Step 1 and ~\eqref{eq:F-p-X}, we have
    \begin{align}
        \|\eta_j^{\lam}-\eta_0\|_{\calH^1}+\| \eta_j^{\te}-\eta_0\|_{\calH^1}\lesssim o_{L\to\infty}(1)+\left(\frac{\lam_{j-1}}{\lam_{j}}\right)^k+\left(\frac{\lam_{j}}{\lam_{j+1}}\right)^k.
    \end{align}
     Denote by
    \begin{align}
         B_{j,\ell}^{\lam,\Lam}&=\La i\frac{1}{\lam_j}T_{\calU}(\p_{\lam_j}\calU),T_{\calU}\left( \chi_{L,\ell}e^{\te_{\ell} R}\Lam Q_{\lam_{\ell}}\right)\Ra+\La \calH_{\calU}(\eta_j^{\lam}-\eta_0),T_{\calU}\left( \chi_{L,\ell}e^{\te_{\ell} R}\Lam Q_{\lam_\ell}\right)\Ra,\\
        B_{j,\ell}^{\te,\Lam}&=\La i\frac{1}{\lam_j^2}T_{\calU}(\p_{\te_j}\calU),T_{\calU}\left( \chi_{L,\ell}e^{\te_\ell R}\Lam Q_{\lam_{\ell}}\right)\Ra+\La \calH_{\calU}(\eta_j^{\te}-\eta_0),T_{\calU}\left( \chi_{L,\ell}e^{\te_{\ell} R}\Lam Q_{\lam_{\ell}}\right)\Ra,\\
        B_{j,\ell}^{\lam,R}&=\La i\frac{1}{\lam_j}T_{\calU}(\p_{\lam_j}\calU),T_{\calU}\left( \chi_{L,\ell}e^{\te_{\ell} R}R Q_{\lam_{\ell}}\right)\Ra+\La \calH_{\calU}(\eta_j^{\lam}-\eta_0),T_{\calU}\left( \chi_{L,\ell}e^{\te_{\ell} R}R Q_{\lam_\ell}\right)\Ra,\\
         B_{j,\ell}^{\te,R}&=\La i\frac{1}{\lam_j^2}T_{\calU}(\p_{\te_j}\calU),T_{\calU}\left( \chi_{L,\ell}e^{\te_\ell R}R Q_{\lam_{\ell}}\right)\Ra+\La \calH_{\calU}(\eta_j^{\te}-\eta_0),T_{\calU}\left( \chi_{L,\ell}e^{\te_{\ell} R}R Q_{\lam_{\ell}}\right)\Ra.
    \end{align}
    Then for $\ell=1,\dots,N$, there holds
    \begin{equation}\label{eq:linear-bB-1}
        \begin{aligned}
             \sum_{j=1}^N b_j^{\lam}B_{j,\ell}^{\lam,\Lam}+\sum_{j=1}^N b_j^{\te}B_{j,\ell}^{\te,\Lam}&=\La -\calH_{\calU}\eta_0+iT_{\calU}(\calU\wedge\De\calU),T_{\calU}\left( \chi_{L,\ell}e^{\te_{\ell} R}\Lam Q_{\lam_{\ell}}\right)\Ra,\\
        \sum_{j=1}^N b_j^{\lam}B_{j,\ell}^{\lam,R}+\sum_{j=1}^N b_j^{\te}B_{j,\ell}^{\te,R}&=\La -\calH_{\calU}\eta_0+iT_{\calU}(\calU\wedge\De\calU),T_{\calU}\left( \chi_{L,\ell}e^{\te_{\ell} R}R Q_{\lam_{\ell}}\right)\Ra.
        \end{aligned}
    \end{equation}

    We first compute the dominant part of $B_{j,\ell}$.
    \begin{align}
        B_{j,\ell}^{\te,\Lam}&=\La i\frac{1}{\lam_j^2}T_{\calU}(\p_{\te_j}\calU),T_{\calU}\left( \chi_{L,\ell}e^{\te_{\ell} R}\Lam Q_{\lam_{\ell}}\right)\Ra+\La \calH_{\calU}(\eta_j^{\te}-\eta_0),T_{\calU}\left( \chi_{L,\ell}e^{\te_{\ell} R}\Lam Q_{\lam_{\ell}}\right)\Ra\\
        &=\La \frac{1}{\lam_j^2}\calU\wedge\p_{\te_j}\calU, \chi_{L,\ell}e^{\te_{\ell} R}\Lam Q_{\lam_{\ell}}\Ra+O(\|\eta_j^{\te}-\eta_0\|_{\calH^1}\|T_{\calU}(\chi_{L,\ell}e^{\te_{\ell} R}\Lam Q_{\lam_\ell}) \|_{\calH^1})\\
        &=-\frac{4\pi^2}{k\sin(\frac{\pi}{k})}\de_{j\ell}+o_{L\to\infty}(1),
    \end{align}
    where $\de_{j\ell}$ denotes the Kronecker delta. Similarly
    \begin{align}
         B_{j,\ell}^{\lam,\Lam}=o_{L\to\infty}(1),\ B_{j,\ell}^{\lam, R}=-\frac{4\pi^2}{k\sin(\frac{\pi}{k})}\de_{j\ell}+o_{L\to\infty}(1), \ B_{j,\ell}^{\te,R}=o_{L\to\infty}(1).
    \end{align}
    Here we write the remainder as $o_{L\to \infty}(1)$ since the choice of  $T_0$  is related to $L$. 

    Then we estimate the right hand side of ~\eqref{eq:linear-bB-1}. Using ~\eqref{eq:F-H-1} we obtain
    \begin{align}
        &\left|\La -\calH_{\calU}\eta_0+iT_{\calU}(\calU\wedge\De\calU),T_{\calU}\left( \chi_{L,\ell}e^{\te_{\ell} R}\Lam Q_{\lam_{\ell}}\right)\Ra\right|\\
        \lesssim& \|\eta_0\|_{\calH^1}\|T_{\calU}\left( \chi_{L,\ell}e^{\te_{\ell} R}\Lam Q_{\lam_{\ell}}\right) \|_{\calH^1}+\left|\La T_{\calU}(\calU\wedge\De\calU),T_{\calU}\left( \calU\wedge\chi_{L,\ell}e^{\te_{\ell} R}\Lam Q_{\lam_{\ell}}\right)\Ra\right|\\
        \lesssim& \|\eta_0\|_{\calH^1}+\| T_{\calU}(\calU\wedge\De\calU)\|_{(\calH^1)^*} 
        \lesssim \| T_{\calU}(\calU\wedge\De\calU)\|_{(\calH^1)^*} \lesssim |t|^{-1-\frac{2}{k-2}}.
    \end{align}
    Similarly
    \begin{align}
        \left|\La -\calH_{\calU}\eta_0+iT_{\calU}(\calU\wedge\De\calU),T_{\calU}\left( \chi_{L,\ell}e^{\te_{\ell} R}R Q_{\lam_{\ell}}\right)\Ra\right|\lesssim  |t|^{-1-\frac{2}{k-2}}.
    \end{align}
    
Therefore the coefficient matrix of $b_j^{\lam}$ and $b_{j}^{\te}$ is invertible, and  we obtain
    \begin{align}\label{eq:bj-H1}
        \sum_{j=1}^N(|b_j^{\lam}|+|b_{j}^{\te}|)\lesssim |t|^{-1-\frac{2}{k-2}}.
    \end{align}

    Combining the results in Step 1 and Step 2 yields the existence and uniqueness of $\eta$, $b_{j}^{\lam}$ and $b_j^{\te}$. ~\eqref{eq:correction-H1} follows from ~\eqref{eq:etab-H1-raw}, ~\eqref{eq:bj-H1} and ~\eqref{eq:F-H-1}.\\
    \textbf{Step 3.} Now we prove ~\eqref{eq:correction-H3} . 
    Using ~\eqref{eq:HU-coercove-H3} we have
    \begin{align}\label{eq:eta-H3-raw}
        \|\eta\|_{\calH^3}\lesssim_L \|\calH_{\calU}\eta\|_{\calH^1}\lesssim \|T_{\calU}\left(\calU\wedge \De\calU\right)\|_{\calH^1}+\sum_{j=1}^N\frac{1}{\lam_j^2}\left(|b_j^{\lam}|+|b_j^{\te}|\right).
    \end{align}
    Take $\ti L\gg L$ to be determined. Since $R_b=0$, we have the following quantities
    \begin{align}
        \lam_j^{-2}\left\langle R_b, T_{\calU}\left( \chi_{\ti L,j}e^{\te_j R}\Lam Q_{\lam_j}\right)\right\rangle=\lam_j^{-2}\left\langle R_b, T_{\calU}\left(\chi_{\ti L,j}e^{\te_j R}R Q_{\lam_j}\right)\right\rangle=0,\ j=1,\dots,N.
    \end{align}
    Define
    \begin{align}
        \ti B_{j,\ell}^{\lam,\Lam}&=\frac{\lam_j^2}{\lam_{\ell}^2}\La i\frac{1}{\lam_j}T_{\calU}(\p_{\lam_j}\calU),T_{\calU}\left( \chi_{\ti L,\ell}e^{\te_{\ell} R}\Lam Q_{\lam_{\ell}}\right)\Ra,\\
        \ti B_{j,\ell}^{\te,\Lam}&=\frac{\lam_j^2}{\lam_{\ell}^2}\La i\frac{1}{\lam_j^2}T_{\calU}(\p_{\te_j}\calU),T_{\calU}\left( \chi_{\ti L,\ell}e^{\te_{\ell} R}\Lam Q_{\lam_{\ell}}\right)\Ra,\\
        \ti B_{j,\ell}^{\lam,R}&=\frac{\lam_j^2}{\lam_{\ell}^2}\La i\frac{1}{\lam_j}T_{\calU}(\p_{\lam_j}\calU),T_{\calU}\left( \chi_{\ti L,\ell}e^{\te_{\ell} R}R Q_{\lam_{\ell}}\right)\Ra,\\
         \ti B_{j,\ell}^{\te,R}&=\frac{\lam_j^2}{\lam_{\ell}^2}\La i\frac{1}{\lam_j^2}T_{\calU}(\p_{\te_j}\calU),T_{\calU}\left( \chi_{\ti L,\ell}e^{\te_{\ell} R}R Q_{\lam_{\ell}}\right)\Ra.
    \end{align}
    Then for $\ell=1,\dots,N$, there holds
    \begin{align}
        \sum_{j=1}^N \frac{b_j^{\lam}}{\lam_j^{2}}\ti B_{j,\ell}^{\lam,\Lam}+\sum_{j=1}^N \frac{b_j^{\te}}{\lam_j^{2}}\ti B_{j,\ell}^{\te,\Lam}&=\frac{1}{\lam_{\ell}^{2}}\La -\calH_{\calU}\eta+iT_{\calU}(\calU\wedge\De\calU),T_{\calU}\left( \chi_{\ti L,\ell}e^{\te_{\ell} R}\Lam Q_{\lam_{\ell}}\right)\Ra,\\
        \sum_{j=1}^N \frac{b_j^{\lam}}{\lam_j^{2}}\ti B_{j,\ell}^{\lam,R}+\sum_{j=1}^N \frac{b_j^{\te}}{\lam_j^{2}}\ti B_{j,\ell}^{\te,R}&=\frac{1}{\lam_{\ell}^{2}}\La -\calH_{\calU}\eta+iT_{\calU}(\calU\wedge\De\calU),T_{\calU}\left( \chi_{\ti L,\ell}e^{\te_{\ell} R}R Q_{\lam_{\ell}}\right)\Ra.
    \end{align}
    Calculating in the similar way as in Step 2, we have
    \begin{align}
        \ti B_{j,\ell}^{\lam,\Lam}&=o_{\ti L\to \infty}(1)+o_{T_0\to -\infty}(1), \\
        \ti B_{j,\ell}^{\te,\Lam}&=-\frac{4\pi^2}{k\sin(\frac{\pi}{k})}\de_{j\ell}+o_{\ti L\to \infty}(1)+o_{T_0\to -\infty}(1),\\
        \ti B_{j,\ell}^{\lam,R}&=-\frac{4\pi^2}{k\sin(\frac{\pi}{k})}\de_{j\ell}+o_{\ti L\to \infty}(1)+o_{T_0\to -\infty}(1), \\
        \ti B_{j,\ell}^{\te,R}&=o_{\ti L\to \infty}(1)+o_{T_0\to -\infty}(1),\\
    \end{align}
   We estimate the right-hand side as
    \begin{align}
        &\left|\frac{1}{\lam_{\ell}^{2}}\La -\calH_{\calU}\eta+iT_{\calU}(\calU\wedge\De\calU),T_{\calU}\left( \chi_{\ti L,\ell}e^{\te_{\ell} R}\Lam Q_{\lam_{\ell}}\right)\Ra\right|\\
        \lesssim & \|\eta\|_{\calH^3}\|\calH_{\calU}\lam_{\ell}^{-2}T_{\calU}\left( \chi_{\ti L,\ell}e^{\te_{\ell} R}\Lam Q_{\lam_{\ell}}\right)\|_{(\calH^3)^*}+\|T_{\calU}(\calU\wedge\De\calU)\|_{\calH^1}\lam_{\ell}^{-2}\|T_{\calU}\left( \chi_{\ti L,\ell}e^{\te_{\ell} R}\Lam Q_{\lam_{\ell}}\right) \|_{(\calH^1)^*}\\
        \lesssim & \|\eta\|_{\calH^3}(o_{\ti L\to \infty}(1)+C(\ti L)o_{T_0\to-\infty}(1))+|t|^{-1}.
    \end{align}
    Note that the same estimates holds for $\frac{1}{\lam_{\ell}^{2}}\La -\calH_{\calU}\eta+iT_{\calU}(\calU\wedge\De\calU),T_{\calU}\left( \chi_{\ti L,\ell}e^{\te_{\ell} R}R Q_{\lam_{\ell}}\right)$.\\
    Therefore taking $T_0$ sufficiently negative after fixing $\ti L$ large enough, we have
    \begin{align}\label{eq:bj-weighted-high-order}
        \sum_{j=1}^N\frac{1}{\lam_j^2}\left(|b_j^{\lam}|+|b_j^{\te}|\right)\lesssim |t|^{-1}+o_{\ti L\to\infty}(1)\|\eta\|_{\calH^3},
    \end{align}
    which combined with ~\eqref{eq:eta-H3-raw} yields ~\eqref{eq:correction-H3} once we take $\ti L$ sufficiently large relative to $L$.\\
    \textbf{Step 4.} Now we prove ~\eqref{eq:correction-t-H1} and ~\eqref{eq:correction-t-H3}. Differentiating ~\eqref{eq:eta-b} by $t$, we obtain
    \begin{align}
        i\calH_{\calU}\p_t\eta+i(\p_t\calH_{\calU})\eta&=-\p_t\left(T_{\calU}(\calU\wedge\De\calU)\right)+\sum_{j=1}^N\frac{\p_tb_j^{\lam}}{\lam_j}T_{\calU}(\p_{\lam_j}\calU)+\sum_{j=1}^N\frac{\p_tb_j^{\te}}{\lam_j^2}T_{\calU}(\p_{\te_j}\calU)\\
        &\quad +\sum_{j=1}^Nb_j^{\lam} \p_t\left(\frac{1}{\lam_j}T_{\calU}(\p_{\lam_j}\calU)\right)+\sum_{j=1}^N b_j^{\te}\p_t\left(\frac{1}{\lam_j^2}T_{\calU}(\p_{\te_j}\calU)\right).
    \end{align}
Since $\eta$ satisfies ~\eqref{eq:eta-ortho}, we have for $j=1,\dots,N$,
    \begin{align}
        \left\langle \ti T_{\calU}\p_t\eta, \chi_{L,j}e^{\te_j R}\Lam Q_{\lam_j}\right\rangle&=-\left\langle \left(\p_t\ti T_{\calU}\right)\eta, \chi_{L,j}e^{\te_j R}\Lam Q_{\lam_j}\right\rangle-\left\langle \ti T_{\calU}\eta, \p_t\left(\chi_{L,j}e^{\te_j R}\Lam Q_{\lam_j}\right)\right\rangle,\\
         \left\langle \ti T_{\calU}\p_t\eta, \chi_{L,j}e^{\te_j R}R Q_{\lam_j}\right\rangle&=-\left\langle \left(\p_t\ti T_{\calU}\right)\eta, \chi_{L,j}e^{\te_j R}RQ_{\lam_j}\right\rangle-\left\langle \ti T_{\calU}\eta, \p_t\left(\chi_{L,j}e^{\te_j R}R Q_{\lam_j}\right)\right\rangle.
    \end{align}
    This combined with Lemma~\ref{lem:frame-estimate} and the estimates on $Q$ and $\chi_{L,j}$ yields
    \begin{align}
        \sum_{j=1}^N \lam_j^{-2}\left[\left|\left\langle \ti T_{\calU}\p_t\eta, \chi_{L,j}e^{\te_j R}\Lam Q_{\lam_j}\right\rangle\right|+\left|\left\langle \ti T_{\calU}\p_t\eta, \chi_{L,j}e^{\te_j R}RQ_{\lam_j}\right\rangle\right|\right]\lesssim \max_{1\leq j\leq N}\left(\frac{|\lam_j'|}{\lam_j}+|\te_j'|\right) \|\eta\|_{\calH^1},\\
        \sum_{j=1}^N \lam_j^{-4}\left[\left|\left\langle \ti T_{\calU}\p_t\eta, \chi_{L,j}e^{\te_j R}\Lam Q_{\lam_j}\right\rangle\right|+\left|\left\langle \ti T_{\calU}\p_t\eta, \chi_{L,j}e^{\te_j R}RQ_{\lam_j}\right\rangle\right|\right]\lesssim_L \max_{1\leq j\leq N}\left(\frac{|\lam_j'|}{\lam_j}+|\te_j'|\right) \|\eta\|_{\calH^3},
    \end{align}
    Then the proof proceeds in the same way as  in Step 1, Step 2, and Step 3 after taking the projection of $\p_t\eta$ onto $X$.\\
\textbf{Step 5.} Now we prove ~\eqref{eq:correction-H4}. 
    Note that ~\eqref{eq:bj-weighted-high-order} and ~\eqref{eq:correction-H3} yields that for $\ti L$ sufficiently large relative to $L$, we have
    \begin{align}
        \sum_{j=1}^N\frac{1}{\lam_j^2}\left(|b_j^{\lam}|+|b_j^{\te}|\right)\lesssim |t|^{-1}
    \end{align}
    with the implicit constant independent of $L$. Combining this estimate with ~\eqref{eq:correction-H1}, ~\eqref{eq:F-pointwise}, and ~\eqref{eq:U-tangent-pointwise} we have
    \begin{align}
        &\quad \ \|r^2\p_r^2 (\calH_{\calU}\eta)\|_{L^2(\udr)}\\
        &\lesssim \|r^2\p_r^2 (T_{\calU}(\calU\wedge\De\calU))\|_{L^2(\udr)}\\
        &\quad +\sum_{j=1}^N \frac{|b_j^{\lam}|}{\lam_j}\|r^2\p_r^2 (T_{\calU}(\p_{\lam_j}\calU))\|_{L^2(\udr)}+\sum_{j=1}^N \frac{|b_j^{\te}|}{\lam_j^2}\|r^2\p_r^2 (T_{\calU}(\p_{\te_j}\calU))\|_{L^2(\udr)}\\
        &\lesssim \|r^2\p_r^2 (T_{\calU}(\calU\wedge\De\calU))\|_{L^2(\udr)}\\
        &\quad +\sum_{j=1}^N \min\{ |t|^{-1-\frac{2}{k-2}}\lam_j^{-1}, |t|^{-1}\lam_j\}\left[\|r^2\p_r^2 (T_{\calU}(\p_{\lam_j}\calU))\|_{L^2(\udr)}+\frac{1}{\lam_j}\|r^2\p_r^2 (T_{\calU}(\p_{\te_j}\calU))\|_{L^2(\udr)}\right]\\
        &\lesssim |t|^{-1-\frac{1}{k-2}}.
    \end{align}
\end{proof}

\begin{lem}\label{lem:b-lam-close}
    There exists $L_0>0$, such that for $L\geq L_0$, there exists $\hat \de_0>0$, such that for any $0<\de_0<\hat\de_0$, there exists $\hat T_0<0$, such that for any $T_0<\hat T_0$, the following holds:
        For $b_j^{\lam},b_j^{\te}$ constructed in Lemma~\ref{lem:correction-z-eta}, there holds
        \begin{align}\label{eq:bj-weighted-sum}
        &\left(\sum_{j=1}^N\frac{\de_0^{N-j}}{\lam_j^5}\left(|b_j^{\lam}-\lam_j\lam_j'|^2+|b_j^{\te}-\lam_j^2\te_j'|^2\right)\right)^{\frac{1}{2}}\\
        \lesssim & \left(o_{L}(1)+C(L)\max_{1\leq j\leq N-1}\left(\frac{\lam_j}{\lam_{j+1}}\right)^{k}\right)\left(\sum_{j=1}^N\de_0^{N-j}\|A_{\lam_j}(z+\eta)\|_{\rm{Mor}_2(\lam_j)}^2\right)^{\frac{1}{2}}\\
        &+C(L)(|t|^{-2}+\|z\|_{\calH^3}^2).
    \end{align}
\end{lem}
\begin{rem}
    The reason why we choose logarithmic cutoff  function $\chi_L$ is shown in the proof.
\end{rem}
\begin{proof}
        The equation ~\eqref{eq:eta-b} yields
        \begin{equation}\label{eq:U-De-U-eta}
            \begin{aligned}
                 \calU\wedge\De\calU&=\calL_{\calU}\ti T_{\calU}\eta-((\calL_{\calU}\ti T_{\calU}\eta)\cdot \calU)\calU+\sum_{\ell=1}^N \frac{b_{\ell}^{\lam}}{\lam_{\ell}}\p_{\lam_{\ell}}\calU+\sum_{\ell=1}^N \frac{b_{\ell}^{\te}}{\lam_{\ell}^2}\p_{\te_{\ell}}\calU\\
            &=\calL_{\calU}\ti T_{\calU}\eta+((\calU\wedge\De\calU)\cdot\ti T_{\calU}\eta)\calU+\sum_{\ell=1}^N \frac{b_{\ell}^{\lam}}{\lam_{\ell}}\p_{\lam_{\ell}}\calU+\sum_{\ell=1}^N \frac{b_{\ell}^{\te}}{\lam_{\ell}^2}\p_{\te_{\ell}}\calU.
            \end{aligned}
        \end{equation}
        Inserting ~\eqref{eq:U-De-U-eta} into ~\eqref{eq:rough-modulation-lambda-raw} and ~\eqref{eq:rough-modulation-theta-raw} yields
        \begin{equation}\label{eq:b-lambda-raw}
        \begin{aligned}
            &\frac{\de_0^{\frac{N-j}{2}}}{\lam_j^{\frac{5}{2}}}(\lam_j\lam_j'-b_j^{\lam}) \frac{1}{\lam_j}\La \p_{\lam_j}\calU , \chi_{L,j}e^{\te_j R}\Lam Q_{\lam_j}\Ra
        +\frac{\de_0^{\frac{N-j}{2}}}{\lam_j^{\frac{5}{2}}}(\lam_j^2\te_j'-b_{j}^{\te})\frac{1}{\lam_j^2}\La \p_{\te_j}\calU , \chi_{L,j}e^{\te_j R}\Lam Q_{\lam_j}\Ra\\
        &+\sum_{\ell\neq j}\frac{\de_0^{\frac{N-\ell}{2}}}{\lam_{\ell}^{\frac{5}{2}}}(\lam_\ell\lam_\ell'-b_{\ell}^{\lam} )\de_0^{\frac{\ell-j}{2}}\left(\frac{\lam_{\ell}}{\lam_j}\right)^{\frac{5}{2}}\frac{1}{\lam_{\ell}}\La \p_{\lam_\ell}\calU , \chi_{L,j}e^{\te_j R}\Lam Q_{\lam_j}\Ra\\
        &+\sum_{\ell\neq j} \frac{\de_0^{\frac{N-\ell}{2}}}{\lam_{\ell}^{\frac{5}{2}}}(\lam_\ell^2\te_\ell'-b_{\ell}^{\te})\de_0^{\frac{\ell-j}{2}}\left(\frac{\lam_{\ell}}{\lam_j}\right)^{\frac{5}{2}}\frac{1}{\lam_{\ell}^2}\La \p_{\te_\ell}\calU , \chi_{L,j}e^{\te_j R}\Lam Q_{\lam_j}\Ra \\
        =& \frac{\de_0^{\frac{N-j}{2}}}{\lam_j^{\frac{5}{2}}}\Big[\La \calL_{\calU}\ti T_{\calU}\eta+((\calU\wedge\De\calU)\cdot\ti T_{\calU}\eta) \calU,  \chi_{L,j}e^{\te_j R}\Lam Q_{\lam_j}\Ra+  \La \calU\wedge\De g,  \chi_{L,j}e^{\te_j R}\Lam Q_{\lam_j}\Ra\\
        &+  \La g\wedge\De \calU,  \chi_{L,j}e^{\te_j R}\Lam Q_{\lam_j}\Ra+  \La g\wedge\De g,  \chi_{L,j}e^{\te_j R}\Lam Q_{\lam_j}\Ra\\
        &+ \lam_j\lam_j'\frac{1}{\lam_j}\La g , \p_{\lam_j}\left( \chi_{L,j}e^{\te_j R}\Lam Q_{\lam_j}\right)\Ra + \lam_j^2\te_j'\frac{1}{\lam_j^2}\La g , \p_{\te_j}\left( \chi_{L,j}e^{\te_j R}\Lam Q_{\lam_j}\right)\Ra\Big]
        \end{aligned}
    \end{equation}

    \begin{equation}\label{eq:b-theta-raw}
        \begin{aligned}
           &\frac{\de_0^{\frac{N-j}{2}}}{\lam_j^{\frac{5}{2}}}(\lam_j\lam_j'-b_j^{\lam}) \frac{1}{\lam_j}\La \p_{\lam_j}\calU , \chi_{L,j}e^{\te_j R}RQ_{\lam_j}\Ra
        +\frac{\de_0^{\frac{N-j}{2}}}{\lam_j^{\frac{5}{2}}}(\lam_j^2\te_j'-b_{j}^{\te})\frac{1}{\lam_j^2}\La \p_{\te_j}\calU , \chi_{L,j}e^{\te_j R}RQ_{\lam_j}\Ra\\
        &+\sum_{\ell\neq j}\frac{\de_0^{\frac{N-\ell}{2}}}{\lam_{\ell}^{\frac{5}{2}}}(\lam_\ell\lam_\ell'-b_{\ell}^{\lam} )\de_0^{\frac{\ell-j}{2}}\left(\frac{\lam_{\ell}}{\lam_j}\right)^{\frac{5}{2}}\frac{1}{\lam_{\ell}}\La \p_{\lam_\ell}\calU , \chi_{L,j}e^{\te_j R}RQ_{\lam_j}\Ra\\
        &+\sum_{\ell\neq j} \frac{\de_0^{\frac{N-\ell}{2}}}{\lam_{\ell}^{\frac{5}{2}}}(\lam_\ell^2\te_\ell'-b_{\ell}^{\te})\de_0^{\frac{\ell-j}{2}}\left(\frac{\lam_{\ell}}{\lam_j}\right)^{\frac{5}{2}}\frac{1}{\lam_{\ell}^2}\La \p_{\te_\ell}\calU , \chi_{L,j}e^{\te_j R}RQ_{\lam_j}\Ra \\
        =& \frac{\de_0^{\frac{N-j}{2}}}{\lam_j^{\frac{5}{2}}}\Big[\La \calL_{\calU}\ti T_{\calU}\eta+((\calU\wedge\De\calU)\cdot\ti T_{\calU}\eta) \calU,  \chi_{L,j}e^{\te_j R}RQ_{\lam_j}\Ra+  \La \calU\wedge\De g,  \chi_{L,j}e^{\te_j R}RQ_{\lam_j}\Ra\\
        &+  \La g\wedge\De \calU,  \chi_{L,j}e^{\te_j R}RQ_{\lam_j}\Ra+  \La g\wedge\De g,  \chi_{L,j}e^{\te_j R}RQ_{\lam_j}\Ra\\
        &+ \lam_j\lam_j'\frac{1}{\lam_j}\La g , \p_{\lam_j}\left( \chi_{L,j}e^{\te_j R}RQ_{\lam_j}\right)\Ra + \lam_j^2\te_j'\frac{1}{\lam_j^2}\La g , \p_{\te_j}\left( \chi_{L,j}e^{\te_j R}R Q_{\lam_j}\right)\Ra\Big]
        \end{aligned}
    \end{equation}
Using ~\eqref{eq:coefficient-diagonal}, ~\eqref{eq:coefficient-off-diagonal-1}, ~\eqref{eq:coefficient-off-diagonal-2} we obtain the estimates for the coefficient matrix of linear systems of ~\eqref{eq:b-lambda-raw}, ~\eqref{eq:b-theta-raw}. Now we estimate the forcing terms. For forcing terms in ~\eqref{eq:b-lambda-raw}, recalling ~\eqref{eq:notations-TPLTE}, ~\eqref{eq:PU-g-z},  we rearrange
    \begin{align}
        &\calL_{\calU}\ti T_{\calU}\eta+((\calU\wedge\De\calU)\cdot\ti T_{\calU}\eta) \calU+\calU\wedge\De g+g\wedge\De \calU+  g\wedge\De g\\
        =&P_{\calU}\calL_{\calU}\ti T_{\calU}(z+\eta)-((\calU\wedge\De\calU)\cdot\ti T_{\calU}z)\calU+\calL_{\calU}\left(\left(\sqrt{1-\frac{|z|^2}{4}}-1\right)\ti T_{\calU}z\right)+N_{\calU}g.
    \end{align}
    Since $P_{\calU}\calL_{\calU}\ti T_{\calU}(z+\eta) \cdot  \calU=0 $, we have
    \begin{align}
        &\La P_{\calU}\calL_{\calU}\ti T_{\calU}(z+\eta),\chi_{L,j}e^{\te_j R}\Lam Q_{\lam_j}\Ra\\
        = &\La T_{\calU}\left( P_{\calU}\calL_{\calU}\ti T_{\calU}(z+\eta)\right), T_{\calU}\left(\chi_{L,j}e^{\te_j R}\Lam Q_{\lam_j}\right)  \Ra\\
        =&\La T_{\calU}\left(\calL_{\calU}\ti T_{\calU}(z+\eta)\right), T_{\calU}\left(\chi_{L,j}e^{\te_j R}\Lam Q_{\lam_j}\right)  \Ra\\
        =&\La H_{\lam_j}(z+\eta), iT_{\calU}\left(\chi_{L,j}e^{\te_j R}\Lam Q_{\lam_j}\right)  \Ra+\La (\calH_{\calU}-H_{\lam_j})(z+\eta), iT_{\calU}\left(\chi_{L,j}e^{\te_j R}\Lam Q_{\lam_j}\right)  \Ra.
    \end{align}
    We estimate
    \begin{align}
        &\left|\La H_{\lam_j}(z+\eta), iT_{\calU}\left(\chi_{L,j}e^{\te_j R}\Lam Q_{\lam_j}\right)  \Ra\right|\\
        \lesssim & \|A_{\lam_j}(z+\eta)\|_{\rm{Mor}_2(\lam_j)}\left(\int_0^{\infty}  \frac{(\lam_j+r)^2}{\lam_j}r^4\left|A_{\lam_j}\left(iT_{\calU}\left(\chi_{L,j}e^{\te_j R}\Lam Q_{\lam_j}\right) \right)\right|^2\udr\right)^{\frac{1}{2}}.
    \end{align}
    Since
    \begin{align}
        &\left|A_{\lam_j}\left(iT_{\calU}\left(\chi_{L,j}e^{\te_j R}\Lam Q_{\lam_j}\right) \right)\right|\\
        =& \left|-\p_r(\chi_{L,j} )iT_{\calU}\left(e^{\te_j R}\Lam Q_{\lam_j}\right) +\chi_{L,j}A_{\lam_j}\left(iT_{\calU}\left(e^{\te_j R}\Lam Q_{\lam_j}\right) \right)\right|\\
        =&\left|-\p_r(\chi_{L,j} )iT_{\calU}\left(e^{\te_j R}\Lam Q_{\lam_j}\right) +i\chi_{L,j}\frac{k^2\sin^2\phi(r/\lam_j)}{r}T_{\calU}(e^{\te_j R} Q_{\lam_j})+i\chi_{L,j}\left(e^{\te_j R}\Lam Q_{\lam_j}\cdot\calU\right)(T_{\calU}(\p_r\calU))\right|\\
        \lesssim & |\p_r\chi_{L,j}|F_j+\chi_{L,j}\frac{F_j^2}{r}|T_{\calU}(\calU-(-1)^{j-1}e^{\te_j R} Q_{\lam_j})|+\chi_{L,j}F_{j}\sum_{\ell\neq j}F_{\ell} \frac{1}{r}\sum_{\ell=1}^NF_{\ell}\\
        \lesssim & |\p_r\chi_{L,j}|F_j+\chi_{L,j}\frac{F_j^2}{r}\sum_{\ell\neq j}F_{\ell}+\chi_{L,j}F_{j}\sum_{\ell\neq j}F_{\ell} \frac{1}{r}\sum_{\ell=1}^NF_{\ell}\\
        \lesssim &  |\p_r\chi_{L,j}|F_j+\chi_{L,j}\frac{F_j}{r}\sum_{\ell\neq j}F_{\ell},
    \end{align}
    we have \footnote{Note that if we use cutoff function like $\chi(r/L)$, we cannot get the coefficient $o_{L\to\infty}(1)$ when $k=3$, which is necessary for the arguments in the remaining of Section ~\ref{sec:high-Sobolev}.}
    \begin{align}
        &\int_0^{\infty}  \frac{(\lam_j+r)^2}{\lam_j}r^4\left|A_{\lam_j}\left(iT_{\calU}\left(\chi_{L,j}e^{\te_j R}\Lam Q_{\lam_j}\right) \right)\right|^2\udr\\
        \lesssim & \int_0^{\infty}  \frac{(\lam_j+r)^2}{\lam_j}r^4\left||\p_r\chi_{L,j}|F_j+\chi_{L,j}\frac{F_j}{r}\sum_{\ell\neq j}F_{\ell}\right|^2\udr\\
        \lesssim &\ \frac{\lam_j^5}{\log L}+C(L)\lam_j^5\left[\left(\frac{\lam_{j-1}}{\lam_j}\right)^{2k}+\left(\frac{\lam_{j}}{\lam_{j+1}}\right)^{2k}\right].
    \end{align}
    Therefore
    \begin{align}
        &\frac{1}{\lam_j^{\frac{5}{2}}}\left|\La H_{\lam_j}(z+\eta), iT_{\calU}\left(\chi_{L,j}e^{\te_j R}\Lam Q_{\lam_j}\right)  \Ra\right|\\
        \lesssim & \|A_{\lam_j}(z+\eta)\|_{\rm{Mor}_2(\lam_j)}\left[\frac{1}{\sqrt{\log L}}+C(L)\left[\left(\frac{\lam_{j-1}}{\lam_j}\right)^{k}+\left(\frac{\lam_{j}}{\lam_{j+1}}\right)^{k}\right]\right].
    \end{align}
    
    For $\lam_j^{-\frac{5}{2}}\La (\calH_{\calU}-H_{\lam_j})(z+\eta), iT_{\calU}\left(\chi_{L,j}e^{\te_j R}\Lam Q_{\lam_j}\right)  \Ra $, using  ~\eqref{eq:Vj-estimate-pointwise}, ~\eqref{eq:EU-estimate-pointwise}, ~\eqref{eq:correction-H3} and bootstrap assumptions, it can be bounded by $C(L)(|t|^{-2}+\|z\|_{\calH^3}^2)$.
    For terms containing $-((\calU\wedge\De\calU)\cdot\ti T_{\calU}z)\calU+\calL_{\calU}\left(\left(\sqrt{1-\frac{|z|^2}{4}}-1\right)\ti T_{\calU}z\right)+N_{\calU}g$,  using Lemma ~\ref{lem:frame-estimate},  the expansion ~\eqref{eq:U-wedge-De-U-expansion-local} and other explicit estimates on $\calU$, we obtain the bound $C(L)(|t|^{-2}+\|z\|_{\calH^3}^2)$.
    For terms containing $\lam_j'$ or $\te_j'$, we use ~\eqref{eq:lam'-rough-mod} and ~\eqref{eq:te'-rough-mod} to obtain the bound $C(L)(|t|^{-2}+\|z\|_{\calH^3}^2)$.

    The forcing terms in ~\eqref{eq:b-theta-raw} can be estimated in the same way.

    Combining the estimates on the coefficient matrix and forcing terms, we obtain  ~\eqref{eq:bj-weighted-sum}.
    
\end{proof}

\subsection{Localized Morawetz correction}

Now we define the truncated version of $\ULam=r\p_r+1$.
For $R>0$ sufficiently large, let
\begin{align}
    \psi(x)=\frac{x}{\la x \ra}-\de \frac{x}{\la x \ra^2},\ \varphi(x)=R\psi\left(\frac{x}{R}\right),\ \varphi_{\lam}(x)=\varphi\left(\frac{x}{\lam}\right).
\end{align}
Define
\begin{align}
    \Lam_{\lam}z=\varphi_{\lam}\p_rz+\frac{1}{2}\left(\p_r\varphi_{\lam}+\frac{1}{r}\varphi_{\lam}\right)z.
\end{align}

\begin{lem}
    \label{lem:Morawetz-one}
    There exist $0<\de\ll 1$, $R\gg 1$ such that, there exists $c_1>0$ only depending on $k,\de,R$ such that for complex-valued function $w$ we have
    \begin{align}
       \La \ti H_{\lam}^2w,\Lam_{\lam}w\Ra\geq c_1\|w\|_{\rm{Mor}_2(\lam)}^2.
    \end{align}
\end{lem}
\begin{proof}
    See Appendix~\ref{Appendix:Mor}.
\end{proof}

From now on we will fix the choice of $\de$ and $R$ and view them as constants.

\bigskip

Take $0<\de_0\ll1$ to be determined. For complex-valued function $w_1,w_2$, define 
\begin{align}\label{eq:def-Morawetz-functional}
    \calM_3(w_1,w_2)=\sum_{j=1}^N \de_0^{N-j}\La A_{\lam_j}\calH_{\calU}w_1 ,i\Lam_{\lam_j}A_{\lam_j}w_2 \Ra.
\end{align}

\subsection{Improving the estimate of high Sobolev norm}\label{subsec:I3}

We first define a suitable energy functional. For complex-valued function $w$, we define 

\begin{align}\label{eq:def-energy-functional-E3}
    \calE_{3}(w)&=\La \calH_{\calU}^2 w,\calH_{\calU}w \Ra+2\pi\Re\int_0^{\infty} \left(\overline{\p_r\left(\calH_{\calU}w\right)}\left(G_z-\rm{Id}\right)(\p_r\left(\calH_{\calU}w\right))\right)\udr\\
    &\quad +2\La i\calH_{\calU}T_{\calU}(\calU\wedge\De\calU),\calH_{\calU}w\Ra\\
    &=:\cE_{3,0}(w)+2\La i\calH_{\calU}T_{\calU}(\calU\wedge\De\calU),\calH_{\calU}w\Ra.
\end{align}

\begin{lem}\label{lem:E3-coercive}
    For the remainder $z$ under our bootstrap assumption, the correction $\eta$ constructed in Lemma ~\ref{lem:correction-z-eta} and $L$ sufficiently large, there exist a constant $c_L>0$ such that
    \begin{align}
        c_L\|z\|_{\calH^3}^2\leq\calE_{3,0}(z)\lesssim \|z\|_{\calH^3}^2,\label{eq:E-30-coercive}
    \end{align}
    \begin{align}
        |\La i\calH_{\calU}T_{\calU}(\calU\wedge\De\calU),\calH_{\calU}z\Ra|\lesssim |t|^{-1}\|z\|_{\calH^3},\label{eq:correction-E-3-small}
    \end{align}
    \begin{align}\label{eq:Mor-bounded}
        \left|\calM_3(z+\eta,z+\eta)\right|\lesssim_L |t|^{-1-\frac{1}{2(k-2)}-\frac{3}{2}\veps_0-\frac{1}{4}\veps_1}.
    \end{align}
\end{lem}
\begin{proof}
    ~\eqref{eq:E-30-coercive} follows from the orthogonal conditions ~\eqref{eq:g-ortho} satisfied by $g$, ~\eqref{eq:v-z} and ~\eqref{eq:HU-coercove-H3}. ~\eqref{eq:correction-E-3-small} follows from Cauchy's inequality and ~\eqref{eq:F-H1}. Using Corollary ~\ref{cor:g-improved}, Lemma ~\ref{lem:operator-estimate} , ~\eqref{eq:g-z-equivalent}, ~\eqref{eq:correction-H1}, ~\eqref{eq:correction-H3} and bootstrap assumption we obtain
    \begin{align}
        \left|\calM_3(z+\eta,z+\eta)\right|\lesssim \|z+\eta\|_{\calH^1}^{\frac{1}{2}}\|z+\eta\|_{\calH^3}^{\frac{3}{2}}\lesssim_L |t|^{-1-\frac{1}{2(k-2)}-\frac{3}{2}\veps_0-\frac{1}{4}\veps_1}
    \end{align}
    since $0<\veps_0<\frac{1}{2}$ and $T_0\ll -1$.
\end{proof}

Define
\begin{align}
    \calI_3=\calE_{3}(z)+\ti\de \calM_3(z+\eta,z+\eta),
\end{align}
where $\ti \de>0$ is to be determined.

\begin{prop}\label{prop:I3'}
    There exist  $L>0$, $\de_0>0$, $T_0<0$, $\ti \de>0$ such that under the bootstrap assumption, for $t\in[T,T_1]$, there holds
    \begin{align}
        \frac{\ud }{\ud t}\calI_3(t)\leq \frac{c_L}{8\hat C}|t|^{-2-2\veps_0},
    \end{align}
    where $\hat C$ is given in ~\eqref{eq:C-hat-def} and $c_L$ is given in ~\eqref{eq:E-30-coercive}.
\end{prop}

In the remaining of this section we will first calculate $\calI_3'(t)$ and divide the expression into several parts, then we will estimate these parts separately.

Define $y=z+\eta$, then we have
\begin{align}
    \p_t y&=-i\calH_{\calU} y-i\left(G_z-\rm{Id}\right)\calH_{\calU}z\\
    &\quad +\sum_{j=1}^N\frac{b_j^{\lam}-\lam_j\lam_j'}{\lam_j} T_{\calU}\left(\p_{\lam_j}\calU\right)+\sum_{j=1}^N\frac{b_j^{\te}-\lam_j^2\te_j'}{\lam_j^2} T_{\calU}\left(\p_{\te_j}\calU\right)\\
    &\quad +\cR_{\rm{NL}}+\cR_{\rm{mod}}+\p_t\eta.
\end{align}
Denote by
\begin{align}
    \calR_{\rm{q}}&= -i\left(G_z-\rm{Id}\right)\calH_{\calU}z,\\
    \calR_{\rm{MOD}}&=\sum_{j=1}^N\frac{b_j^{\lam}-\lam_j\lam_j'}{\lam_j} T_{\calU}\left(\p_{\lam_j}\calU\right)+\sum_{j=1}^N\frac{b_j^{\te}-\lam_j^2\te_j'}{\lam_j^2} T_{\calU}\left(\p_{\te_j}\calU\right)+\cR_{\rm{mod}}
\end{align}
where the subscript `q' stands for `quasi-linear', then we write the evolution satisfied by $y$ and $z$ in a cleaner form:
\begin{align}
    \p_t y&=-i\calH_{\calU} y+\calR_{\rm{q}}+\cR_{\rm{NL}}
     +\calR_{\rm{MOD}}+\p_t\eta. \label{eq:y-evolution}
\end{align}
For $z$, recall that we have
\begin{align}
    \p_t z=-iG_z\calH_{\calU}z+ T_{\calU}\left(\calU\wedge\De\calU-\p_t\calU\right)+\cR_{\rm{NL}}+\calR_{\rm{mod}}.  \label{eq:z-evolution}
\end{align}
We calculate
\begin{align}
    \frac{\ud }{\ud t}\La \calH_{\calU}^2 z,\calH_{\calU}z \Ra=\La (\p_t\calH_{\calU})(\calH_{\calU}z),\calH_{\calU} z \Ra
     +2\La \calH_{\calU}^2 z,(\p_t\calH_{\calU})z    \Ra
     +2\La \calH_{\calU}^2 z,\calH_{\calU}(\p_t z)   \Ra.
\end{align}

\begin{align}
     &\frac{\ud }{\ud t} \Re\int_0^{\infty} \left(\overline{\p_r\left(\calH_{\calU}z\right)}\left(G_z-\rm{Id}\right)(\p_r\left(\calH_{\calU}z\right))\right)\udr\\
     =&\Re\int_0^{\infty} \left(\overline{\p_r\left(\calH_{\calU}z\right)}\left(\vD G_z(\p_t z)\right)(\p_r\left(\calH_{\calU}z\right))\right)\udr\\
     &+2\Re\int_0^{\infty} \left(\overline{\p_r\left((\p_t\calH_{\calU})z+\calH_{\calU}(\p_tz)\right)}\left(G_z-\rm{Id}\right)(\p_r\left(\calH_{\calU}z\right))\right)\udr.
\end{align}

\begin{align}
    &\frac{\ud }{\ud t}\La i\calH_{\calU}T_{\calU}(\calU\wedge\De\calU),\calH_{\calU}z\Ra\\
    =& \La i(\p_t\calH_{\calU})T_{\calU}(\calU\wedge\De\calU),\calH_{\calU}z \Ra  +\La  i\calH_{\calU} (\p_t\left(T_{\calU}(\calU\wedge\De\calU)\right)),\calH_{\calU}z\Ra\\
    & +\La  i\calH_{\calU}T_{\calU}(\calU\wedge\De\calU),(\p_t\calH_{\calU})z\Ra +\La  i\calH_{\calU}T_{\calU}(\calU\wedge\De\calU),\calH_{\calU}(\p_tz)\Ra.
\end{align}

\begin{align}
    \frac{\ud }{\ud t}\calM_3(y,y)&= \sum_{j=1}^N \de_0^{N-j}\La (\p_tA_{\lam_j})\calH_{\calU}y ,i\Lam_{\lam_j}A_{\lam_j}y \Ra 
     +\sum_{j=1}^N \de_0^{N-j}\La A_{\lam_j}(\p_t\calH_{\calU})y ,i\Lam_{\lam_j}A_{\lam_j}y \Ra\\
    &\quad +\sum_{j=1}^N \de_0^{N-j}\La A_{\lam_j}\calH_{\calU}(\p_t y) ,i\Lam_{\lam_j}A_{\lam_j}y \Ra
     +\sum_{j=1}^N \de_0^{N-j}\La A_{\lam_j}\calH_{\calU}y ,i(\p_t\Lam_{\lam_j})A_{\lam_j}y \Ra\\
    &\quad +\sum_{j=1}^N \de_0^{N-j}\La A_{\lam_j}\calH_{\calU}y ,i\Lam_{\lam_j}(\p_tA_{\lam_j})y \Ra
     +\sum_{j=1}^N \de_0^{N-j}\La A_{\lam_j}\calH_{\calU}y ,i\Lam_{\lam_j}A_{\lam_j}(\p_t y) \Ra.
\end{align}
Inserting ~\eqref{eq:y-evolution} and ~\eqref{eq:z-evolution}, we have
\begin{align}
    \frac{\ud }{\ud t}\calI_3&=\La (\p_t\calH_{\calU})(\calH_{\calU}z),\calH_{\calU} z \Ra
     +2\La \calH_{\calU}^2 z,(\p_t\calH_{\calU})z    \Ra\\
     &\quad +4\pi\Re\int_0^{\infty} \left(\overline{\p_r\left((\p_t\calH_{\calU})z\right)}\left(G_z-\rm{Id}\right)(\p_r\left(\calH_{\calU}z\right))\right)\udr\\
     & \quad +2\La i(\p_t\calH_{\calU})T_{\calU}(\calU\wedge\De\calU),\calH_{\calU}z \Ra +2\La  i\calH_{\calU} (\p_t\left(T_{\calU}(\calU\wedge\De\calU)\right)),\calH_{\calU}z\Ra\\
    &\quad +2\La  i\calH_{\calU}T_{\calU}(\calU\wedge\De\calU),(\p_t\calH_{\calU})z\Ra\\
     &\quad + \ti \de\sum_{j=1}^N \de_0^{N-j}\La (\p_tA_{\lam_j})\calH_{\calU}y ,i\Lam_{\lam_j}A_{\lam_j}y \Ra 
     +\ti \de\sum_{j=1}^N \de_0^{N-j}\La A_{\lam_j}(\p_t\calH_{\calU})y ,i\Lam_{\lam_j}A_{\lam_j}y \Ra\\
     &\quad +\ti \de\sum_{j=1}^N \de_0^{N-j}\La A_{\lam_j}\calH_{\calU}y ,i(\p_t\Lam_{\lam_j})A_{\lam_j}y \Ra+\ti \de\sum_{j=1}^N \de_0^{N-j}\La A_{\lam_j}\calH_{\calU}y ,i\Lam_{\lam_j}(\p_tA_{\lam_j})y \Ra\\
     &\quad +2\La \calH_{\calU}^2 z,\calH_{\calU}(-iG_z\calH_{\calU}z+ T_{\calU}\left(\calU\wedge\De\calU-\p_t\calU\right)+\cR_{\rm{NL}}+\calR_{\rm{mod}})   \Ra\\
     &\quad + 2\pi \Re\int_0^{\infty} \left(\overline{\p_r\left(\calH_{\calU}z\right)}\left(\vD G_z(-iG_z\calH_{\calU} z)\right)(\p_r\left(\calH_{\calU}z\right))\right)\udr\\
     &\quad + 2\pi \Re\int_0^{\infty} \left(\overline{\p_r\left(\calH_{\calU}z\right)}\left(\vD G_z( T_{\calU}\left(\calU\wedge\De\calU-\p_t\calU\right)+\cR_{\rm{NL}}+\calR_{\rm{mod}})\right)(\p_r\left(\calH_{\calU}z\right))\right)\udr\\
     &\quad +4\pi\Re\int_0^{\infty} \left(\overline{\p_r\left(\calH_{\calU}(-iG_z\calH_{\calU} z)\right)}\left(G_z-\rm{Id}\right)(\p_r\left(\calH_{\calU}z\right))\right)\udr\\
     &\quad +4\pi\Re\int_0^{\infty} \left(\overline{\p_r\left(\calH_{\calU}(T_{\calU}\left(\calU\wedge\De\calU-\p_t\calU\right)+\cR_{\rm{NL}}+\calR_{\rm{mod}})\right)}\left(G_z-\rm{Id}\right)(\p_r\left(\calH_{\calU}z\right))\right)\udr\\
     &\quad + 2\La  i\calH_{\calU}T_{\calU}(\calU\wedge\De\calU),\calH_{\calU}(-iG_z\calH_{\calU}z+ T_{\calU}\left(\calU\wedge\De\calU-\p_t\calU\right)+\cR_{\rm{NL}}+\calR_{\rm{mod}})\Ra\\
     &\quad +\ti \de\sum_{j=1}^N \de_0^{N-j}\La A_{\lam_j}\calH_{\calU}(-i\calH_{\calU} y+\calR_{\rm{q}}+\cR_{\rm{NL}}
     +\calR_{\rm{MOD}}+\p_t\eta) ,i\Lam_{\lam_j}A_{\lam_j}y \Ra\\
     &\quad +\ti \de\sum_{j=1}^N \de_0^{N-j}\La A_{\lam_j}\calH_{\calU}y ,i\Lam_{\lam_j}A_{\lam_j}(-i\calH_{\calU} y+\calR_{\rm{q}}+\cR_{\rm{NL}}
     +\calR_{\rm{MOD}}+\p_t\eta) \Ra.
\end{align}
We divide the identity in the right-hand side into six parts.

The first part collects the terms where the time derivative do not hit $z$ or $y$:
\begin{equation}\label{eq:A-expression}
    \begin{aligned}
        (\rm{A})&=\La (\p_t\calH_{\calU})(\calH_{\calU}z),\calH_{\calU} z \Ra
     +2\La \calH_{\calU}^2 z,(\p_t\calH_{\calU})z    \Ra\\
     &\quad +4\pi\Re\int_0^{\infty} \left(\overline{\p_r\left((\p_t\calH_{\calU})z\right)}\left(G_z-\rm{Id}\right)(\p_r\left(\calH_{\calU}z\right))\right)\udr\\
     & \quad +2\La i(\p_t\calH_{\calU})T_{\calU}(\calU\wedge\De\calU),\calH_{\calU}z \Ra +2\La  i\calH_{\calU} (\p_t\left(T_{\calU}(\calU\wedge\De\calU)\right)),\calH_{\calU}z\Ra\\
    &\quad +2\La  i\calH_{\calU}T_{\calU}(\calU\wedge\De\calU),(\p_t\calH_{\calU})z\Ra\\
     &\quad + \ti \de\sum_{j=1}^N \de_0^{N-j}\La (\p_tA_{\lam_j})\calH_{\calU}y ,i\Lam_{\lam_j}A_{\lam_j}y \Ra 
     +\ti \de\sum_{j=1}^N \de_0^{N-j}\La A_{\lam_j}(\p_t\calH_{\calU})y ,i\Lam_{\lam_j}A_{\lam_j}y \Ra\\
     &\quad +\ti \de\sum_{j=1}^N \de_0^{N-j}\La A_{\lam_j}\calH_{\calU}y ,i(\p_t\Lam_{\lam_j})A_{\lam_j}y \Ra+\ti \de\sum_{j=1}^N \de_0^{N-j}\La A_{\lam_j}\calH_{\calU}y ,i\Lam_{\lam_j}(\p_tA_{\lam_j})y \Ra.
    \end{aligned}
\end{equation}

The second part collects the terms containing $T_{\calU}(\p_t\calU)$, $\cR_{\rm{mod}}$, $\cR_{\rm{MOD}}$ or $\p_t\eta$.
\begin{equation}\label{eq:B-expression}
    \begin{aligned}
         (\rm{B})&=2\La \calH_{\calU}^2 z,\calH_{\calU}(- T_{\calU}\left(\p_t\calU\right)+\calR_{\rm{mod}})   \Ra\\
     &\quad + 2\pi \Re\int_0^{\infty} \left(\overline{\p_r\left(\calH_{\calU}z\right)}\left(\vD G_z( -T_{\calU}\left(\p_t\calU\right)+\calR_{\rm{mod}})\right)(\p_r\left(\calH_{\calU}z\right))\right)\udr\\
     &\quad +4\pi\Re\int_0^{\infty} \left(\overline{\p_r\left(\calH_{\calU}(-T_{\calU}\left(\p_t\calU\right)+\calR_{\rm{mod}})\right)}\left(G_z-\rm{Id}\right)(\p_r\left(\calH_{\calU}z\right))\right)\udr\\
     &\quad + 2\La  i\calH_{\calU}T_{\calU}(\calU\wedge\De\calU),\calH_{\calU}(-T_{\calU}\left(\p_t\calU\right)+\calR_{\rm{mod}})\Ra\\
     &\quad +\ti \de\sum_{j=1}^N \de_0^{N-j}\La A_{\lam_j}\calH_{\calU}(\calR_{\rm{MOD}}+\p_t\eta) ,i\Lam_{\lam_j}A_{\lam_j}y \Ra\\
     &\quad +\ti \de\sum_{j=1}^N \de_0^{N-j}\La A_{\lam_j}\calH_{\calU}y ,i\Lam_{\lam_j}A_{\lam_j}(\calR_{\rm{MOD}}+\p_t\eta) \Ra.
    \end{aligned}
\end{equation}

The third part collects the terms containing $\cR_{\rm{NL}}$.

\begin{equation}\label{eq:C-expression}
    \begin{aligned}
        (\rm{C})&=2\La \calH_{\calU}^2 z,\calH_{\calU}(\cR_{\rm{NL}}
     )   \Ra
      \\
     &\quad + 2\pi \Re\int_0^{\infty} \left(\overline{\p_r\left(\calH_{\calU}z\right)}\left(\vD G_z(\cR_{\rm{NL}})\right)(\p_r\left(\calH_{\calU}z\right))\right)\udr\\
     &\quad +4\pi\Re\int_0^{\infty} \left(\overline{\p_r\left(\calH_{\calU}(\cR_{\rm{NL}}
     )\right)}\left(G_z-\rm{Id}\right)(\p_r\left(\calH_{\calU}z\right))\right)\udr\\
     &\quad +2\La  i\calH_{\calU}T_{\calU}(\calU\wedge\De\calU),\calH_{\calU}(\cR_{\rm{NL}})\Ra\\
     &\quad +\ti \de\sum_{j=1}^N \de_0^{N-j}\La A_{\lam_j}\calH_{\calU}(\calR_{\rm{NL}}) ,i\Lam_{\lam_j}A_{\lam_j}y \Ra\\
     &\quad +\ti \de\sum_{j=1}^N \de_0^{N-j}\La A_{\lam_j}\calH_{\calU}y ,i\Lam_{\lam_j}A_{\lam_j}(\calR_{\rm{NL}}
     ) \Ra.
    \end{aligned}
\end{equation}

The fourth part containing the terms containing $T_{\calU}(\calU\wedge\De\calU)$ and the related correction term.

\begin{equation}\label{eq:D-expression}
    \begin{aligned}
         (\rm{D})&=2\La \calH_{\calU}^2 z,\calH_{\calU}( T_{\calU}\left(\calU\wedge\De\calU\right))   \Ra\\
     &\quad + 2\pi \Re\int_0^{\infty} \left(\overline{\p_r\left(\calH_{\calU}z\right)}\left(\vD G_z( T_{\calU}\left(\calU\wedge\De\calU\right))\right)(\p_r\left(\calH_{\calU}z\right))\right)\udr\\
     &\quad +4\pi\Re\int_0^{\infty} \left(\overline{\p_r\left(\calH_{\calU}(T_{\calU}\left(\calU\wedge\De\calU\right))\right)}\left(G_z-\rm{Id}\right)(\p_r\left(\calH_{\calU}z\right))\right)\udr\\
     &\quad + 2\La  i\calH_{\calU}T_{\calU}(\calU\wedge\De\calU),\calH_{\calU}(-iG_z\calH_{\calU}z+T_{\calU}(\calU\wedge\De\calU))\Ra.
    \end{aligned}
\end{equation}

The fifth part containing the terms containing high derivatives which will cancel when putting together.
\begin{equation}\label{eq:E-expression}
    \begin{aligned}
        (\rm{E})&=2\La \calH_{\calU}^2 z,\calH_{\calU}(-iG_z\calH_{\calU}z)   \Ra\\
     &\quad + 2\pi \Re\int_0^{\infty} \left(\overline{\p_r\left(\calH_{\calU}z\right)}\left(\vD G_z(-iG_z\calH_{\calU} z)\right)(\p_r\left(\calH_{\calU}z\right))\right)\udr\\
     &\quad +4\pi\Re\int_0^{\infty} \left(\overline{\p_r\left(\calH_{\calU}(-iG_z\calH_{\calU} z)\right)}\left(G_z-\rm{Id}\right)(\p_r\left(\calH_{\calU}z\right))\right)\udr\\
     &\quad +\ti \de\sum_{j=1}^N \de_0^{N-j}\La A_{\lam_j}\calH_{\calU}(\calR_{\rm{q}}) ,i\Lam_{\lam_j}A_{\lam_j}y \Ra\\
     &\quad +\ti \de\sum_{j=1}^N \de_0^{N-j}\La A_{\lam_j}\calH_{\calU}y ,i\Lam_{\lam_j}A_{\lam_j}(\calR_{\rm{q}}) \Ra.
    \end{aligned}
\end{equation}

The sixth part is the dominant part that absorbs bad terms.
\begin{equation}\label{eq:F-expression}
    \begin{aligned}
         (\rm{F})&=\ti \de\sum_{j=1}^N \de_0^{N-j}\La A_{\lam_j}\calH_{\calU}(-i\calH_{\calU} y) ,i\Lam_{\lam_j}A_{\lam_j}y \Ra
      +\ti \de\sum_{j=1}^N \de_0^{N-j}\La A_{\lam_j}\calH_{\calU}y ,i\Lam_{\lam_j}A_{\lam_j}(-i\calH_{\calU} y) \Ra.
    \end{aligned}
\end{equation}

Now we have
\begin{align}
    \frac{\ud }{\ud t}\calI_3(t)=(\rm{A})+(\rm{B})+(\rm{C})+(\rm{D})+(\rm{E})+(\rm{F}),
\end{align}
we will estimate these parts in the following sections. Before that we define
\begin{align}
    \Ga(t)=\max_{1\leq j\leq N }\left\{\frac{|\lam_j'(t)|}{\lam_j(t)},|\te_j'(t)|\right\}.
\end{align}
Using rough estimates in Proposition~\ref{prop:mod}, we have 
\begin{align}\label{eq:Gamma-pointwise}
    \Ga(t)\lesssim \left(o_{L\to \infty}(1)+o_{T_0\to-\infty}(1)\right)|t|^{-\frac{1}{2}-\veps_0}.
\end{align}
 We also recall the definition of $F_j(t,r)$ in ~\eqref{eq:def-Fj}. Under the bootstrap assumption, one have
\begin{align}
    \label{eq:FkFl-pointwise}
    \sup_{0<r<\infty}r^{-2}\sum_{1\leq m< \ell\leq N}F_m(t,r)F_{\ell}(t,r)\lesssim \sum_{1\leq j< N}\frac{1}{\lam_j^2}\left(\frac{\lam_j}{\lam_{j+1}}\right)^k\lesssim |t|^{-1}.
\end{align}

It is worth pointing out that, though the bootstrap assumption only gives $\|z\|_{\calH^1}\leq |t|^{-\frac{1}{2}}$ and $\|z\|_{\calH^3}\leq|t|^{-\frac{1}{2}-\veps_0}$, we will still regards $\|z\|_{\calH^3}$ as the larger quantity , since in Section~\ref{sec:energy-estimate} we have already improved the $\calH^1$ estimate to $|t|^{-\frac{1}{2}-\frac{1}{k-2}-\frac{\veps_1}{2}}$, which is much smaller than $|t|^{-\frac{1}{2}-\veps_0}$. However we will usually use the rough bound $|t|^{-\frac{1}{2}}$ to control $\|z\|_{\calH^1}$ for writing convenience when the improved decay is not needed.

\bigskip
\noindent
\textbf{Estimates for $(\rm{A})$.}  We recall the expression of $(\rm{A})$ in ~\eqref{eq:A-expression}.
\begin{equation}\label{eq:A-all}
    \begin{aligned}
        (\rm{A})&=\La (\p_t\calH_{\calU})(\calH_{\calU}z),\calH_{\calU} z \Ra
     +2\La \calH_{\calU}^2 z,(\p_t\calH_{\calU})z    \Ra \\
     &\quad +4\pi\Re\int_0^{\infty} \left(\overline{\p_r\left((\p_t\calH_{\calU})z\right)}\left(G_z-\rm{Id}\right)(\p_r\left(\calH_{\calU}z\right))\right)\udr\\
     & \quad +2\La i(\p_t\calH_{\calU})T_{\calU}(\calU\wedge\De\calU),\calH_{\calU}z \Ra +2\La  i\calH_{\calU} (\p_t\left(T_{\calU}(\calU\wedge\De\calU)\right)),\calH_{\calU}z\Ra\\
    &\quad +2\La  i\calH_{\calU}T_{\calU}(\calU\wedge\De\calU),(\p_t\calH_{\calU})z\Ra\\
     &\quad + \ti \de\sum_{j=1}^N \de_0^{N-j}\La (\p_tA_{\lam_j})\calH_{\calU}y ,i\Lam_{\lam_j}A_{\lam_j}y \Ra 
     +\ti \de\sum_{j=1}^N \de_0^{N-j}\La A_{\lam_j}(\p_t\calH_{\calU})y ,i\Lam_{\lam_j}A_{\lam_j}y \Ra   \\
     &\quad +\ti \de\sum_{j=1}^N \de_0^{N-j}\La A_{\lam_j}\calH_{\calU}y ,i(\p_t\Lam_{\lam_j})A_{\lam_j}y \Ra+\ti \de\sum_{j=1}^N \de_0^{N-j}\La A_{\lam_j}\calH_{\calU}y ,i\Lam_{\lam_j}(\p_tA_{\lam_j})y \Ra.
    \end{aligned}
\end{equation}

We first estimate the terms in the first line of ~\eqref{eq:A-all}.
By definition we have
\begin{align}
    \p_t \calH_{\calU}=\sum_{j=1}^N\p_tV_j+\p_t\calE_{\calU}=\p_tV_{\calU}.
\end{align}
We calculate
\begin{align}
   \calH_{\calU} ((\p_t\calH_{\calU})z) &=(\p_tV_{\calU}) (\calH_{\calU}z)+[V_{\calU},\p_t V_{\calU}]z\\
   &\quad +\left(-\p_r^2\p_tV_{\calU}-\frac{1}{r}\p_r\p_tV_{\calU}\right)z-2\p_r\p_tV_{\calU}\p_rz.
\end{align}
Since $\calH_{\calU}$ is symmetric, we have
\begin{align}
    &\La (\p_t\calH_{\calU})(\calH_{\calU}z),\calH_{\calU} z \Ra
     +2\La \calH_{\calU}^2 z,(\p_t\calH_{\calU})z    \Ra\\
     =&3\La (\p_tV_{\calU})(\calH_{\calU}z),\calH_{\calU} z \Ra+2\La \calH_{\calU} z,[V_{\calU},\p_t V_{\calU}]z  \Ra\\
     &+2\La \calH_{\calU} z,\left(-\p_r^2\p_tV_{\calU}-\frac{1}{r}\p_r\p_tV_{\calU}\right)z  \Ra-4\La \calH_{\calU} z,\p_r\p_tV_{\calU}\p_rz  \Ra.
\end{align}
Inserting $\sum_{j=1}^N\p_tV_j+\p_t\calE_{\calU}=\p_tV_{\calU}$, we obtain
\begin{equation}\label{eq:A-1}
    \begin{aligned}
        &\La (\p_t\calH_{\calU})(\calH_{\calU}z),\calH_{\calU} z \Ra
     +2\La \calH_{\calU}^2 z,(\p_t\calH_{\calU})z    \Ra\\
     =&\sum_{j=1}^N\left[3\La (\p_tV_{j})(\calH_{\calU}z),\calH_{\calU} z \Ra+2\La \calH_{\calU} z,\left(-\p_r^2\p_tV_{j}-\frac{1}{r}\p_r\p_tV_{j}\right)z  \Ra-4\La \calH_{\calU} z,\p_r\p_tV_{j}\p_rz  \Ra\right]\\
     &+3\La (\p_t\calE_{\calU})(\calH_{\calU}z),\calH_{\calU} z \Ra+2\La \calH_{\calU} z,\left(-\p_r^2\p_t\calE_{\calU}-\frac{1}{r}\p_r\p_t\calE_{\calU}\right)z  \Ra-4\La \calH_{\calU} z,\p_r\p_t\calE_{\calU}\p_rz  \Ra\\
     &+2\La \calH_{\calU} z,[V_{\calU},\p_t \E_{\calU}]z  \Ra.
    \end{aligned}
\end{equation}
Using ~\eqref{eq:Vj-estimate-pointwise} and ~\eqref{eq:VU-bounded-pointwise},
we obtain
\begin{align}
    &\left|3\La (\p_tV_{j})(\calH_{\calU}z),\calH_{\calU} z \Ra+2\La \calH_{\calU} z,\left(-\p_r^2\p_tV_{j}-\frac{1}{r}\p_r\p_tV_{j}\right)z  \Ra-4\La \calH_{\calU} z,\p_r\p_tV_{j}\p_rz  \Ra\right|\\
    \lesssim& \frac{|\lam_j'|}{\lam_j}\int_0^{\infty}F_j^2(t,r)r^{-2}|z|_{-2}^2\udr
    \lesssim \frac{|\lam_j'|}{\lam_j}\int_0^{\infty}\frac{\lam_j^2}{(\lam_j+r)^2} r^{-2}|z|_{-2}^2\udr
    \lesssim  |\lam_j'|\|z\|_{\rm{Mor}(\lam_j)}^2.
\end{align}
The terms in ~\eqref{eq:A-1} containing $\calE_{\calU}$ can be estimated directly using ~\eqref{eq:HU}~\eqref{eq:FkFl-pointwise} and Lemma~\ref{lem:operator-estimate}. 
\begin{align}
    &\left| \La (\p_t\calE_{\calU})(\calH_{\calU}z),\calH_{\calU} z \Ra \right|\lesssim\Ga(t)|t|^{-1}\int_0^{\infty}\left( |\p_{r}^2 z|^2+ \frac{|\p_r z|^2}{r^2}+\frac{|z|^2}{r^4}\right)\udr\\
    \lesssim & \Ga(t)|t|^{-1}\|z\|_{\calH^1}\|z\|_{\calH^3}\lesssim |t|^{-\frac{5}{2}-\veps_0}.  
\end{align}
The terms $2\La \calH_{\calU} z,\left(-\p_r^2\p_t\calE_{\calU}-\frac{1}{r}\p_r\p_t\calE_{\calU}\right)z  \Ra-4\La \calH_{\calU} z,\p_r\p_t\calE_{\calU}\p_rz  \Ra+2\La \calH_{\calU} z,[V_{\calU},\p_t \E_{\calU}]z  \Ra$ can be estimate similarly to obtain the bound $|t|^{-\frac{5}{2}-\veps_0}$, we omit the details.\\
     Combining these estimates, we have
\begin{align}
    \left|\La (\p_t\calH_{\calU})(\calH_{\calU}z),\calH_{\calU} z \Ra
     +2\La \calH_{\calU}^2 z,(\p_t\calH_{\calU})z\Ra\right|\lesssim \sum_{j=1}^N |\lam_j'|\|z\|_{\rm{Mor}(\lam_j)}^2+|t|^{-\frac{5}{2}-\veps_0}.
\end{align}

Next we estimate the terms in the last two lines of ~\eqref{eq:A-all}. We estimate
\begin{align}
    &\left|\La (\p_tA_{\lam_j})\calH_{\calU}y ,i\Lam_{\lam_j}A_{\lam_j}y \Ra \right|\\
    \lesssim&\sum_{\ell\neq j}\left|\La (\p_tA_{\lam_j})V_{\ell}y ,i\Lam_{\lam_j}A_{\lam_j}y \Ra \right|+\left|\La (\p_tA_{\lam_j})\calE_{\calU}y ,i\Lam_{\lam_j}A_{\lam_j}y \Ra \right|
    +\left|\La (\p_tA_{\lam_j})H_{\lam_j}y ,i\Lam_{\lam_j}A_{\lam_j}y \Ra \right|
\end{align}
Using ~\eqref{eq:A-t-estimate-pointwise} and ~\eqref{eq:Lambda-estimate-pointwise}, we obtain
\begin{align}
    &\left|\La (\p_tA_{\lam_j})H_{\lam_j}y ,i\Lam_{\lam_j}A_{\lam_j}y \Ra \right|\\
    \lesssim& \frac{|\lam_j'|}{\lam_j}\int_0^{\infty} r^{-1}F_j^2(t,r)|y|_{-2}^2\frac{r}{\lam_j+r}\udr,\\
    \lesssim&\frac{|\lam_j'|}{\lam_j}\int_0^{\infty} rF_j^2(t,r)|y|_{-3}^2\frac{r}{\lam_j+r}\udr\\
    \lesssim & \lam_j|\lam_j'|\int_0^{\infty} \frac{\lam_j}{(\lam_j+r)^2}|y|_{-3}^2\udr=\lam_j|\lam_j'|\|y\|_{\rm{Mor}(\lam_j)}^2.
\end{align}
For $\ell\neq j$, we estimate using ~\eqref{eq:Vj-estimate-pointwise},
\begin{align}
    \left|\La (\p_tA_{\lam_j})V_{\ell}y ,i\Lam_{\lam_j}A_{\lam_j}y \Ra \right|
    \lesssim&\frac{|\lam_j'|}{\lam_j}\int_0^{\infty}r^{-1}F_j^2(t,r)r^{-2}F_{\ell}^2(t,r)|y|\frac{r}{\lam_j+r}|y|_{-2}\udr\\
    \lesssim & \frac{|\lam_j'|}{\lam_j}\int_0^{\infty}F_j^2(t,r)\frac{|y|}{r^3}\frac{r^2}{\lam_j+r}|y|_{-3}\udr\lesssim\lam_j|\lam_j'|\|y\|_{\rm{Mor}(\lam_j)}^2.
\end{align}
For $\calE_{\calU}$, we estimate using ~\eqref{eq:EU-estimate-pointwise}
\begin{align}
    &\left|\La (\p_tA_{\lam_j})\calE_{\calU}y ,i\Lam_{\lam_j}A_{\lam_j}y \Ra \right|\\
    \lesssim & \frac{|\lam_j'|}{\lam_j}\int_0^{\infty}r^{-1}F_j^2(t,r)r^{-2}\sum_{1\leq m<\ell\leq N}F_{m}(t,r)F_{\ell}(t,r)|y|\frac{r}{\lam_j+r}|y|_{-2}\udr\\
    \lesssim & \frac{|\lam_j'|}{\lam_j}\int_0^{\infty}rF_j^2(t,r)\frac{|y|}{r^3}\frac{r}{\lam_j+r}|y|_{-3}\udr\lesssim \lam_j|\lam_j'|\|y\|_{\rm{Mor}(\lam_j)}^2.
\end{align}
The terms $\La A_{\lam_j}(\p_t\calH_{\calU})y ,i\Lam_{\lam_j}A_{\lam_j}y \Ra$ and $ \La A_{\lam_j}\calH_{\calU}y ,i\Lam_{\lam_j}(\p_tA_{\lam_j})y \Ra$ can be estimated in the same way. We use $\calH_{\calU}=H_{\lam_j}+\sum_{\ell\neq j}V_{\ell}+\calE_{\calU}$ and $\p_t\calH_{\calU}=\sum_{\ell=1}^N\p_tV_{\ell}+\p_t\calE_{\calU}$ to decompose the terms and estimate them separately using Lemma ~\ref{lem:operator-estimate}.\\
Combining these estimates, we have
\begin{align}
    &\ti \de\sum_{j=1}^N \de_0^{N-j}\left|\La (\p_tA_{\lam_j})\calH_{\calU}y ,i\Lam_{\lam_j}A_{\lam_j}y \Ra \right|
     +\ti \de\sum_{j=1}^N \de_0^{N-j}\left|\La A_{\lam_j}(\p_t\calH_{\calU})y ,i\Lam_{\lam_j}A_{\lam_j}y \Ra\right|   \\
     &\quad +\ti \de\sum_{j=1}^N \de_0^{N-j}\left|\La A_{\lam_j}\calH_{\calU}y ,i\Lam_{\lam_j}(\p_tA_{\lam_j})y \Ra\right|\lesssim \ti \de\sum_{j=1}^N |\lam_j'|\lam_j\|y\|_{\rm{Mor}(\lam_j)}^2.
\end{align}
We estimate the term $\ti \de\sum_{j=1}^N \de_0^{N-j}\La A_{\lam_j}\calH_{\calU}y ,i(\p_t\Lam_{\lam_j})A_{\lam_j}y \Ra$  using the results in Lemma~\ref{lem:operator-estimate} to obtain
\begin{align}
    \left|\La A_{\lam_j}\calH_{\calU}y ,i(\p_t\Lam_{\lam_j})A_{\lam_j}y \Ra\right|
    \lesssim |\lam_j'|\int_0^{\infty}\frac{r^2}{(\lam_j+r)^2}|y|_{-3}^2\udr\lesssim |\lam_j'|\|y\|_{\calH^3}^2.
\end{align}
Therefore
\begin{align}
    \left|\ti \de\sum_{j=1}^N \de_0^{N-j}\La A_{\lam_j}\calH_{\calU}y ,i(\p_t\Lam_{\lam_j})A_{\lam_j}y \Ra\right|
    \lesssim \ti \de\sum_{j=1}^N \de_0^{N-j} |\lam_j'|\|y\|_{\calH^3}^2
    \lesssim\ti \de\sum_{j=1}^N \de_0^{N-j} \|y\|_{\calH^3}^2\left(\frac{|b_j^{\lam}|}{\lam_j}+\frac{|b_j^{\lam}-\lam_j\lam_j'|}{\lam_j}\right).
\end{align}
Using ~\eqref{eq:correction-H3} we have
\begin{align}
    \ti \de\sum_{j=1}^N \de_0^{N-j} \|y\|_{\calH^3}^2\frac{|b_j^{\lam}|}{\lam_j}\lesssim_L \ti \de\sum_{j=1}^N \de_0^{N-j} \|y\|_{\calH^3}^2\frac{\lam_j}{|t|}
    \lesssim \ti \de\sum_{j=1}^N \de_0^{N-j}(|t|^{-1-2\veps_0})\frac{\lam_j}{|t|}\lesssim \ti\de |t|^{-2-2\veps_0}.
\end{align}
Using ~\eqref{eq:bj-weighted-sum} we obtain
\begin{align}
    &\ti \de\sum_{j=1}^N \de_0^{N-j} \|y\|_{\calH^3}^2\frac{|b_j^{\lam}-\lam_j\lam_j'|}{\lam_j}\\
    \lesssim& \ti\de |t|^{-1-2\veps_0}\left(\sum_{j=1}^N\frac{\de_0^{N-j}}{\lam_j^5}\left(|b_j^{\lam}-\lam_j\lam_j'|^2+|b_j^{\te}-\lam_j^2\te_j'|^2\right)\right)^{\frac{1}{2}}\left(\sum_{j=1}^N\de_0^{N-j}\lam_j^3\right)^{\frac{1}{2}}\\
    \lesssim & \ti\de |t|^{-1-2\veps_0}\left(o_{L}(1)+C(L)\max_{1\leq j\leq N-1}\left(\frac{\lam_j}{\lam_{j+1}}\right)^{k}\right)\left(\sum_{j=1}^N\de_0^{N-j}\|A_{\lam_j}y\|_{\rm{Mor}_2(\lam_j)}^2\right)^{\frac{1}{2}}\\
        &+C(L)\ti\de |t|^{-1-2\veps_0}(|t|^{-2}+\|z\|_{\calH^3}^2)\\
        \lesssim & \ti \de |t|^{-2-3\veps_0}+\ti \de |t|^{-\veps_0}\left(o_{L}(1)+C(L)\max_{1\leq j\leq N-1}\left(\frac{\lam_j}{\lam_{j+1}}\right)^{2k}\right)\left(\sum_{j=1}^N\de_0^{N-j}\|A_{\lam_j}y\|_{\rm{Mor}_2(\lam_j)}^2\right)\\
        &+C(L)\ti\de |t|^{-2-4\veps_0}.
\end{align}
Then we estimate the terms in the third and fourth lines of ~\eqref{eq:A-all}.
\begin{align}
    \left|\La i(\p_t\calH_{\calU})T_{\calU}(\calU\wedge\De\calU),\calH_{\calU}z \Ra\right|
    \lesssim \Ga(t)\|T_{\calU}(\calU\wedge\De\calU)\|_{\calH^1}\|\calH_{\calU}z\|_{\calH^1}\lesssim\Ga(t) |t|^{-1}\|z\|_{\calH^3} ,
\end{align}
\begin{align}
    \left|\La  i\calH_{\calU} (\p_t\left(T_{\calU}(\calU\wedge\De\calU)\right)),\calH_{\calU}z\Ra\right|\lesssim \|\p_t\left(T_{\calU}(\calU\wedge\De\calU)\right)\|_{\calH^1}\|\calH_{\calU}z\|_{\calH^1}\lesssim \Ga(t) |t|^{-1}\|z\|_{\calH^3} ,
\end{align}
\begin{align}
    \left|\La  i\calH_{\calU}T_{\calU}(\calU\wedge\De\calU),(\p_t\calH_{\calU})z\Ra\right|
    \lesssim \|T_{\calU}(\calU\wedge\De\calU)\|_{\calH^1}\|(\p_t\calH_{\calU})z\|_{\calH^1}\lesssim \Ga(t)|t|^{-1}\|z\|_{\calH^3}.
\end{align}
For $\Ga(t)$, similarly we have
\begin{equation}\label{eq:Ga-estimate-refined}
    \begin{aligned}
         \Ga(t)\lesssim & \sum_{j=1}^N\frac{|b_j^{\lam}|+|b_j^{\lam}-\lam_j\lam_j'|}{\lam_j^2}+\sum_{j=1}^N\frac{|b_j^{\te}|+|b_j^{\te}-\lam_j^2\te_j'|}{\lam_j^2}\\
    \lesssim & \left(o_{L}(1)+C(L)\max_{1\leq j\leq N-1}\left(\frac{\lam_j}{\lam_{j+1}}\right)^{k}\right)\left(\sum_{j=1}^N\de_0^{N-j}\|A_{\lam_j}y\|_{\rm{Mor}_2(\lam_j)}^2\right)^{\frac{1}{2}}\\
    &+C(L)|t|^{-1}+C(L)(|t|^{-2}+\|z\|_{\calH^3}^2).
    \end{aligned}
\end{equation}
Then we have
\begin{equation}\label{eq:Ga-t-1-z-H3}
    \begin{aligned}
         &\left(o_{L}(1)+C(L)\max_{1\leq j\leq N-1}\left(\frac{\lam_j}{\lam_{j+1}}\right)^{k}\right)\left(\sum_{j=1}^N\de_0^{N-j}\|A_{\lam_j}y\|_{\rm{Mor}_2(\lam_j)}^2\right)^{\frac{1}{2}}|t|^{-\frac{3}{2}-\veps_0}\\
     \lesssim & |t|^{-\frac{5}{2}-2\veps_0}+|t|^{-\frac{1}{2}}\sum_{j=1}^N\de_0^{N-j}\|A_{\lam_j}y\|_{\rm{Mor}_2(\lam_j)}^2.
    \end{aligned}
\end{equation}

Last, the term in the second line of ~\eqref{eq:A-all} can be estimated directly using Lemma~\ref{lem:F-estimate} and Lemma~\ref{lem:operator-estimate}.
\begin{align}
    &\left|\int_0^{\infty} \left(\overline{\p_r\left((\p_t\calH_{\calU})z\right)}\left(G_z-\rm{Id}\right)(\p_r\left(\calH_{\calU}z\right))\right)\udr\right|\\
    \lesssim& \Ga(t) \int_0^{\infty} r^{-2}|z|^2|z|_{-1}|z|_{-3}\udr\\
    \lesssim & \Ga(t) \|z\|_{\calH^1}^2\|z\|_{\calH^3}^2\lesssim (o_{L\to\infty}(1)+o_{T_0\to-\infty}(1))|t|^{-\frac{5}{2}-3\veps_0}.
\end{align}

Combining these estimates, we obtain that there exists a constant $C_{L,A}>0$ depending on $L$ such that
\begin{equation}\label{eq:A-estimate-final}
    \begin{aligned}
         |(\rm{A})|\lesssim& \sum_{j=1}^N |\lam_j'|\|z\|_{\rm{Mor}(\lam_j)}^2+ \sum_{j=1}^N \lam_j|\lam_j'|\|y\|_{\rm{Mor}(\lam_j)}^2+ o_{T_0\to-\infty}(1)\sum_{j=1}^N\de_0^{N-j}\|A_{\lam_j}y\|_{\rm{Mor}_2(\lam_j)}^2\\
    &+C_{L,A}\left(\ti \de +o_{T_0\to-\infty}(1)\right)|t|^{-2-2\veps_0},
    \end{aligned}
\end{equation}

\bigskip
\noindent
\textbf{Estimates for $(\rm{B})$.} We recall the expression of $(\rm{B})$ in ~\eqref{eq:B-expression}.
\begin{equation}\label{eq:B-all}
    \begin{aligned}
         (\rm{B})&=2\La \calH_{\calU}^2 z,\calH_{\calU}(- T_{\calU}\left(\p_t\calU\right)+\calR_{\rm{mod}})   \Ra\\
     &\quad + 2\pi \Re\int_0^{\infty} \left(\overline{\p_r\left(\calH_{\calU}z\right)}\left(\vD G_z( -T_{\calU}\left(\p_t\calU\right)+\calR_{\rm{mod}})\right)(\p_r\left(\calH_{\calU}z\right))\right)\udr\\
     &\quad +4\pi\Re\int_0^{\infty} \left(\overline{\p_r\left(\calH_{\calU}(-T_{\calU}\left(\p_t\calU\right)+\calR_{\rm{mod}})\right)}\left(G_z-\rm{Id}\right)(\p_r\left(\calH_{\calU}z\right))\right)\udr\\
     &\quad + 2\La  i\calH_{\calU}T_{\calU}(\calU\wedge\De\calU),\calH_{\calU}(-T_{\calU}\left(\p_t\calU\right)+\calR_{\rm{mod}})\Ra\\
     &\quad +\ti \de\sum_{j=1}^N \de_0^{N-j}\La A_{\lam_j}\calH_{\calU}(\calR_{\rm{MOD}}+\p_t\eta) ,i\Lam_{\lam_j}A_{\lam_j}y \Ra\\
     &\quad +\ti \de\sum_{j=1}^N \de_0^{N-j}\La A_{\lam_j}\calH_{\calU}y ,i\Lam_{\lam_j}A_{\lam_j}(\calR_{\rm{MOD}}+\p_t\eta) \Ra.
    \end{aligned}
\end{equation}

For the terms in ~\eqref{eq:B-all} containing $\cR_{\rm{MOD}}$, by definition we have
\begin{align}\label{eq:RMOD-Rmod}
    \calR_{\rm{MOD}}=\sum_{\ell=1}^N\frac{b_{\ell}^{\lam}-\lam_{\ell}\lam_{\ell}'}{\lam_{\ell}} T_{\calU}\left(\p_{\lam_{\ell}}\calU\right)+\sum_{\ell=1}^N\frac{b_{\ell}^{\te}-\lam_{\ell}^2\te_{\ell}'}{\lam_{\ell}^2} T_{\calU}\left(\p_{\te_{\ell}}\calU\right)+\cR_{\rm{mod}}.
\end{align}
With the decomposition ~\eqref{eq:RMOD-Rmod}, the expression ~\eqref{eq:B-all} of $(\rm{B})$  can be divided into three parts: The terms containing  $\sum_{\ell=1}^N\frac{b_{\ell}^{\lam}-\lam_{\ell}\lam_{\ell}'}{\lam_{\ell}} T_{\calU}\left(\p_{\lam_{\ell}}\calU\right)+\sum_{\ell=1}^N\frac{b_{\ell}^{\te}-\lam_{\ell}^2\te_{\ell}'}{\lam_{\ell}^2} T_{\calU}\left(\p_{\te_{\ell}}\calU\right)$, the terms containing $\cR_{\rm{mod}}$ or $\p_t\eta$, and the remaining terms.

First we estimate the part of ~\eqref{eq:B-all} containing  $\sum_{\ell=1}^N\frac{b_{\ell}^{\lam}-\lam_{\ell}\lam_{\ell}'}{\lam_{\ell}} T_{\calU}\left(\p_{\lam_{\ell}}\calU\right)+\sum_{\ell=1}^N\frac{b_{\ell}^{\te}-\lam_{\ell}^2\te_{\ell}'}{\lam_{\ell}^2} T_{\calU}\left(\p_{\te_{\ell}}\calU\right)$, that is,
\begin{align}
    &\ti \de\sum_{j=1}^N \de_0^{N-j}\La A_{\lam_j}\calH_{\calU}\left[\sum_{\ell=1}^N\frac{b_{\ell}^{\lam}-\lam_{\ell}\lam_{\ell}'}{\lam_{\ell}} T_{\calU}\left(\p_{\lam_{\ell}}\calU\right)+\sum_{\ell=1}^N\frac{b_{\ell}^{\te}-\lam_{\ell}^2\te_{\ell}'}{\lam_{\ell}^2} T_{\calU}\left(\p_{\te_{\ell}}\calU\right)\right],i\Lam_{\lam_j}A_{\lam_j}y \Ra\\
     &\quad +\ti \de\sum_{j=1}^N \de_0^{N-j}\La A_{\lam_j}\calH_{\calU}y ,i\Lam_{\lam_j}A_{\lam_j}\left[\sum_{\ell=1}^N\frac{b_{\ell}^{\lam}-\lam_{\ell}\lam_{\ell}'}{\lam_{\ell}} T_{\calU}\left(\p_{\lam_{\ell}}\calU\right)+\sum_{\ell=1}^N\frac{b_{\ell}^{\te}-\lam_{\ell}^2\te_{\ell}'}{\lam_{\ell}^2} T_{\calU}\left(\p_{\te_{\ell}}\calU\right)\right] \Ra.
\end{align}
We first estimate the coefficients of $\frac{b_j^{\lam}-\lam_j\lam_j'}{\lam_j}$ using Lemma~\ref{lem:F-tangent-estimate} as follows:
\begin{align}
    &\left|\La A_{\lam_j}\calH_{\calU}T_{\calU}\left(\p_{\lam_j}\calU\right),i\Lam_{\lam_j}A_{\lam_j}y \Ra
      \right|\\
      \lesssim & \lam_j\int_0^{\infty} r^{-3}\lam_j^{-2}F_j(t,r)\frac{r}{\lam_j+r}\left(|\p_r (A_{\lam_j}y)|+\frac{|A_{\lam_j}y|}{r}\right)\udr\\
      \lesssim&\int_0^{\infty}\frac{F_j(t,r)}{(\lam_j+r)\lam_jr}  |A_{\lam_j}y|_{-2}\udr\\
      \lesssim&    \left(\int_0^{\infty}\frac{F_j^2(t,r)}{(\lam_j+r)^2\lam_j^2r^2}\frac{(\lam_j+r)^2}{\lam_j}  \udr  \right)^{\frac{1}{2}}\|A_{\lam_j}y\|_{\rm{Mor}_2(\lam_j)} \\
      \lesssim &\lam_j^{-\frac{3}{2}}\|A_{\lam_j}y\|_{\rm{Mor}_2(\lam_j)}.
\end{align}
For $\La A_{\lam_j}\calH_{\calU}y ,i\Lam_{\lam_j}A_{\lam_j}(T_{\calU}\left(\p_{\lam_j}\calU\right))\Ra$, we write
\begin{align}
    \calH_{\calU}=H_{\lam_j}+\sum_{\ell\neq j}V_{\ell}+\calE_{\calU},
\end{align}
since $A_{\lam_j}H_{\lam_j}=\ti H_{\lam_j}A_{\lam_j}$, this term can be estimated using Lemma~\ref{lem:operator-estimate} as follows:
\begin{align}
    &\left|\La A_{\lam_j}\calH_{\calU}y ,i\Lam_{\lam_j}A_{\lam_j}(T_{\calU}\left(\p_{\lam_j}\calU\right))\Ra\right|\\
    \lesssim& \left|\La A_{\lam_j}H_{\lam_j}y ,i\Lam_{\lam_j}A_{\lam_j}(T_{\calU}\left(\p_{\lam_j}\calU\right))\Ra\right|\\
    &+\sum_{\ell\neq j} \left|\La A_{\lam_j}V_{\ell}y ,i\Lam_{\lam_j}A_{\lam_j}(T_{\calU}\left(\p_{\lam_j}\calU\right))\Ra\right|+\left|\La A_{\lam_j}\calE_{\calU}y ,i\Lam_{\lam_j}A_{\lam_j}(T_{\calU}\left(\p_{\lam_j}\calU\right))\Ra\right|\\
    \lesssim & \frac{1}{\lam_j}\int_0^{\infty}\frac{F_j}{r(\lam_j+r)}|\ti H_{\lam_j}A_{\lam_j}y|\udr \\
    &+\frac{1}{\lam_j}\int_0^{\infty}\frac{F_j}{r(\lam_j+r)}\frac{1}{r^2}\left(\sum_{\ell\neq j}F_{\ell}^2+\sum_{1\leq m<n\leq N}F_{m}F_n\right)|y|_{-1}\udr\\
    \lesssim & \frac{1}{\lam_j^{\frac{3}{2}}}\left(\int_0^{\infty}\frac{F_j^2}{r^2}\udr\right)^{\frac{1}{2}}\|A_{\lam_j}y\|_{\rm{Mor}_2(\lam_j)}\\
    &+\frac{1}{\lam_j^{\frac{3}{2}}}\left(\sum_{\ell\neq j}\left(\int_0^{\infty}\frac{F_j^2F_{\ell}^4}{r^2}\udr\right)^{\frac{1}{2}}+\sum_{1\leq m<n\leq N}\left(\int_0^{\infty}\frac{F_j^2F_{m}^2F_n^2}{r^2}\udr\right)^{\frac{1}{2}}\right)\|y\|_{\rm{Mor}(\lam_j)}\\
    \lesssim & \frac{1}{\lam_j^{\frac{3}{2}}}\left( \|A_{\lam_j}y\|_{\rm{Mor}_2(\lam_j)}+\max_{1\leq \ell\leq N-1}\left(\frac{\lam_{\ell}}{\lam_{\ell+1}}\right)^k\|y\|_{\rm{Mor}(\lam_j)}\right)
\end{align} 
 
The coefficients of $\frac{b_j^{\te}-\lam_j^2\te_j'}{\lam_j^2} $ can be estimated in the same way:
\begin{align}
    &\left|\La A_{\lam_j}\calH_{\calU}T_{\calU}\left(\p_{\te_j}\calU\right),i\Lam_{\lam_j}A_{\lam_j}y \Ra
      +\La A_{\lam_j}\calH_{\calU}y ,i\Lam_{\lam_j}A_{\lam_j}(T_{\calU}\left(\p_{\te_j}\calU\right))\Ra\right|\\
      \lesssim &\lam_j^{-\frac{1}{2}}\left(\|A_{\lam_j}y\|_{\rm{Mor}_2(\lam_j)}+\max_{1\leq \ell\leq N-1}\left(\frac{\lam_{\ell}}{\lam_{\ell+1}}\right)^k\|y\|_{\rm{Mor}(\lam_j)}\right).
\end{align}
For $\ell\neq j$, the coefficients of $\frac{b_{\ell}^{\lam}-\lam_{\ell}\lam_{\ell}'}{\lam_{\ell}}$ and $\frac{b_{\ell}^{\te}-\lam_{\ell}^2\te_{\ell}'}{\lam_{\ell}^2} $  can be estimated similarly to obtain
\begin{equation}
    \begin{aligned}
        &\left|\La A_{\lam_j}\calH_{\calU}T_{\calU}\left(\p_{\lam_{\ell}}\calU\right),i\Lam_{\lam_j}A_{\lam_j}y \Ra
      \right|+\left|\La A_{\lam_j}\calH_{\calU}y ,i\Lam_{\lam_j}A_{\lam_j}(T_{\calU}\left(\p_{\lam_{\ell}}\calU\right))\Ra\right|\\
      \lesssim&\begin{cases}
          \lam_{\ell}^{-\frac{3}{2}}\left(\frac{\lam_{\ell}}{\lam_j}\right)^{\frac{1}{2}} \|y\|_{\rm{Mor}(\lam_j)} , & \ell<j,\\
          \lam_{\ell}^{-\frac{3}{2}} \|y\|_{\rm{Mor}(\lam_{\ell})}, & \ell>j.
      \end{cases}
    \end{aligned}
\end{equation}

\begin{equation}
    \begin{aligned}
        &\left|\La A_{\lam_j}\calH_{\calU}T_{\calU}\left(\p_{\te_{\ell}}\calU\right),i\Lam_{\lam_j}A_{\lam_j}y \Ra
      \right|+\left|\La A_{\lam_j}\calH_{\calU}y ,i\Lam_{\lam_j}A_{\lam_j}(T_{\calU}\left(\p_{\te_{\ell}}\calU\right))\Ra\right|\\
      \lesssim&\begin{cases}
          \lam_{\ell}^{-\frac{1}{2}}\left(\frac{\lam_{\ell}}{\lam_j}\right)^{\frac{1}{2}} \|y\|_{\rm{Mor}(\lam_j)} , & \ell<j,\\
          \lam_{\ell}^{-\frac{1}{2}}\|y\|_{\rm{Mor}(\lam_{\ell})}, & \ell>j.
      \end{cases}
    \end{aligned}
\end{equation}
Therefore using Lemma ~\ref{lem:b-lam-close} we have
\begin{align}
    &\Bigg|\ti \de\sum_{j=1}^N \de_0^{N-j}\La A_{\lam_j}\calH_{\calU}\left[\sum_{\ell=1}^N\frac{b_{\ell}^{\lam}-\lam_{\ell}\lam_{\ell}'}{\lam_{\ell}} T_{\calU}\left(\p_{\lam_{\ell}}\calU\right)+\sum_{\ell=1}^N\frac{b_{\ell}^{\te}-\lam_{\ell}^2\te_{\ell}'}{\lam_{\ell}^2} T_{\calU}\left(\p_{\te_{\ell}}\calU\right)\right],i\Lam_{\lam_j}A_{\lam_j}y \Ra\\
     &\quad +\ti \de\sum_{j=1}^N \de_0^{N-j}\La A_{\lam_j}\calH_{\calU}y ,i\Lam_{\lam_j}A_{\lam_j}\left[\sum_{\ell=1}^N\frac{b_{\ell}^{\lam}-\lam_{\ell}\lam_{\ell}'}{\lam_{\ell}} T_{\calU}\left(\p_{\lam_{\ell}}\calU\right)+\sum_{\ell=1}^N\frac{b_{\ell}^{\te}-\lam_{\ell}^2\te_{\ell}'}{\lam_{\ell}^2} T_{\calU}\left(\p_{\te_{\ell}}\calU\right)\right] \Ra \Bigg|\\
     \lesssim &\ \ti\de \sum_{j=1}^N\de_0^{N-j}\frac{1}{\lam_j^{\frac{5}{2}}}\left[|b_j^{\lam}-\lam_j\lam_j'|+|b_j^{\te}-\lam_j^2\te_j'|\right]\left(\|A_{\lam_j}y\|_{\rm{Mor}_2(\lam_j)}+\max_{1\leq \ell\leq N-1}\left(\frac{\lam_{\ell}}{\lam_{\ell+1}}\right)^k \|y\|_{\rm{Mor}(\lam_j)}\right)\\
     &+ \ti\de\sum_{1\leq\ell<j\leq N}\de_0^{N-j}\frac{1}{\lam_{\ell}^{\frac{5}{2}}}\left[|b_{\ell}^{\lam}-\lam_{\ell}\lam_{\ell}'|+|b_{\ell}^{\te}-\lam_{\ell}^2\te_{\ell}'|\right]\left(\frac{\lam_{\ell}}{\lam_j}\right)^{\frac{1}{2}}  \|y\|_{\rm{Mor}(\lam_j)} \\
     &+\ti\de\sum_{1\leq j<\ell\leq N}\de_0^{N-j}\frac{1}{\lam_{\ell}^{\frac{5}{2}}}\left[|b_{\ell}^{\lam}-\lam_{\ell}\lam_{\ell}'|+|b_{\ell}^{\te}-\lam_{\ell}^2\te_{\ell}'|\right] \|y\|_{\rm{Mor}(\lam_{\ell})} \\
     \lesssim &\ \ti \de \left(\sum_{j=1}^N\frac{\de_0^{N-j}}{\lam_j^5}\left(|b_j^{\lam}-\lam_j\lam_j'|^2+|b_j^{\te}-\lam_j^2\te_j'|^2\right)\right)^{\frac{1}{2}}\\
     &\times\left(\left(\sum_{j=1}^N\de_0^{N-j}\|A_{\lam_j}y\|_{\rm{Mor}_2(\lam_j)}^2\right)^{\frac{1}{2}}+\max_{1\leq \ell\leq N-1}\left(\frac{\lam_{\ell}}{\lam_{\ell+1}}\right)^k\left(\sum_{j=1}^N\de_0^{N-j}\|y\|_{\rm{Mor}(\lam_j)}^2\right)^{\frac{1}{2}}\right)\\
     &+  \ti\de   \left(\sum_{\ell=1}^N\frac{\de_0^{N-\ell}}{\lam_\ell^5}\left(|b_\ell^{\lam}-\lam_\ell\lam_\ell'|^2+|b_\ell^{\te}-\lam_\ell^2\te_\ell'|^2\right)\right)^{\frac{1}{2}}  \\
     &\times \left(\de_0+\de_0^{-\frac{N-1}{2}}\max_{1\leq \ell\leq N-1}\left(\frac{\lam_{\ell}}{\lam_{\ell+1}}\right)^{\frac{1}{2}}\right)\left(\sum_{j=1}^N\de_0^{N-j}\|y\|_{\rm{Mor}(\lam_j)}^2\right)^{\frac{1}{2}}\\
     \lesssim & \ti\de\left( o_{L\to\infty}(1)+  C(L,\de_0) o_{T_0\to-\infty}(1)\right)\sum_{j=1}^N\de_0^{N-j}\|A_{\lam_j}y\|_{\rm{Mor}_2(\lam_j)}^2\\
     &+\ti\de \left(o_{\de_0\to 0}(1)+C(L,\de_0)o_{T_0\to-\infty}(1)\right)\sum_{j=1}^N\de_0^{N-j}\|y\|_{\rm{Mor}(\lam_j)}^2+C(L,\de_0)|t|^{-2-3\veps_0}.
\end{align}

Next we estimate the second part of ~\eqref{eq:B-all}: the terms containing $\cR_{\rm{mod}}$ or $\p_t\eta$. We first present a decomposition of $\cR_{\rm{mod}}$. Denote by $\om(t,r)=\p_tf_1(t,r)\cdot f_2(t,r)$.
\begin{claim}\label{claim:Qmod}
    The term $\cR_{\rm{mod}}$ admits a decomposition of the form
    \begin{align}
        \cR_{\rm{mod}}=-i\om z+\calQ_{\rm{mod}}
    \end{align}
    with the following properties:
    \begin{align}\label{eq:omega-estimate-pointwise}
        |(r\p_r)^a \om(t,r)|\lesssim \Ga(t)\sum_{j=1}^NF_j(t,r),\ a=1,2,3.
    \end{align}
    \begin{align}\label{eq:Q-mod-estimate}
    \|\calQ_{\rm{mod}}\|_{\calH^1}\lesssim \Ga(t)\|z\|_{\calH^1}^2,\  \|\calQ_{\rm{mod}}\|_{\calH^3}\lesssim \Ga(t)\|z\|_{\calH^1}\|z\|_{\calH^3}.
    \end{align}
\end{claim}

\begin{proof}[Proof of Claim ~\ref{claim:Qmod}]
    
By definition, we have
\begin{align}
    g=\sqrt{1-\frac{|z|^2}{4}}(z_1f_1+z_2f_2)-\frac{|z|^2}{2}\calU,
\end{align} 
\begin{align}
    T_{\calU}(\p_t\calU)&=\p_t\calU\cdot f_1+i\p_t\calU\cdot f_2,\\
    \p_tf_1&=\om f_2-(f_1\cdot\p_t\calU)\calU,\\
    \p_tf_2&=-\om f_1-(f_2\cdot\p_t\calU)\calU.
\end{align}
Therefore
\begin{align}
     \calR_{\rm{mod}}&= -\left(\frac{1}{\sqrt{1-\frac{|z|^2}{4}}}-1\right)T_{\calU}\left(\p_t \calU\right)-\frac{1}{4}\frac{z}{\left(1-\frac{|z|^2}{4}\right)}\sqrt{1-\frac{|z|^2}{4}}(z_1f_1+z_2f_2)\cdot \left(\p_t \calU\right)\\
    &\quad +\frac{1}{\sqrt{1-\frac{|z|^2}{4}}}\left[\om \sqrt{1-\frac{|z|^2}{4}}z_2+\frac{|z|^2}{2}(f_1\cdot\p_t\calU)-i\om \sqrt{1-\frac{|z|^2}{4}}z_1+i\frac{|z|^2}{2}(f_2\cdot\p_t\calU)\right]\\
    &=-i\om z+\left(1-\frac{1}{\sqrt{1-\frac{|z|^2}{4}}}+\frac{|z|^2}{2\sqrt{1-\frac{|z|^2}{4}}}\right)T_{\calU}\left(\p_t \calU\right)-\frac{1}{4}\frac{z}{\sqrt{1-\frac{|z|^2}{4}}}\Re(\bar z\cdot T_{\calU}\left(\p_t \calU\right)).
\end{align}
Denote by
\begin{align}
    \calQ_{\rm{mod}}=\left(1-\frac{1}{\sqrt{1-\frac{|z|^2}{4}}}+\frac{|z|^2}{2\sqrt{1-\frac{|z|^2}{4}}}\right)T_{\calU}\left(\p_t \calU\right)-\frac{1}{4}\frac{z}{\sqrt{1-\frac{|z|^2}{4}}}\Re(\bar z\cdot T_{\calU}\left(\p_t \calU\right)).
\end{align}
For $\om$ using Lemma ~\ref{lem:frame-estimate} we obtain ~\eqref{eq:omega-estimate-pointwise}. Using Lemma~\ref{lem:frame-estimate} we obtain ~\eqref{eq:Q-mod-estimate}.

\end{proof}

Now we estimate the terms in the second part of ~\eqref{eq:B-all} containing  $-i\om z $.
\begin{align}
    \left|2\La\calH_{\calU}^2z, \calH_{\calU}(-i\om z)\Ra\right|
    \lesssim\left|\La\calH_{\calU}^2z, i\om\calH_{\calU}z\Ra  \right|+\left|\La\calH_{\calU}^2z, i[\calH_{\calU},\om]z\Ra  \right|+\left|\La\calH_{\calU}^2z, [\calH_{\calU},i](\om z)\Ra\right|.
\end{align}
For $\left|\La\calH_{\calU}^2z, i\om\calH_{\calU}z\Ra  \right|$, note that
\begin{align}
   \La\calH_{\calU}^2z, i\om\calH_{\calU}z\Ra &=\La \calH_{\calU}z,i\om\calH^2_{\calU}z     \Ra+\La \calH_{\calU}z,i[\calH_{\calU},\om]\calH_{\calU}z  \Ra+\La \calH_{\calU}z,[\calH_{\calU},i](\om\calH_{\calU}z)\Ra,
\end{align}
hence
\begin{align}
    \La\calH_{\calU}^2z, i\om\calH_{\calU}z\Ra&=\frac{1}{2}\left[\La \calH_{\calU}z,i[\calH_{\calU},\om]\calH_{\calU}z  \Ra+\La \calH_{\calU}z,[\calH_{\calU},i](\om\calH_{\calU}z)\Ra\right]\\
    &=\frac{1}{2}\Re\int_0^{\infty}\overline{i\calH_{\calU}z}\left[2(\p_r\om)\p_r+\left(\p_r^2\om+\frac{1}{r}\p_r\om\right)\Id\right]\left(\calH_{\calU}z\right)\udr\\
    &\quad +\frac{1}{2}\La \calH_{\calU}z,[\calE_{\calU},i](\om\calH_{\calU}z)\Ra\\
    &=\Re\int_0^{\infty}\overline{i\calH_{\calU}z}(\p_r\om)\p_r\left(\calH_{\calU}z\right)\udr
    +\La \calH_{\calU}z,\calE_{\calU}(i\om\calH_{\calU}z)\Ra.
\end{align}
We estimate using ~\eqref{eq:omega-estimate-pointwise}
\begin{align}
    \left|\int_0^{\infty}\overline{i\calH_{\calU}z}(\p_r\om)\p_r\left(\calH_{\calU}z\right)\udr\right|
    \lesssim  \Ga(t)\sum_{j=1}^N\int_0^{\infty}|z|_{-3}^2F_{j}(t,r)\udr
    \lesssim \Ga(t)\sum_{j=1}^N\lam_j\|z\|_{\rm{Mor}(\lam_j)}^2.
\end{align}
We estimate using ~\eqref{eq:FkFl-pointwise}
\begin{align}
    \left|\La \calH_{\calU}z,\calE_{\calU}(i\om\calH_{\calU}z)\Ra\right|&\lesssim \Ga(t)\sum_{1\leq m<\ell\leq N}\int_0^{\infty}r^{-2}F_m(t,r)F_{\ell}(t,r)|\calH_{\calU}z|^2\udr\\
    &\lesssim \Ga(t)|t|^{-1}\|z\|_{\calH^1}\|z\|_{\calH^3}\lesssim |t|^{-\frac{5}{2}-\veps_0}.
\end{align}
For $\La\calH_{\calU}^2z, i[\calH_{\calU},\om]z\Ra $, we estimate
\begin{align}
    &\left|\La\calH_{\calU}^2z, i[\calH_{\calU},\om]z\Ra \right|\\
    \lesssim & \int_0^{\infty}|\p_r(\calH_{\calU}z)||\p_r([\calH_{\calU},\om]z)|\udr
    +\int_0^{\infty}|\calH_{\calU} z||V_{\calU}( i[\calH_{\calU},\om]z)|\udr\\
    \lesssim & \Ga(t)\sum_{j=1}^N\int_0^{\infty} F_j(t,r)|z|_{-3}^2\udr\lesssim \Ga(t)\sum_{j=1}^N\lam_j\|z\|_{\rm{Mor}(\lam_j)}^2.
\end{align}
For $\left|\La\calH_{\calU}^2z, [\calH_{\calU},i](\om z)\Ra\right|$, we use ~\eqref{eq:commutator-pointwise} to obtain
\begin{align}
    &\left|\La\calH_{\calU}^2z, [\calH_{\calU},i](\om z)\Ra\right|\\
    \lesssim & \|\calH_{\calU}z\|_{\calH^1}\|[\calE_{\calU},i](\om z)\|_{\calH^1}\\
    \lesssim &\|z\|_{\calH^3}\left(\int_0^{\infty}       |\p_r\left([\calE_{\calU},i](\om z)\right) |^2+\frac{|[\calE_{\calU},i](\om z)|^2}{r^2}     \udr\right)^{\frac{1}{2}}\\
    \lesssim & \|z\|_{\calH^3}\Ga(t)|t|^{-1}\|z\|_{\calH^1}\lesssim |t|^{-\frac{5}{2}-2\veps_0}.
\end{align}
We use ~\eqref{eq:DGz-bounded} and ~\eqref{eq:Kz-bounded} to estimate
\begin{align}
    &\left|\int_0^{\infty} \left(\overline{\p_r\left(\calH_{\calU}z\right)}\left(\vD G_z(-i\om z)\right)(\p_r\left(\calH_{\calU}z\right))\right)\udr\right|\\
    \lesssim & \int_0^{\infty}|z|_{-3}^2|z|^2|\om|\udr\lesssim \Ga(t)\|z\|_{\calH^1}^2\|z\|_{\calH^3}^2\lesssim |t|^{-\frac{5}{2}-3\veps_0}.
\end{align}
\begin{align}
    \left|\int_0^{\infty} \left(\overline{\p_r\left(\calH_{\calU}(-i\om z)\right)}\left(G_z-\rm{Id}\right)(\p_r\left(\calH_{\calU}z\right))\right)\udr\right|\lesssim \Ga(t)\|z\|_{\calH^1}^2\|z\|_{\calH^3}^2\lesssim |t|^{-\frac{5}{2}-3\veps_0}.
\end{align}
We use ~\eqref{eq:F-H1} to obtain
\begin{align}
    &\left|\La  i\calH_{\calU}T_{\calU}(\calU\wedge\De\calU),\calH_{\calU}(-i\om z)\Ra\right|\\\lesssim& \|T_{\calU}(\calU\wedge\De\calU)\|_{\calH^1}\| \calH_{\calU}(-i\om z)\|_{\calH^1}
    \lesssim  |t|^{-1}\Ga(t)\|z\|_{\calH^3},
\end{align}
which has been treated when estimating $(\rm{A})$, see ~\eqref{eq:Ga-estimate-refined}.

To estimate
\begin{align}
    &\quad \ti \de\sum_{j=1}^N \de_0^{N-j}\La A_{\lam_j}\calH_{\calU}(-i\om z) ,i\Lam_{\lam_j}A_{\lam_j}y \Ra
    +\ti \de\sum_{j=1}^N \de_0^{N-j}\La A_{\lam_j}\calH_{\calU}y ,i\Lam_{\lam_j}A_{\lam_j}(-i\om z) \Ra,
\end{align}
we write $-i\om z=-i\om y+i\om \eta$ to decompose the term into two parts.
Note that
\begin{align}
    [A_{\lam_j},\om]=-\p_r \om,\ [\Lam_{\lam_j},\om]=\varphi_{\lam_j}\p_r \om,
\end{align}
we have
\begin{align}
    &\La A_{\lam_j}\calH_{\calU}(-i\om z) ,i\Lam_{\lam_j}A_{\lam_j}y \Ra+\La A_{\lam_j}\calH_{\calU}y ,i\Lam_{\lam_j}A_{\lam_j}(-i\om z) \Ra\\
    =&\La  i(\p_r\om)\calH_{\calU}y,i\Lam_{\lam_j}A_{\lam_j}y\Ra+\La iA_{\lam_j}\left[2(\p_r\om)\p_r y+(\p_r^2\om+r^{-1}\p_r\om)y\right],i\Lam_{\lam_j}A_{\lam_j}y  \Ra\\
    &-\La A_{\lam_j}[\calH_{\calU},i](\om y),i\Lam_{\lam_j}A_{\lam_j}y \Ra+\La  A_{\lam_j}\calH_{\calU}y , \varphi_{\lam_j}(\p_r \om) A_{\lam_j}y\Ra-\La  A_{\lam_j}\calH_{\calU}y , \Lam_{\lam_j}((\p_r\om)y)\Ra\\
    &+\La A_{\lam_j}\calH_{\calU}(i\om\eta),i\Lam_{\lam_j}A_{\lam_j}y\Ra+\La A_{\lam_j}\calH_{\calU}y,i\Lam_{\lam_j}A_{\lam_j}(i\om \eta)\Ra.
\end{align}
Each of the term can be estimated in the same way when estimating $\La\calH_{\calU}^2z, \calH_{\calU}(-i\om z)\Ra$ to obtain
\begin{align}
    \left|\ti \de\sum_{j=1}^N \de_0^{N-j}\La A_{\lam_j}\calH_{\calU}(-i\om z) ,i\Lam_{\lam_j}A_{\lam_j}y \Ra
    +\ti \de\sum_{j=1}^N \de_0^{N-j}\La A_{\lam_j}\calH_{\calU}y ,i\Lam_{\lam_j}A_{\lam_j}(-i\om z) \Ra\right|
    \lesssim  \Ga(t)\sum_{j=1}^N\lam_j^2\|y\|_{\rm{Mor}(\lam_j)}^2
\end{align}
and we omit the details. 

The terms containing $\calQ_{\rm{mod}}$ or $\p_t\eta$ in the second part of ~\eqref{eq:B-all} can be estimated directly using ~\eqref{eq:Q-mod-estimate}, ~\eqref{eq:correction-H1} and ~\eqref{eq:correction-H3} to obtain the bound $C(L)o_{T_{0}\to -\infty}(1)|t|^{-2-2\veps_0}$, and we only estimate one of them.
\begin{align}
    &\left|\La A_{\lam_j}\calH_{\calU}(\p_t\eta) ,i\Lam_{\lam_j}A_{\lam_j}y \Ra\right|\lesssim \|\calH_{\calU}(\p_t\eta)\|_{L^2}\| A_{\lam_j}^*\Lam_{\lam_j}A_{\lam_j}y\|_{L^2}\\
    \lesssim&\|\p_t\eta\|_{\calH^1}^{\frac{1}{2}}\|\p_t\eta\|_{\calH^3}^{\frac{1}{2}}\|y\|_{\calH^3}\lesssim |t|^{-2-2\veps_0-\frac{1}{k-2}}.
\end{align}

Last, we estimate the remaining terms in the expression of $(\rm{B})$.
Since $\calH_{\calU}h=iT_{\calU}(\calL_{\calU}\ti T_{\calU}h)$, we have
\begin{align}
    \calH_{\calU}T_{\calU}(\p_t\calU)&=iT_{\calU}(\calL_{\calU}\p_t\calU)\\
    &=iT_{\calU}(\calU\wedge\De\p_t\calU+\p_t\calU\wedge\De\calU)\\
    &=iT_{\calU}(\p_t(\calU\wedge\De\calU))\\
    &=i\p_t(T_{\calU}(\calU\wedge\De\calU))-i(\p_t T_{\calU})(\calU\wedge\De\calU),
\end{align}
then ~\eqref{eq:FkFl-pointwise}, Lemma~\ref{lem:frame-estimate} and ~\eqref{eq:Ft-H1} yields
\begin{align}
    \|\calH_{\calU}T_{\calU}(\p_t\calU)\|_{\calH^1}\lesssim \Ga(t)|t|^{-1},\ \|T_{\calU}(\p_t\calU)\|_{L^{\infty}}\lesssim\Ga(t).
\end{align}
Therefore
\begin{align}
    \left|\La \calH_{\calU}^2 z,\calH_{\calU}(- T_{\calU}\left(\p_t\calU\right))   \Ra\right|\lesssim\|z\|_{\calH^3}\|\calH_{\calU}T_{\calU}(\p_t\calU)\|_{\calH^1}\lesssim \Ga(t)|t|^{-1}\|z\|_{\calH^3},
\end{align}
\begin{align}
    &\left|\int_0^{\infty} \left(\overline{\p_r\left(\calH_{\calU}z\right)}\left(\vD G_z( -T_{\calU}\left(\p_t\calU\right))\right)(\p_r\left(\calH_{\calU}z\right))\right)\udr\right|\\
    \lesssim & \|z\|_{\calH^1}\|z\|_{\calH^3}^2\|T_{\calU}(\p_t\calU)\|_{L^{\infty}}\lesssim \Ga(t)|t|^{-1}\|z\|_{\calH^3},
\end{align}
where $\Ga(t)|t|^{-1}\|z\|_{\calH^3}$ has been estimated before, see ~\eqref{eq:Ga-estimate-refined} and ~\eqref{eq:Ga-t-1-z-H3}.
\begin{align}
    &\left|\int_0^{\infty} \left(\overline{\p_r\left(\calH_{\calU}(-T_{\calU}\left(\p_t\calU\right))\right)}\left(G_z-\rm{Id}\right)(\p_r\left(\calH_{\calU}z\right))\right)\udr\right|\\
    \lesssim & \|\calH_{\calU}T_{\calU}(\p_t\calU)\|_{\calH^1}\|z\|_{\calH^1}\|z\|_{\calH^3}\lesssim \Ga(t)|t|^{-2}\lesssim |t|^{-\frac{5}{2}-\veps_0}.
\end{align}
\begin{align}
    &\left|\La  i\calH_{\calU}T_{\calU}(\calU\wedge\De\calU),\calH_{\calU}(-T_{\calU}\left(\p_t\calU\right))\Ra\right|\\
    \lesssim&\|\calH_{\calU}T_{\calU}(\p_t\calU)\|_{\calH^1}\|T_{\calU}(\calU\wedge\De\calU)\|_{\calH^1}\lesssim\Ga(t)|t|^{-2}\lesssim |t|^{-\frac{5}{2}-\veps_0}.
\end{align}

Combining these estimates,  we obtain that there exists a constant $C_{L,B}>0$ depending on $L$ such that
\begin{equation}\label{eq:B-estimate-final}
    \begin{aligned}
         |(\rm{B})|\lesssim &\ \Ga(t)\sum_{j=1}^N \lam_j\|z\|_{\rm{Mor}(\lam_j)}^2+\left(\ti\de o_{L\to\infty}(1)+C_{L,B}C(\de_0)o_{T_0\to-\infty}(1)\right)\sum_{j=1}^N\de_0^{N-j}\|A_{\lam_j}y\|_{\rm{Mor}_2(\lam_j)}^2\\
        &+ \left(\ti\de o_{\de_0\to 0}(1)+C_{L,B}C(\de_0)o_{T_0\to-\infty}(1) \right)\sum_{j=1}^N\de_0^{N-j}\|y\|_{\rm{Mor}(\lam_j)}^2\\
        &+C_{L,B}C(\de_0)o_{T_0\to-\infty}(1)     |t|^{-2-2\veps_0}.
    \end{aligned}
\end{equation}

\bigskip
\noindent
\textbf{Estimates for $(\rm{C})$.} We first study the structure of $\cR_{\rm{NL}}$.

\begin{claim}\label{claim:RNL}
    $\calR_{\rm{NL}}$ has the following decomposition
    \begin{align}
        \calR_{\rm{NL}}=\cR_{\rm{NL,1}}+\cR_{\rm{NL},2}
    \end{align}
    with the following properties:
    $\cR_{\rm{NL},1}$ admits a decomposition $\cR_{\rm{NL},1}= \cR_{\rm{NL},1,1}+\cR_{\rm{NL},1,2}$ where
    \begin{align}
        \cR_{\rm{NL},1,1}=-2\Im\left(\overline{T_{\calU}(\p_r\calU)}z\right)\p_rz+2i\frac{k^2}{r^2}\left(\left(R\calU\right)\cdot\left(R\ti T_{\calU}z\right)\right)z,
    \end{align}
    and $\cR_{\rm{NL},1,2}$ satisfies
    \begin{align}\label{eq:R-NL12-estimate}
        |\cR_{\rm{NL},1,2}|\lesssim |z||z|_{-1}^2, \  \|\cR_{\rm{NL},1,2} \|_{\calH^1}\lesssim\|z\|_{\calH^1}^2\|z\|_{\calH^3},\ \|\cR_{\rm{NL},1,2}\|_{\calH^2}\lesssim\|z\|_{\calH^1}^{\frac{3}{2}}\|z\|_{\calH^3}^{\frac{3}{2}}.
    \end{align}
$\cR_{\rm{NL},2}$ satisfies
    \begin{align}\label{eq:R-NL2-estimate}
        \|\cR_{\rm{NL},2} \|_{\calH^1}\lesssim |t|^{-1}\|z\|_{\calH^1}^2,\ \|\cR_{\rm{NL},2} \|_{\calH^3}\lesssim |t|^{-1}\|z\|_{\calH^1}\|z\|_{\calH^3}.
    \end{align}

\end{claim}

\begin{proof}[Proof of Claim ~\ref{claim:RNL}]
    The notation $a,b$ we introduced in this proof is independent of other sections. Since $u=\calU+g$, by definition we have
    \begin{align}
    \calR_{\rm{NL}}&=iG_z\calH_{\calU}z-T_{\calU}\left(\calU\wedge \De \calU\right)
     +\frac{1}{\sqrt{1-\frac{|z|^2}{4}}}T_{\calU}\left(u\wedge \De u\right) +\frac{z}{4-|z|^2}g\cdot \left(u\wedge\De u\right).
\end{align}
    We use ~\eqref{eq:v-z} to expand $u\wedge \De u$ appearing in the last two terms. Let $a(z)=\sqrt{1-\frac{|z|^2}{4}} $, $b(z)=1-\frac{|z|^2}{2}$  for convenience.  Then
    \begin{align}
        u\wedge\De u&=\left(a\ti T_{\calU}z+b\calU \right)\wedge \left((\De a)\ti T_{\calU}z+2\p_r a\p_r(\ti T_{\calU}z)+a\De \ti T_{\calU}z+(\De b)\calU+2\p_r b\p_r(\calU)+b\De \calU    \right)\\
        &=a^2\ti T_{\calU}z \wedge\De(\ti T_{\calU}z)+2 a(\p_r a)\ti T_{\calU}z\wedge(\p_r \ti T_{\calU}z)+2a(\p_r b)\ti T_{\calU}z\wedge\p_r \calU\label{eq:u-wegde-De-u-1}\\
        &\quad +b^2\calU\wedge\De\calU+2b(\p_r b)\calU\wedge(\p_r\calU)+2b(\p_r a)\calU\wedge (\p_r \ti T_{\calU}z) \label{eq:u-wegde-De-u-2}\\
        &\quad +(a(\De b)-b(\De a))\ti T_{\calU}z\wedge\calU+ab(\ti T_{\calU}z\wedge\De\calU+\calU\wedge \De \ti T_{\calU}z).\label{eq:u-wegde-De-u-3}
    \end{align}

    Firstly, we compute $T_{\calU}\left(u\wedge \De u\right) $.
    
    We first compute the terms in ~\eqref{eq:u-wegde-De-u-1}.  Since $\ti T_{\calU}z\cdot \calU=0$, for any $k$-equivariant vector $\eta$ there holds
    \begin{align}
        T_{\calU}\left(\ti T_{\calU}z\wedge\eta\right)=-i\left(\eta\cdot \calU\right)z.
    \end{align}
    For the term containing $\ti T_{\calU}z \wedge\De(\ti T_{\calU}z)$. Since $\De\left(\ti T_{\calU}z\cdot \calU\right)=0$, we have
    \begin{align}
        T_{\calU}\left(\ti T_{\calU}z \wedge\De(\ti T_{\calU}z)\right)&=-i\left(\calU\cdot \De(\ti T_{\calU}z)\right)z\\
        &=i\left( (\De \calU)\cdot \ti T_{\calU}z+2\p_r\calU\cdot(\p_r \ti T_{\calU}z)+2\frac{k^2}{r^2}\left((R\calU)\cdot(R\ti T_{\calU}z)\right)\right)z\\
        &=i\left( \Im\left(\bar z\cdot T_{\calU}\left(\calU\wedge\De\calU\right)\right)+2\p_r\calU\cdot(\p_r \ti T_{\calU}z)+2\frac{k^2}{r^2}\left((R\calU)\cdot(R\ti T_{\calU}z)\right)\right)z.
    \end{align}
    For the term containing $\ti T_{\calU}z\wedge(\p_r \ti T_{\calU}z)$. 
    \begin{align}
        T_{\calU}\left(\ti T_{\calU}z\wedge(\p_r \ti T_{\calU}z)\right)=-i\left(\calU\cdot(\p_r \ti T_{\calU}z)\right)z=i\left((\p_r\calU)\cdot  \ti T_{\calU}z\right)z.
    \end{align}
    The term containing $\ti T_{\calU}z\wedge\p_r \calU$ cancels since $\p_r\calU\cdot\calU=0$.

    Next we compute the terms in ~\eqref{eq:u-wegde-De-u-2}. Since $T_{\calU}\left(\calU\wedge\eta\right)=iT_{\calU}\eta$, and $\p_rf_1,\p_r f_2 \parallel \calU$, we have
    \begin{align}
        T_{\calU}\left(\calU\wedge(\p_r\calU)\right)=iT_{\calU}\p_r\calU,\ T_{\calU}\left(\calU\wedge (\p_r \ti T_{\calU}z)\right)=iT_{\calU}\left((\p_r \ti T_{\calU}z)\right)=i\p_r z.
    \end{align}

    Last we compute the terms in  ~\eqref{eq:u-wegde-De-u-3}. By definition we have $T_{\calU}\left(\ti T_{\calU}z\wedge\calU\right)=-iz$.
    ~\eqref{eq:eq:HU-2} yields
    \begin{align}
        T_{\calU}\left((\ti T_{\calU}z\wedge\De\calU+\calU\wedge \De \ti T_{\calU}z)\right)=-i \calH_{\calU}z.
    \end{align}

    Combining these computations, we get
    \begin{equation}
    \begin{aligned}
         T_{\calU}\left(u\wedge \De u\right) &=a^2i\left( \Im\left(\bar z\cdot T_{\calU}\left(\calU\wedge\De\calU\right)\right)+2\p_r\calU\cdot(\p_r (\ti T_{\calU}z))+2\frac{k^2}{r^2}\left((R\calU)\cdot(R\ti T_{\calU}z)\right)\right)z\\
         &\quad +2 a(\p_r a)i\left((\p_r\calU)\cdot  \ti T_{\calU}z\right)z\\
        &\quad +b^2T_{\calU}\left(\calU\wedge\De\calU\right)+2b(\p_r b)iT_{\calU}\p_r\calU+2b(\p_r a)i\p_r z \\
        &\quad -i(a(\De b)-b(\De a))z-abi\calH_{\calU}z.
    \end{aligned}
    \end{equation}
       
    Secondly, we compute $g\cdot \left(u\wedge\De u\right)$. Since $u=\calU+g$, 
    \begin{align}
        g\cdot \left(u\wedge\De u\right)&=-\calU\cdot \left(u\wedge\De u\right)\\
        &=-a^2\calU\cdot\left(\ti T_{\calU}z \wedge\De(\ti T_{\calU}z)\right)-2 a(\p_r a)\calU\cdot \left( \ti T_{\calU}z\wedge(\p_r (\ti T_{\calU}z))\right)\\
        &\quad -2a(\p_r b)\calU\cdot\left(\ti T_{\calU}z\wedge\p_r \calU\right)-ab\calU\cdot\left(\ti T_{\calU}z\wedge\De\calU\right).
    \end{align}
    For general vector-valued function $h$,
    \begin{align}
        \calU\cdot\left( \ti T_{\calU}z\wedge h \right)&=h\cdot\left(\calU\wedge \ti T_{\calU}z\right)=\Re\left(\overline{T_{\calU}h}\cdot i z  \right)=\Im\left(\bar z\cdot T_{\calU}h\right).
    \end{align}
    Therefore
    \begin{align}
        \calU\cdot\left(\ti T_{\calU}z \wedge\De(\ti T_{\calU}z)\right)=\calU\cdot\left(\ti T_{\calU}z \wedge(\De+|\na \calU|^2)(\ti T_{\calU}z)\right)
        =-\Im\left(\bar z\cdot \calH_{\calU}z\right).
    \end{align}
    Since $\p_rf_1,\p_r f_2 \parallel \calU$, 
    \begin{align}
        \calU\cdot \left( \ti T_{\calU}z\wedge(\p_r \ti T_{\calU}z)\right)=\Im\left(\bar z\cdot T_{\calU}\left(\p_r \ti T_{\calU}z\right)\right)=\Im\left(\bar z\cdot \p_r z\right).
    \end{align}
    We also have
    \begin{align}
        \calU\cdot\left(\ti T_{\calU}z\wedge\p_r \calU\right)&=\Im\left(\bar z\cdot T_{\calU}\left(\p_r \calU\right)\right),
    \end{align}
    \begin{align}
        \calU\cdot\left(\ti T_{\calU}z\wedge\De\calU\right)&=\Im\left(\bar z\cdot T_{\calU}\left(\De \calU\right)\right)=-\Re\left(\bar z\cdot T_{\calU}\left(\calU\wedge\De \calU\right)\right).
    \end{align}
     Combining these computations, we get 
     \begin{align}
         g\cdot \left(u\wedge\De u\right)&=a^2\Im\left(\bar z\cdot \calH_{\calU}z\right)-2 a(\p_r a)\Im\left(\bar z\cdot \p_r z\right)\\
        &\quad -2a(\p_r b)\Im\left(\bar z\cdot T_{\calU}\left(\p_r \calU\right)\right)+ab\Re\left(\bar z\cdot T_{\calU}\left(\calU\wedge\De \calU\right)\right).
     \end{align}
     
     Define
    \begin{align}
        \cR_{\rm{NL,2}}=&-T_{\calU}(\calU\wedge\De \calU)+\frac{1}{\sqrt{1-\frac{|z|^2}{4}}}\left(a^2i\Im\left(\bar z\cdot T_{\calU}\left(\calU\wedge\De\calU\right)\right)z+b^2T_{\calU}(\calU\wedge\De \calU)\right)\\
        &+\frac{z}{4-|z|^2}ab\Re\left(\bar z\cdot T_{\calU}\left(\calU\wedge\De \calU\right)\right).
    \end{align}
    Since $iz\Im(\bar z w)=|z|^2 w-z\Re(\bar z w)$ for $z,w\in\C$, we get
    \begin{align}
        \cR_{\rm{NL,2}}=&\left(\frac{1}{\sqrt{1-\frac{|z|^2}{4}}}-1\right)T_{\calU}\left(\calU\wedge\De \calU\right)-\frac{6-|z|^2}{8\sqrt{1-\frac{|z|^2}{4}}}z\Re\left(\bar z\cdot T_{\calU}\left(\calU\wedge\De \calU\right)\right).
    \end{align}
    The estimates ~\eqref{eq:R-NL2-estimate} follow directly from the expression of $\cR_{\rm{NL,2}}$.
    
    Then for $\cR_{\rm{NL,1}}$ we have
    \begin{align}
        \cR_{\rm{NL},1}&=ia^2\calH_{\calU}z+i\frac{8-|z|^2}{4(4-|z|^2)}\Re(\bar z\calH_{\calU}z)z+ia\left(2\p_r\calU\cdot(\p_r (\ti T_{\calU}z))+2\frac{k^2}{r^2}\left((R\calU)\cdot(R\ti T_{\calU}z)\right)\right)z\\
        &\quad +2i(\p_r a)(\p_r\calU\cdot\ti T_{\calU}z)z+2i\frac{b\p_r b}{a}T_{\calU}\p_r\calU+2i\frac{b\p_ra}{a}\p_r z-i\left(\De b-\frac{b}{a}\De a\right)z-bi\calH_{\calU}z\\
        &\quad +\frac{z}{4-|z|^2}\left[a^2\Im(\bar z\calH_{\calU}z)-2a(\p_ra)\Im(\bar z\p_r z)-2a(\p_r b)\Im(\bar zT_{\calU}(\p_r\calU))\right].
    \end{align}
    A direct calculation using ~\eqref{eq:HU-expression}  yields that the terms containing $\p_r^2 z$ or $\p_r^2\bar z$ cancel. Therefore for $|z|$ sufficiently small,  $\cR_{\rm{NL},1}$ admits a convergent power series expansion in $z,\bar z,\p_r z,\p_r\bar z$ with coefficients depending on $\calU$ and its derivatives.  The homogeneous components of degree $0$ and $1$ vanish, while the quadratic component is
    \begin{align}
        &2i\p_r\calU\cdot \left(\p_r\left(\ti T_{\calU}z\right)\right)z+2i\frac{k^2}{r^2}\left(\left(R\calU\right)\cdot\left(R\ti T_{\calU}z\right)\right)z-i\p_r(|z|^2)T_{\calU}(\p_r\calU)\\
        =& 2i \Re\left(\overline{T_{\calU}(\p_r\calU)}\p_r z\right)z+2i\frac{k^2}{r^2}\left(\left(R\calU\right)\cdot\left(R\ti T_{\calU}z\right)\right)z-2i\Re (\bar z \p_r z)T_{\calU}(\p_r\calU)\\
        =&-2\Im\left(\overline{T_{\calU}(\p_r\calU)}z\right)\p_rz+2i\frac{k^2}{r^2}\left(\left(R\calU\right)\cdot\left(R\ti T_{\calU}z\right)\right)z.
    \end{align}
    Every cubic component has  at least one factor  $z$ or $\bar z$ with no derivatives. The estimates ~\eqref{eq:R-NL12-estimate} follow directly from the expression of $\cR_{\rm{NL},1,2}$.

\end{proof}

Now we rearrange the expression of $(\rm{C})$ in ~\eqref{eq:C-expression}. We will put the first and the third line together to obtain
\begin{equation}\label{eq:C-all}
    \begin{aligned}
        (\rm{C})&=4\pi\Re\int_0^{\infty} \left(\overline{\calH_{\calU}(\cR_{\rm{NL}}
     )}V_{\calU}(\calH_{\calU}z)\right)\udr\\
     &\quad + 2\pi \Re\int_0^{\infty} \left(\overline{\p_r\left(\calH_{\calU}z\right)}\left(\vD G_z(\cR_{\rm{NL}})\right)(\p_r\left(\calH_{\calU}z\right))\right)\udr\\
     &\quad +4\pi\Re\int_0^{\infty} \left(\overline{\p_r\left(\calH_{\calU}(\cR_{\rm{NL}}
     )\right)}G_z(\p_r\left(\calH_{\calU}z\right))\right)\udr\\
     &\quad +2\La  i\calH_{\calU}T_{\calU}(\calU\wedge\De\calU),\calH_{\calU}(\cR_{\rm{NL}})\Ra\\
     &\quad +\ti \de\sum_{j=1}^N \de_0^{N-j}\La \calH_{\calU}(\calR_{\rm{NL}}) ,iA_{\lam_j}^*\Lam_{\lam_j}A_{\lam_j}y \Ra\\
     &\quad +\ti \de\sum_{j=1}^N \de_0^{N-j}\La A_{\lam_j}\calH_{\calU}y ,i\Lam_{\lam_j}A_{\lam_j}(\calR_{\rm{NL}}
     ) \Ra.
    \end{aligned}
\end{equation}
 Inserting the decomposition  $\cR_{\rm{NL}}= \cR_{\rm{NL},1,1}+\cR_{\rm{NL},1,2}+\cR_{\rm{NL,2}}$ into ~\eqref{eq:C-all}, we denote by $(\rm{C})_{1,1}$, $(\rm{C})_{1,2}$ and $(\rm{C})_{2}$ the corresponding contribution to $(\rm{C})$, respectively.

First we estimate $(\rm{C})_{1,1}$.  One of the worst terms is $ 4\pi\Re\int_0^{\infty} \left(\overline{\p_r\left(\calH_{\calU}(\cR_{\rm{NL},1,1})\right)}G_z(\p_r\left(\calH_{\calU}z\right))\right)\udr$, in which the worst part is
\begin{align}
    4\pi\Re\int_0^{\infty} \left(\overline{2\Im\left(\overline{T_{\calU}(\p_r\calU)}z\right)\p_r^4z}G_z(-\p_r^3 z)\right)\udr.
\end{align}
Since $\Im\left(\overline{T_{\calU}(\p_r\calU)}z\right)\in \R $, we have
\begin{align}
    &4\pi\Re\int_0^{\infty} \left(\overline{2\Im\left(\overline{T_{\calU}(\p_r\calU)}z\right)\p_r^4z}G_z(-\p_r^3 z)\right)\udr\\
    =&-4\pi\Re\int_0^{\infty}\overline{\p_r^3 z} \left(\p_r\left(2\Im\left(\overline{T_{\calU}(\p_r\calU)}z\right)G_z r\right)\right)(-\p_r^3 z)\ud r\\
    &-4\pi\Re\int_0^{\infty}\overline{\p_r^3 z} \left(2\Im\left(\overline{T_{\calU}(\p_r\calU)}z\right)G_z\right)(-\p_r^4 z)\udr\\
    =&-4\pi\Re\int_0^{\infty}\overline{\p_r^3 z} \left(\p_r\left(2\Im\left(\overline{T_{\calU}(\p_r\calU)}z\right)G_zr\right)\right)(-\p_r^3 z)\ud r\\
    &-4\pi\Re\int_0^{\infty} \left(\overline{2\Im\left(\overline{T_{\calU}(\p_r\calU)}z\right)\p_r^4z}G_z(-\p_r^3 z)\right)\udr.
\end{align}
Hence using ~\eqref{eq:Kz-bounded},  Lemma ~\ref{lem:frame-estimate}  and $F_j(t,r)\lesssim\frac{\lam_j^2}{(\lam_j+r)^2}$ we obtain
\begin{align}
    &\left|4\pi\Re\int_0^{\infty} \left(\overline{2\Im\left(\overline{T_{\calU}(\p_r\calU)}z\right)\p_r^4z}G_z(-\p_r^3 z)\right)\udr\right|\\
    =&\left|-4\pi\Re\int_0^{\infty}\overline{\p_r^3 z} \left(\p_r\left(\Im\left(\overline{T_{\calU}(\p_r\calU)}z\right)G_zr\right)\right)(-\p_r^3 z)\ud r\right|\\
    \lesssim & \sum_{j=1}^N \int_0^{\infty}r^{-1}|z|_{-3}^2|z|_{-1}F_j\udr\\
    \lesssim & \sum_{j=1}^{N}\|z\|_{\calH^3}\int_0^{\infty}\frac{\lam_j^2}{(\lam_j+r)^2}|z|_{-3}^2\udr\\
    \lesssim & |t|^{-\frac{1}{2}-\veps_0}\sum_{j=1}^N\lam_j\|z\|_{\rm{Mor}(\lam_j)}^2.
\end{align}

The remaining parts in $ 4\pi\Re\int_0^{\infty} \left(\overline{\p_r\left(\calH_{\calU}(\cR_{\rm{NL},1,1})\right)}G_z(\p_r\left(\calH_{\calU}z\right))\right)\udr$ can be estimated directly using ~\eqref{eq:Kz-bounded}, Lemma ~\ref{lem:frame-estimate} and Lemma ~\ref{lem:operator-estimate} to obtain the same bound.

We estimate the remaining terms of $(\rm{C})_{1,1}$. Using Lemma ~\ref{lem:frame-estimate} and Lemma~\ref{lem:operator-estimate} we obtain
\begin{align}
    &\left|4\pi\Re\int_0^{\infty} \left(\overline{\calH_{\calU}(\cR_{\rm{NL,1,1}}
     )}V_{\calU}(\calH_{\calU}z)\right)\udr \right|
    \lesssim  \sum_{j=1}^N\int_0^{\infty} r^{-2}|z|_{-3}|z|_{-1}F_j^2 |z|_{-2}\udr
    \lesssim  |t|^{-\frac{1}{2}-\veps_0}\sum_{j=1}^N\lam_j\|z\|_{\rm{Mor}(\lam_j)}^2,
\end{align}
\begin{align}
    &\left|2\pi \Re\int_0^{\infty} \left(\overline{\p_r\left(\calH_{\calU}z\right)}\left(\vD G_z(\cR_{\rm{NL,1,1}})\right)(\p_r\left(\calH_{\calU}z\right))\right)\udr\right|
    \lesssim  \int_0^{\infty} |z|_{-3}^2|z| |z|_{-1}^2\udr\lesssim\|z\|_{\calH^1}^2\|z\|_{\calH^3}^3\lesssim|t|^{-\frac{5}{2}-3\veps_0}.
\end{align}
Using ~\eqref{eq:F-pointwise} we obtain
\begin{align}
    &\left|2\La  i\calH_{\calU}T_{\calU}(\calU\wedge\De\calU),\calH_{\calU}(\cR_{\rm{NL,1,1}})\Ra\right|\\
    \lesssim & \sum_{1\leq m<\ell\leq N}\int_0^{\infty}r^{-4}F_{m}F_{\ell}|z|_{-1}|z|_{-3}\udr\\
    \lesssim & \|z\|_{\calH^3}\sum_{1\leq m<\ell\leq N}\|z\|_{\rm{Mor}(\lam_{\ell})}\left(\int_0^{\infty}\frac{(\lam_{\ell}+r)^2}{\lam_{\ell}}r^{-6}F_{m}^2F_{\ell}^2\right)^{\frac{1}{2}}\\
    \lesssim & |t|^{-1}\|z\|_{\calH^3}\sum_{1\leq m<\ell\leq N}\lam_{\ell}^{\frac{1}{2}}\|z\|_{\rm{Mor}(\lam_{\ell})}\\
    \lesssim & |t|^{-\frac{5}{2}-\veps_0}+|t|^{-\frac{1}{2}-\veps_0}\sum_{j=1}^N\lam_{j}\|z\|_{\rm{Mor}(\lam_{j})}^2.
\end{align}
Using Lemma~\ref{lem:operator-estimate} and $rF_{\ell}\lesssim \lam_{\ell}\frac{\lam_{\ell}^2}{(\lam_{\ell}+r)^2}$ we obtain
\begin{align}
    &\left|\ti \de\sum_{j=1}^N \de_0^{N-j}\La \calH_{\calU}(\calR_{\rm{NL,1,1}}) ,iA_{\lam_j}^*\Lam_{\lam_j}A_{\lam_j}y \Ra\right|\\
    \lesssim & \ti\de \sum_{j=1}^N \de_0^{N-j}\sum_{\ell=1}^N\int_0^{\infty}F_{\ell}|z|_{-1}|z|_{-3}\frac{r}{\lam_j+r}|y|_{-3}\udr\\
    \lesssim & \ti\de \|z\|_{\calH^3}\sum_{\ell=1}^N\lam_{\ell}^2\int_0^{\infty}\frac{\lam_{\ell}}{(\lam_{\ell}+r)^2}|z|_{-3}|y|_{-3}\udr\\
    \lesssim & \ti\de \|z\|_{\calH^3}\sum_{\ell=1}^N\lam_{\ell}^2\left(\|z\|_{\rm{Mor}(\lam_{\ell})}^2+\|y\|_{\rm{Mor}(\lam_{\ell})}^2\right),
\end{align}
Similarly
\begin{align}
    \left|\ti \de\sum_{j=1}^N \de_0^{N-j}\La A_{\lam_j}\calH_{\calU}y ,i\Lam_{\lam_j}A_{\lam_j}(\calR_{\rm{NL},1,1}
     ) \Ra\right|\lesssim& \ti\de \sum_{j=1}^N \de_0^{N-j}\sum_{\ell=1}^N\int_0^{\infty}F_{\ell}|z|_{-1}|z|_{-3}\frac{r}{\lam_j+r}|y|_{-3}\udr\\
     \lesssim & \ti\de \|z\|_{\calH^3}\sum_{\ell=1}^N\lam_{\ell}^2\left(\|z\|_{\rm{Mor}(\lam_{\ell})}^2+\|y\|_{\rm{Mor}(\lam_{\ell})}^2\right),
\end{align}

Next we estimate $(\rm{C})_{1,2}$. The worst component in $ 4\pi\Re\int_0^{\infty} \left(\overline{\p_r\left(\calH_{\calU}(\cR_{\rm{NL},1,2})\right)}G_z(\p_r\left(\calH_{\calU}z\right))\right)\udr$ is
\begin{align}
    4\pi\Re\int_0^{\infty} \left(\overline{\left(\p_{\p_r z}\cR_{\rm{NL},1,2}\p_r^4 z+ \p_{\p_r \bar z}\cR_{\rm{NL},1,2}\p_r^4 \bar z \right)}G_z(\p_r^3 z)\right)\udr,
\end{align}
where we regard $\cR_{\rm{NL},1,2}$ as a function of  $(z,\bar z,\p_r z,\p_r \bar z,t,r)$. The expression of $\p_{\p_r z}\cR_{\rm{NL},1,2}$ and $\p_{\p_r \bar z}\cR_{\rm{NL},1,2}$ are calculated as follows:
\begin{align}
    \p_{\p_r z}\cR_{\rm{NL},1,2}&=-i\frac{|z|^4-9|z|^2+16}{2(4-|z|^2)^2}\bar z\p_r z+i\frac{|z|^4-10|z|^2+32}{4(4-|z|^2)^2}z\p_r\bar z\\
    &\quad +i\left(\sqrt{1-\frac{|z|^2}{4}}-1\right)\left(\overline{T_{\calU}\left(\p_r\calU\right)}z-T_{\calU}\left(\p_r\calU\right)\bar z\right),\\
    \p_{\p_r \bar z}\cR_{\rm{NL},1,2}&=i\frac{|z|^4-10|z|^2+32}{4(4-|z|^2)^2}z\p_r z+i\frac{z^2}{2(4-|z|^2)^2}z\p_r\bar z.
\end{align}
Then we have
\begin{align}
    &4\pi\Re\int_0^{\infty} \left(\overline{\left(\p_{\p_r z}\cR_{\rm{NL},1,2}\p_r^4 z+ \p_{\p_r \bar z}\cR_{\rm{NL},1,2}\p_r^4 \bar z \right)}G_z(\p_r^3 z)\right)\udr\\
    =&4\pi\Re\int_0^{\infty}\left(\left(\overline{\p_{\p_r z}\cR_{\rm{NL},1,2}}\p_r^4 \bar z+\overline{\p_{\p_r \bar z}\cR_{\rm{NL},1,2}}\p_r^4z\right)\left(1-\frac{|z|^2}{4}+\frac{(8-|z|^2)|z|^2}{8(4-|z|^2)}\right)\p_r^3 z\right)\udr\\
    &+4\pi\Re\int_0^{\infty}\left(\left(\overline{\p_{\p_r z}\cR_{\rm{NL},1,2}}\p_r^4 \bar z+\overline{\p_{\p_r \bar z}\cR_{\rm{NL},1,2}}\p_r^4z\right)\frac{(8-|z|^2)z^2}{8(4-|z|^2)}\p_r^3\bar z\right)\udr\\
    =&4\pi\Re\int_0^{\infty}\left[\overline{\p_{\p_r \bar z}\cR_{\rm{NL},1,2}} \left(1-\frac{|z|^2}{4}+\frac{(8-|z|^2)|z|^2}{8(4-|z|^2)}\right)+ \p_{\p_r z}\cR_{\rm{NL},1,2}\overline{\frac{(8-|z|^2)z^2}{8(4-|z|^2)}}\right]\p_r^4z \p_r^3 z\udr\\
    &+4\pi\Re\int_0^{\infty}\left[\overline{\p_{\p_r z}\cR_{\rm{NL},1,2}}\left(1-\frac{|z|^2}{4}+\frac{(8-|z|^2)|z|^2}{8(4-|z|^2)}\right)+\p_{\p_r \bar z}\cR_{\rm{NL},1,2}\overline{\frac{(8-|z|^2)z^2}{8(4-|z|^2)}}\right]\p_r^4 \bar z\p_r^3 z\udr
\end{align}
A direct calculation yields
\begin{align}
    &\overline{\p_{\p_r z}\cR_{\rm{NL},1,2}}\left(1-\frac{|z|^2}{4}+\frac{(8-|z|^2)|z|^2}{8(4-|z|^2)}\right)+\p_{\p_r \bar z}\cR_{\rm{NL},1,2}\overline{\frac{(8-|z|^2)z^2}{8(4-|z|^2)}}\\
    =&i\left(1-\frac{|z|^2}{4}+\frac{(8-|z|^2)|z|^2}{8(4-|z|^2)}\right)\left(\sqrt{1-\frac{|z|^2}{4}}-1\right)\left(\overline{T_{\calU}\left(\p_r\calU\right)}z-T_{\calU}\left(\p_r\calU\right)\bar z\right)\\
    &-i\left(1-\frac{|z|^2}{4}\right)\frac{|z|^4-10|z|^2+32}{4(4-|z|^2)^2}\left(\bar z\p_rz-z\p_r\bar z\right),
\end{align}
which implies
\begin{align}
    \overline{\p_{\p_r z}\cR_{\rm{NL},1,2}}\left(1-\frac{|z|^2}{4}+\frac{(8-|z|^2)|z|^2}{8(4-|z|^2)}\right)+\p_{\p_r \bar z}\cR_{\rm{NL},1,2}\overline{\frac{(8-|z|^2)z^2}{8(4-|z|^2)}}\in\R.
\end{align}
 Denote by
 \begin{align}
     \sigma_1&=\overline{\p_{\p_r \bar z}\cR_{\rm{NL},1,2}} \left(1-\frac{|z|^2}{4}+\frac{(8-|z|^2)|z|^2}{8(4-|z|^2)}\right)+ \p_{\p_r z}\cR_{\rm{NL},1,2}\overline{\frac{(8-|z|^2)z^2}{8(4-|z|^2)}},\\
     \sigma_2&=\overline{\p_{\p_r z}\cR_{\rm{NL},1,2}}\left(1-\frac{|z|^2}{4}+\frac{(8-|z|^2)|z|^2}{8(4-|z|^2)}\right)+\p_{\p_r \bar z}\cR_{\rm{NL},1,2}\overline{\frac{(8-|z|^2)z^2}{8(4-|z|^2)}}.
 \end{align}
Since $\sigma_2\in \R$, we have
\begin{align}
    \Re\int_0^{\infty}\sigma_2\p_r^4\bar z\p_r^3 z\udr&=-\Re\int_0^{\infty}|\p_r^3 z |^2 \p_r(\sigma_2 r)\ud r-\Re\int_0^{\infty}\sigma_2\p_r^3\bar z\p_r^4z\\
    &=-\Re\int_0^{\infty}|\p_r^3 z |^2 \p_r(\sigma_2 r)\ud r-\Re\int_0^{\infty}\sigma_2\p_r^3 z\p_r^4 \bar z,
\end{align}
which implies
\begin{align}
    &4\pi\Re\int_0^{\infty} \left(\overline{\left(\p_{\p_r z}\cR_{\rm{NL},1,2}\p_r^4 z+ \p_{\p_r \bar z}\cR_{\rm{NL},1,2}\p_r^4 \bar z \right)}G_z(\p_r^3 z)\right)\udr\\
    =&4\pi\Re\int_0^{\infty}\frac{1}{2}\sigma_1 \p_r((\p_r^3z)^2)\udr-2\pi\Re\int_0^{\infty}|\p_r^3 z |^2 \p_r(\sigma_2 r)\ud r\\
    =&-2\pi\Re\int_0^{\infty}\left( (\p_r(\sigma_1 r))(\p_r^3 z)^2+(\p_r(\sigma_2 r))|\p_r^3z|^2 \right)\ud r.
\end{align}
By definition we have
\begin{align}
    |\sigma_1|+|\sigma_2|\lesssim |z||z|_{-1},\ |\p_r\sigma_1|+|\p_r\sigma_2|\lesssim |z|_{-1}^2+|z||z|_{-2},
\end{align}
which implies
\begin{align}
    &\left|4\pi\Re\int_0^{\infty} \left(\overline{\left(\p_{\p_r z}\cR_{\rm{NL},1,2}\p_r^4 z+ \p_{\p_r \bar z}\cR_{\rm{NL},1,2}\p_r^4 \bar z \right)}G_z(\p_r^3 z)\right)\udr\right|\\
    \lesssim& \int_0^{\infty}|z|_{-3}^2  \left(|z||z|_{-1}+r\left(|z|_{-1}^2+|z||z|_{-2}\right)\right) \ud r\\
    \lesssim & \|z\|_{\calH^1}\|z\|_{\calH^3}^3\lesssim|t|^{-2-3\veps_0}.
\end{align}

The remaining parts in $ 4\pi\Re\int_0^{\infty} \left(\overline{\p_r\left(\calH_{\calU}(\cR_{\rm{NL},1,2})\right)}G_z(\p_r\left(\calH_{\calU}z\right))\right)\udr$ can be estimated directly using ~\eqref{eq:Kz-bounded}, Lemma ~\ref{lem:frame-estimate} and Lemma ~\ref{lem:operator-estimate} to obtain the same bound.

The term $4\pi\Re\int_0^{\infty} \left(\overline{\calH_{\calU}(\cR_{\rm{NL},1,2}
     )}V_{\calU}(\calH_{\calU}z)\right)\udr$ can be estimated using $V_{\calU}=\frac{k^2}{r^2}+\sum_{j=1}^N V_j+\calE_{\calU}$ and Lemma ~\ref{lem:operator-estimate}:
     \begin{align}
         &\left|4\pi\Re\int_0^{\infty} \left(\overline{\calH_{\calU}(\cR_{\rm{NL},1,2}
     )}V_{\calU}(\calH_{\calU}z)\right)\udr\right|\\
     \lesssim & \sum_{j=1}^N\int_0^{\infty}r^{-2}|z|_{-3}|z|_{-2}|z|_{-1}|z|F_j^2\udr+\int_0^{\infty}r^{-2}|z|_{-3}|z|_{-2}|z|_{-1}|z| \udr\\
     \lesssim & \|z\|_{\calH^1}\|z\|_{\calH^3}^3\lesssim |t|^{-2-3\veps_0}.
     \end{align}

The remaining components in $(\rm{C})_{1,2}$ can be directly estimated using ~\eqref{eq:R-NL12-estimate} and Cauchy's inequality to obtain the bound $C(L)|t|^{-2-\frac{5}{2}\veps_0}$.

Last we estimate $(\rm{C})_2$. It follows directly from ~\eqref{eq:R-NL2-estimate} and Cauchy's inequality that
\begin{align}
    |(\rm{C})_2|\lesssim_L |t|^{-\frac{5}{2}-2\veps_0}+|t|^{-\frac{7}{2}-2\veps_0}+|t|^{-\frac{5}{2}-2\veps_0}+|t|^{-3-\veps_0}+|t|^{-\frac{5}{2}-\frac{3}{2}\veps_0}+|t|^{-\frac{5}{2}-\frac{3}{2}\veps_0}\lesssim |t|^{-\frac{5}{2}-\frac{3}{2}\veps_0}.
\end{align}

Combining these estimates, we obtain that there exists a constant $C_{L,C}>0$ depending on $L$ such that
\begin{align}\label{eq:C-estimate-final}
    |(\rm{C})|&\lesssim |t|^{-\frac{1}{2}-\veps_0}\sum_{j=1}^N\lam_j\|z\|_{\rm{Mor}(\lam_j)}^2+o_{T_0\to-\infty}(1)\ti\de\sum_{j=1}^N \|y\|_{\rm{Mor}(\lam_j)}^2 +C_{L,C}o_{T_0\to-\infty}(1)|t|^{-2-2\veps_0}.
\end{align}

\bigskip
\noindent
\textbf{Estimates for $(\rm{D})$.} The last line in the expression of $(\rm{D})$ in ~\eqref{eq:D-expression} can be simplified. Recalling the definition of ~\eqref{eq:Kz-def}, we have
\begin{align}
    (\rm{D})
     &=-2\La  \calH_{\calU}T_{\calU}(\calU\wedge\De\calU),\calH_{\calU}(K_z\calH_{\calU}z)\Ra\\
     &\quad +4\pi\Re\int_0^{\infty} \left(\overline{\p_r\left(\calH_{\calU}(T_{\calU}\left(\calU\wedge\De\calU\right))\right)}K_z(\p_r\left(\calH_{\calU}z\right))\right)\udr\\
     &\quad + 2\pi \Re\int_0^{\infty} \left(\overline{\p_r\left(\calH_{\calU}z\right)}\left(\vD G_z( T_{\calU}\left(\calU\wedge\De\calU\right))\right)(\p_r\left(\calH_{\calU}z\right))\right)\udr\\
     &\quad -2\La  i\calH_{\calU}T_{\calU}(\calU\wedge\De\calU),[\calH_{\calU},i](G_z\calH_{\calU}z)\Ra.
\end{align}
Using ~\eqref{eq:HU}, we have
\begin{align}
    &\quad -2\La  \calH_{\calU}T_{\calU}(\calU\wedge\De\calU),\calH_{\calU}(K_z\calH_{\calU}z)\Ra\\
    &= -4\pi\Re\int_0^{\infty}\left(\overline{\p_r\left(\calH_{\calU}(T_{\calU}\left(\calU\wedge\De\calU\right))\right)}\left((\p_rK_z)(\calH_{\calU}z)+K_z(\p_r\left(\calH_{\calU}z\right))\right)\right)\udr\\
    &\quad -4\pi\Re\int_0^{\infty} \overline{\calH_{\calU}(T_{\calU}\left(\calU\wedge\De\calU\right)) } V_{\calU}(K_z\calH_{\calU}z)\udr.
\end{align}
Therefore
\begin{equation}\label{eq:D-all}
    \begin{aligned}
         (\rm{D})
     &=-4\pi\Re\int_0^{\infty}\left(\overline{\p_r\left(\calH_{\calU}(T_{\calU}\left(\calU\wedge\De\calU\right))\right)}\left((\p_rK_z)(\calH_{\calU}z)\right)\right)\udr\\
    &\quad -4\pi\Re\int_0^{\infty} \overline{\calH_{\calU}(T_{\calU}\left(\calU\wedge\De\calU\right)) } V_{\calU}(K_z\calH_{\calU}z)\udr\\
     &\quad + 2\pi \Re\int_0^{\infty} \left(\overline{\p_r\left(\calH_{\calU}z\right)}\left(\vD G_z( T_{\calU}\left(\calU\wedge\De\calU\right))\right)(\p_r\left(\calH_{\calU}z\right))\right)\udr\\
     &\quad -2\La  i\calH_{\calU}T_{\calU}(\calU\wedge\De\calU),[\calH_{\calU},i](G_z\calH_{\calU}z)\Ra.
    \end{aligned}
\end{equation}

These terms can be estimated separately. For the term in the first line of ~\eqref{eq:D-all}, we estimate using ~\eqref{eq:Kz-bounded}, ~\eqref{eq:FkFl-pointwise}, ~\eqref{eq:F-pointwise}, ~\eqref{eq:VU-bounded-pointwise} to obtain
\begin{align}
    &\left|-4\pi\Re\int_0^{\infty}\left(\overline{\p_r\left(\calH_{\calU}(T_{\calU}\left(\calU\wedge\De\calU\right))\right)}\left((\p_rK_z)(\calH_{\calU}z)\right)\right)\udr\right|\\
    \lesssim & \sum_{1\leq m< \ell\leq N}\int_0^{\infty} r^{-5}F_m(t,r)F_{\ell}(t,r) |z||\p_r z||z|_{-2}\udr\\
    \lesssim &|t|^{-1}\int_0^{\infty}|z|\frac{|\p_r z|}{r^2} \frac{|z|_{-2}}{r}\udr\\
    \lesssim &|t|^{-1} \|z\|_{\calH^1}\|z\|_{\calH^3}^2\lesssim |t|^{-\frac{5}{2}-2\veps_0}.
\end{align}
The term in the second line of ~\eqref{eq:D-all} can be treated similarly using ~\eqref{eq:VU-bounded-pointwise}.
\begin{align}
    &\left|4\pi\Re\int_0^{\infty} \overline{\calH_{\calU}(T_{\calU}\left(\calU\wedge\De\calU\right)) } V_{\calU}(K_z\calH_{\calU}z)\udr\right|\\
    \lesssim & |t|^{-1}\int_0^{\infty} r^{-4} |z|^2|z|_{-2}\udr\\
    \lesssim & |t|^{-1}\|z\|_{\calH^1}\|z\|_{\calH^3}^2\lesssim |t|^{-\frac{5}{2}-2\veps_0}.
\end{align}
For the term in the third line of ~\eqref{eq:D-all}, we use ~\eqref{eq:DGz-bounded}, ~\eqref{eq:FkFl-pointwise} and ~\eqref{eq:F-pointwise} to obtain
\begin{align}
    &\left|2\pi \Re\int_0^{\infty} \left(\overline{\p_r\left(\calH_{\calU}z\right)}\left(\vD G_z( T_{\calU}\left(\calU\wedge\De\calU\right))\right)(\p_r\left(\calH_{\calU}z\right))\right)\udr\right|\\
    \lesssim & |t|^{-1}\int_0^{\infty}|\p_r\left(\calH_{\calU}z\right)||z||\p_r\left(\calH_{\calU}z\right)|\udr\lesssim |t|^{-1}\|z\|_{\calH^1}\|z\|_{\calH^3}^2\lesssim |t|^{-\frac{5}{2}-2\veps_0}.
\end{align}
For the term in the last line of ~\eqref{eq:D-all}, we use ~\eqref{eq:commutator-pointwise} and the definition ~\eqref{eq:Gz-def} of $G_z$ to obtain
\begin{align}
    &\left|2\La  i\calH_{\calU}T_{\calU}(\calU\wedge\De\calU),[\calH_{\calU},i](G_z\calH_{\calU}z)\Ra\right|\\
    \lesssim & |t|^{-1}\sum_{1\leq m< \ell\leq N}\int_0^{\infty}r^{-3}F_m(t,r)F_{\ell}(t,r)r^{-1}|z|_{-2}\udr\\
    \lesssim & |t|^{-1}\sum_{1\leq m< \ell\leq N}\left(\frac{1}{\lam_m^4}\left(\frac{\lam_m}{\lam_{\ell}}\right)^{2k}\right)^{\frac{1}{2}}\|z\|_{\calH^3}\lesssim |t|^{-\frac{5}{2}-\veps_0}.
\end{align}

Combining these estimates, we obtain
\begin{align}\label{eq:D-estimate-final}
    |(\rm{D})|\lesssim |t|^{-\frac{5}{2}-\veps_0}.
\end{align}

\bigskip
\noindent
\textbf{Estimates for $(\rm{E})$.} We recall the expression of $(\rm{E})$ in ~\eqref{eq:E-expression}.
\begin{equation}\label{eq:E-all}
    \begin{aligned}
        (\rm{E})&=2\La \calH_{\calU}^2 z,\calH_{\calU}(-iG_z\calH_{\calU}z)   \Ra\\
     &\quad + 2\pi \Re\int_0^{\infty} \left(\overline{\p_r\left(\calH_{\calU}z\right)}\left(\vD G_z(-iG_z\calH_{\calU} z)\right)(\p_r\left(\calH_{\calU}z\right))\right)\udr\\
     &\quad +4\pi\Re\int_0^{\infty} \left(\overline{\p_r\left(\calH_{\calU}(-iG_z\calH_{\calU} z)\right)}\left(G_z-\rm{Id}\right)(\p_r\left(\calH_{\calU}z\right))\right)\udr\\
     &\quad +\ti \de\sum_{j=1}^N \de_0^{N-j}\La A_{\lam_j}\calH_{\calU}(\calR_{\rm{q}}) ,i\Lam_{\lam_j}A_{\lam_j}y \Ra\\
     &\quad +\ti \de\sum_{j=1}^N \de_0^{N-j}\La A_{\lam_j}\calH_{\calU}y ,i\Lam_{\lam_j}A_{\lam_j}(\calR_{\rm{q}}) \Ra.
    \end{aligned}
\end{equation}

First we estimate the terms in the first and the third line of ~\eqref{eq:E-all} together. We calculate
\begin{align}
    &2\La \calH_{\calU}^2 z,\calH_{\calU}(-iG_z\calH_{\calU}z)   \Ra\\
    =&4\pi\Re\int_0^{\infty}\overline{\p_r(\calH_{\calU}z)}\p_r\left(\calH_{\calU}(-iG_z\calH_{\calU}z)\right)\udr\\
    &+4\pi\Re\int_0^{\infty}\overline{\calH_{\calU}z}V_{\calU}\left(\calH_{\calU}(-iG_z\calH_{\calU}z)\right)\udr.
\end{align}
The expression of $G_z$ implies $K_z^*=K_z$, which yields
\begin{equation}\label{eq:E-1}
    \begin{aligned}
        &2\La \calH_{\calU}^2 z,\calH_{\calU}(-iG_z\calH_{\calU}z)   \Ra+4\pi\Re\int_0^{\infty} \left(\overline{\p_r\left(\calH_{\calU}(-iG_z\calH_{\calU} z)\right)}\left(G_z-\rm{Id}\right)(\p_r\left(\calH_{\calU}z\right))\right)\udr\\
    =&4\pi\Re\int_0^{\infty}\overline{\p_r(\calH_{\calU}z)}\p_r\left(\calH_{\calU}(-iG_z\calH_{\calU}z)\right)\udr+4\pi\Re\int_0^{\infty}\overline{\calH_{\calU}z}V_{\calU}\left(\calH_{\calU}(-iG_z\calH_{\calU}z)\right)\udr\\
    &+4\pi\Re\int_0^{\infty} \left(\overline{\p_r\left(\calH_{\calU}z\right)}\left(G_z-\rm{Id}\right)(\p_r\left(\calH_{\calU}(-iG_z\calH_{\calU} z)\right))\right)\udr\\
    =&4\pi\Re\int_0^{\infty} \left(\overline{\p_r\left(\calH_{\calU}z\right)}G_z(\p_r\left(\calH_{\calU}(-iG_z\calH_{\calU} z)\right))\right)\udr+4\pi\Re\int_0^{\infty}\overline{\calH_{\calU}z}V_{\calU}\left(\calH_{\calU}(-iG_z\calH_{\calU}z)\right)\udr.
    \end{aligned}
\end{equation}
We calculate the first term in the last line of ~\eqref{eq:E-1}, which we denote by $(\rm{E})_{*}$. We have
\begin{align}
    &G_z(\p_r\left(\calH_{\calU}(-iG_z\calH_{\calU} z)\right))\\
    =&G_z\left(\p_r\left( \p_r^2(iG_z\calH_{\calU} z) +\frac{\p_r}{r}(iG_z\calH_{\calU} z)-V_{\calU}(iG_z\calH_{\calU} z)      \right)\right)\\
    =&G_z i \left(\p_r^3(G_z\calH_{\calU} z)+\frac{\p_r^2}{r}(G_z\calH_{\calU} z)-\frac{\p_r}{r^2}(G_z\calH_{\calU} z)\right)-G_z\p_r(V_{\calU}(iG_z\calH_{\calU} z)  )\\
    =&G_ziG_z\left(\left(\p_r^3+\frac{\p_r^2}{r}-\frac{\p_r}{r^2}   \right)\left(\calH_{\calU}z\right)\right)\\
    &+G_zi(\p_rG_z)\left( \left(3\p_r^2+2\frac{\p_r}{r}-\frac{1}{r^2}\right)\left(\calH_{\calU}z\right) \right)\\
    &+G_zi(\p_r^2G_z)\left(\left(3\p_r+\frac{1}{r}\right)(\calH_{\calU}z)\right)\\
    &+G_zi(\p_r^3G_z)(\calH_{\calU}z)-G_z\p_r(V_{\calU}(iG_z\calH_{\calU} z)  )
\end{align}
A direct calculation yields that
\begin{align}\label{eq:GziGz}
     G_ziG_z=i,\ G_zi\p_rG_z=-(\p_rG_z)iG_z.
\end{align}

For the component of $(\rm{E})_*$ containing $G_ziG_z\left(\left(\p_r^3+\frac{\p_r^2}{r}-\frac{\p_r}{r^2}   \right)\left(\calH_{\calU}z\right)\right)$, we calculate
\begin{align}
    &4\pi\Re\int_0^{\infty} \left(\overline{\p_r\left(\calH_{\calU}z\right)}G_ziG_z\left(\left(\p_r^3+\frac{\p_r^2}{r}-\frac{\p_r}{r^2}   \right)\left(\calH_{\calU}z\right)\right)\right)\udr\\
    =&4\pi\Re\int_0^{\infty} \left(\overline{\p_r\left(\calH_{\calU}z\right)}i\left(\left(\p_r^3+\frac{\p_r^2}{r}-\frac{\p_r}{r^2}   \right)\left(\calH_{\calU}z\right)\right)\right)\udr\\
    =&-4\pi\Re\int_0^{\infty} \left(\overline{\p_r^2\left(\calH_{\calU}z\right)}i\left(\p_r^2\left(\calH_{\calU}z\right)\right)\right)\udr-4\pi\Re\int_0^{\infty} \left(\overline{\frac{1}{r}\p_r\left(\calH_{\calU}z\right)}i\left(\p_r^2\left(\calH_{\calU}z\right)\right)\right)\udr\\
    &+4\pi\Re\int_0^{\infty} \left(\overline{\p_r\left(\calH_{\calU}z\right)}i\left(\left(\frac{\p_r^2}{r}-\frac{\p_r}{r^2}   \right)\left(\calH_{\calU}z\right)\right)\right)\udr\\
    =&-4\pi\Re\int_0^{\infty} \left(\overline{\p_r\left(\calH_{\calU}z\right)}i\left(\frac{\p_r}{r^2}   \left(\calH_{\calU}z\right)\right)\right)\udr=0.
\end{align}

For the component of $(\rm{E})_*$ containing $3\p_r^2(\calH_{\calU} z)$, we calculate
\begin{align}
    &4\pi\Re\int_0^{\infty} \left(\overline{\p_r\left(\calH_{\calU}z\right)}G_zi(\p_rG_z)\left( 3\p_r^2\left(\calH_{\calU}z\right) \right)\right)\udr\\
    =&12\pi\Re\int_0^{\infty} \left(\overline{\p_r\left(\calH_{\calU}z\right)}\p_r\left[G_zi(\p_rG_z)\left( \p_r\left(\calH_{\calU}z\right) \right)\right]\right)\udr\\
    &-12\pi\Re\int_0^{\infty} \left(\overline{\p_r\left(\calH_{\calU}z\right)}\left[\p_r\left(G_zi(\p_rG_z)\right)\right]\left( \p_r\left(\calH_{\calU}z\right) \right)\right)\udr\\
    =&-12\pi\Re\int_0^{\infty} \left(\overline{\p_r^2\left(\calH_{\calU}z\right)}\left[G_zi(\p_rG_z)\left( \p_r\left(\calH_{\calU}z\right) \right)\right]\right)\udr\\
    &-12\pi\Re\int_0^{\infty} \left(\overline{\p_r\left(\calH_{\calU}z\right)}\left[\p_r\left(G_zi(\p_rG_z)\right)\right]\left( \p_r\left(\calH_{\calU}z\right) \right)\right)\udr\\
    &-12\pi\Re\int_0^{\infty} \left(\overline{\p_r\left(\calH_{\calU}z\right)}\left[G_zi(\p_rG_z)\left( \p_r\left(\calH_{\calU}z\right) \right)\right]\right)\dr\\
    =&-12\pi\Re\int_0^{\infty} \left(\overline{\p_r\left(\calH_{\calU}z\right)}\left[(\p_rG_z)(-i)(G_z)\left( \p_r^2\left(\calH_{\calU}z\right) \right)\right]\right)\udr\\
    &-12\pi\Re\int_0^{\infty} \left(\overline{\p_r\left(\calH_{\calU}z\right)}\left[\p_r\left(G_zi(\p_rG_z)\right)\right]\left( \p_r\left(\calH_{\calU}z\right) \right)\right)\udr\\
    &-12\pi\Re\int_0^{\infty} \left(\overline{\p_r\left(\calH_{\calU}z\right)}\left[G_zi(\p_rG_z)\left( \p_r\left(\calH_{\calU}z\right) \right)\right]\right)\dr,
\end{align}
which combined with ~\eqref{eq:GziGz} yields
\begin{align}
    &4\pi\Re\int_0^{\infty} \left(\overline{\p_r\left(\calH_{\calU}z\right)}G_zi(\p_rG_z)\left( 3\p_r^2\left(\calH_{\calU}z\right) \right)\right)\udr\\
    =&-6\pi\Re\int_0^{\infty} \left(\overline{\p_r\left(\calH_{\calU}z\right)}\left[\p_r\left(G_zi(\p_rG_z)\right)\right]\left( \p_r\left(\calH_{\calU}z\right) \right)\right)\udr\\
    &-6\pi\Re\int_0^{\infty} \left(\overline{\p_r\left(\calH_{\calU}z\right)}\left[G_zi(\p_rG_z)\left( \p_r\left(\calH_{\calU}z\right) \right)\right]\right)\dr
\end{align}
Therefore we have
\begin{equation}\label{eq:E-2}
    \begin{aligned}
         &2\La \calH_{\calU}^2 z,\calH_{\calU}(-iG_z\calH_{\calU}z)   \Ra+4\pi\Re\int_0^{\infty} \left(\overline{\p_r\left(\calH_{\calU}(-iG_z\calH_{\calU} z)\right)}\left(G_z-\rm{Id}\right)(\p_r\left(\calH_{\calU}z\right))\right)\udr\\
    =&-6\pi\Re\int_0^{\infty} \left(\overline{\p_r\left(\calH_{\calU}z\right)}\left[\p_r\left(G_zi(\p_rG_z)\right)\right]\left( \p_r\left(\calH_{\calU}z\right) \right)\right)\udr\\
    &-6\pi\Re\int_0^{\infty} \left(\overline{\p_r\left(\calH_{\calU}z\right)}\left[G_zi(\p_rG_z)\left( \p_r\left(\calH_{\calU}z\right) \right)\right]\right)\dr\\
    &+4\pi\Re\int_0^{\infty} \left(\overline{\p_r\left(\calH_{\calU}z\right)}G_zi(\p_rG_z)\left( \left(2\frac{\p_r}{r}-\frac{1}{r^2}\right)\left(\calH_{\calU}z\right) \right)\right)\udr\\
    &+4\pi\Re\int_0^{\infty} \left(\overline{\p_r\left(\calH_{\calU}z\right)}G_zi(\p_r^2G_z)\left(\left(3\p_r+\frac{1}{r}\right)(\calH_{\calU}z)\right)\right)\udr\\
    &+4\pi\Re\int_0^{\infty} \left(\overline{\p_r\left(\calH_{\calU}z\right)}G_zi(\p_r^3G_z)(\calH_{\calU}z)\right)\udr\\
    & -4\pi\Re\int_0^{\infty} \left(\overline{\p_r\left(\calH_{\calU}z\right)}G_z\p_r(V_{\calU}(iG_z\calH_{\calU} z)  )\right)\udr   \\    
    &+4\pi\Re\int_0^{\infty}\overline{\calH_{\calU}z}V_{\calU}\left(\calH_{\calU}(-iG_z\calH_{\calU}z)\right)\udr.
    \end{aligned}
\end{equation}
Note that the terms containing high derivatives of $\calH_{\calU}z$ in the right hand side of the identity ~\eqref{eq:E-2} have been canceled out. The remaining terms except for the terms containing $V_{\calU}$ can be estimated directly using ~\eqref{eq:Kz-bounded}, ~\eqref{eq:Kz-rr-bounded} and ~\eqref{eq:Kz-rrr-bounded} to obtain a bound $\|z\|_{\calH^1}\|z\|_{\calH^3}^3$, under our bootstrap assumption this bound can be controlled by $|t|^{-2-3\veps_0}$. For the terms containing $V_\calU$, we will use the decomposition $ G_z=K_z+\rm{Id} $ and $V_{\calU}=\frac{k^2}{r^2}+\sum_{j=1}^NV_j+\calE_{\calU}$. Note that the most dangerous part cancels:
\begin{align}
    &-\Re\int_0^{\infty} \overline{\p_r(\calH_{\calU}z)}\p_r(r^{-2}i\calH_{\calU}z) \udr\\
    &+\Re\int_0^{\infty}\overline{\calH_{\calU}z}\frac{1}{r^2}\left(-\p_r^2-\frac{\p_r}{r}+\frac{k^2}{r^2}\right)\left(-i\calH_{\calU}z\right)\udr\\
    =&-\Re\int_0^{\infty} \overline{\p_r(\calH_{\calU}z)}(-2r^{-3})(i\calH_{\calU}z) \udr\\
    &+\Re\int_0^{\infty}\overline{\calH_{\calU}z}\frac{1}{r^2}\left(-\frac{\p_r}{r}\right)\left(-i\calH_{\calU}z\right)\udr\\
    &+\Re\int_0^{\infty}\p_r\left(\overline{\calH_{\calU}z}\frac{1}{r}\right)\p_r\left(-i\calH_{\calU}z\right)\dr=0.
\end{align}
The other components containing $K_z$ or $\p_rK_z$ can be estimated easily using ~\eqref{eq:Kz-bounded} to obtain the bound $\|z\|_{\calH^1}\|z\|_{\calH^3}^3$, which can be controlled by $|t|^{-2-3\veps_0}$, and we omit the details. The remaining components are
\begin{equation}\label{eq:E-3}
    \begin{aligned}
         & -4\pi\sum_{j=1}^N\Re\int_0^{\infty} \left(\overline{\p_r\left(\calH_{\calU}z\right)}\p_r(V_j(i\calH_{\calU} z)  )\right)\udr   \\    
    &+4\pi\sum_{j=1}^N\Re\int_0^{\infty}\overline{\calH_{\calU}z}V_j\left(\left(-\p_r^2-r^{-1}\p_r+V_{\calU}\right)(-i\calH_{\calU}z)\right)\udr\\
    &-4\pi\Re\int_0^{\infty} \left(\overline{\p_r\left(\calH_{\calU}z\right)}\p_r(\E_{\calU}(i\calH_{\calU} z)  )\right)\udr\\
    &+4\pi\Re\int_0^{\infty}\overline{\calH_{\calU}z}\calE_{\calU}\left(\left(-\p_r^2-r^{-1}\p_r+\frac{k^2}{r^2}+\sum_{j=1}^NV_j+\calE_{\calU}\right)(-i\calH_{\calU}z)\right)\udr.
    \end{aligned}
\end{equation}

The terms in the first and second lines of ~\eqref{eq:E-3} cancels in the following way:
\begin{align}
    &4\pi\sum_{j=1}^N\Re\int_0^{\infty}\overline{\calH_{\calU}z}V_j\left(\left(-\p_r^2-r^{-1}\p_r+\frac{k^2}{r^2}\right)(-i\calH_{\calU}z)\right)\udr\\
    =&-  4\pi\sum_{j=1}^N\Re\int_0^{\infty}\p_r\left(r\overline{\calH_{\calU}z}V_j\right)\left(\p_r(i\calH_{\calU}z)\right)\dr     \\
    &+4\pi\sum_{j=1}^N\Re\int_0^{\infty}\overline{\calH_{\calU}z}V_j\left(\left(-r^{-1}\p_r\right)(-i\calH_{\calU}z)\right)\udr\\
    =&-  4\pi\sum_{j=1}^N\Re\int_0^{\infty}r\cdot\overline{\calH_{\calU}z}(\p_rV_j)\left(\p_r(i\calH_{\calU}z)\right)\dr   \\
    =& 4\pi\sum_{j=1}^N\Re\int_0^{\infty}\overline{i\calH_{\calU}z}(\p_rV_j)\left(\p_r(\calH_{\calU}z)\right)\udr\\
    =& 4\pi\sum_{j=1}^N\Re\int_0^{\infty} \left(\overline{\p_r\left(\calH_{\calU}z\right)}\p_r(V_j(i\calH_{\calU} z)  )\right)\udr 
\end{align}
Therefore these terms can be estimated as
\begin{align}
    &\left|4\pi\sum_{j=1}^N\Re\int_0^{\infty}\overline{\calH_{\calU}z}V_j\left(\E_{\calU}(-i\calH_{\calU}z)\right)\udr\right|\\
    \lesssim &\sum_{j=1}^N\sum_{1\leq m<\ell\leq N}\int_0^{\infty}|z|_{-2}^2\cdot r^{-4}F_j^2(t,r)F_m(t,r)F_{\ell}(t,r)\udr\\
    \lesssim &\sum_{1\leq m<\ell\leq N}\int_0^{\infty}|z|_{-3}^2\cdot r^{-2}F_m(t,r)F_{\ell}(t,r)\udr\\
    \lesssim & \sum_{1\leq m<\ell\leq N}\int_0^{\infty}|z|_{-3}^2\cdot \frac{\lam_{\ell}}{(\lam_{\ell}+r)^2}\frac{1}{\lam_{\ell}}\left(\frac{\lam_m}{\lam_{\ell}}\right)^{k-2}\udr\\
    \lesssim & \sum_{\ell=1}^N |t|^{-1}\lam_{\ell}\|z\|_{\rm{Mor}(\lam_{\ell})}^2.
\end{align}
The terms in the third and fourth lines of ~\eqref{eq:E-3} can be estimated similarly to obtain the same bound.

Next we estimate the term in the second line of ~\eqref{eq:E-all} using ~\eqref{eq:DGz-bounded} as follows:
\begin{align}
    &\left|\int_0^{\infty} \left(\overline{\p_r\left(\calH_{\calU}z\right)}\left(\vD G_z(-iG_z\calH_{\calU} z)\right)(\p_r\left(\calH_{\calU}z\right))\right)\udr\right|\\
    \lesssim&\int_0^{\infty}|z|_{-3}|z||z|_{-2}|z|_{-3}\udr
    \lesssim \|z\|_{\calH^1}\|z\|_{\calH^3}^3\lesssim|t|^{-2-3\veps_0}.
\end{align}
For the terms in the last two lines of ~\eqref{eq:E-all}, we write $z=y-\eta$ to obtain
\begin{equation}\label{eq:E-4}
    \begin{aligned}
        &\La A_{\lam_j}\calH_{\calU}(\calR_{\rm{q}}) ,i\Lam_{\lam_j}A_{\lam_j}y \Ra+\La A_{\lam_j}\calH_{\calU}y ,i\Lam_{\lam_j}A_{\lam_j}(\calR_{\rm{q}}) \Ra\\
    =&\La A_{\lam_j}\calH_{\calU}(-iK_z\calH_{\calU}(y-\eta)) ,i\Lam_{\lam_j}A_{\lam_j}y \Ra+\La A_{\lam_j}\calH_{\calU}y ,i\Lam_{\lam_j}A_{\lam_j}(-iK_z\calH_{\calU}(y-\eta)) \Ra\\
    =&-\La A_{\lam_j}\calH_{\calU}(iK_z\calH_{\calU}y) ,i\Lam_{\lam_j}A_{\lam_j}y \Ra-\La A_{\lam_j}\calH_{\calU}y ,i\Lam_{\lam_j}A_{\lam_j}(iK_z\calH_{\calU}y) \Ra\\
    &+\La A_{\lam_j}\calH_{\calU}(iK_z\calH_{\calU}\eta) ,i\Lam_{\lam_j}A_{\lam_j}y \Ra+\La A_{\lam_j}\calH_{\calU}y ,i\Lam_{\lam_j}A_{\lam_j}(iK_z\calH_{\calU}\eta) \Ra
    \end{aligned}
\end{equation}
We denote by $(E)_j$ the sums of the four terms in the last two lines in ~\eqref{eq:E-4} for each $1\leq j\leq N$.

We first estimate the terms in $(E)_j$ containing $-iK_z\calH_{\calU}y$.
\begin{align}
    &-\La A_{\lam_j}\calH_{\calU}(iK_z\calH_{\calU}y) ,i\Lam_{\lam_j}A_{\lam_j}y \Ra-\La A_{\lam_j}\calH_{\calU}y ,i\Lam_{\lam_j}A_{\lam_j}(iK_z\calH_{\calU}y) \Ra\\
    =&-\La A_{\lam_j}\calH_{\calU}(K_z\calH_{\calU}y) ,\Lam_{\lam_j}A_{\lam_j}y \Ra+\La A_{\lam_j}\calH_{\calU}y ,\Lam_{\lam_j}A_{\lam_j}(K_z\calH_{\calU}y) \Ra-\La A_{\lam_j}[\calH_{\calU},i](K_z\calH_{\calU}y) ,i\Lam_{\lam_j}A_{\lam_j}y \Ra\\
    =&-\La K_z\calH_{\calU}y,A_{\lam_j}^*\Lam_{\lam_j}A_{\lam_j}\calH_{\calU}y \Ra-  \La K_z\calH_{\calU}y,[\calH_{\calU},A_{\lam_j}^*\Lam_{\lam_j}A_{\lam_j}]y \Ra\\
    &+\La A_{\lam_j}\calH_{\calU}y ,\Lam_{\lam_j}A_{\lam_j}(K_z\calH_{\calU}y) \Ra-\La A_{\lam_j}[\calH_{\calU},i](K_z\calH_{\calU}y) ,i\Lam_{\lam_j}A_{\lam_j}y \Ra\\
    =&-  \La K_z\calH_{\calU}y,[\calH_{\calU},A_{\lam_j}^*\Lam_{\lam_j}A_{\lam_j}]y \Ra+\La \calH_{\calU}y,[ A_{\lam_j}^*\Lam_{\lam_j}A_{\lam_j},K_z   ]\calH_{\calU}y\Ra-\La A_{\lam_j}[\calE_{\calU},i](K_z\calH_{\calU}y) ,i\Lam_{\lam_j}A_{\lam_j}y \Ra.
\end{align}
The coefficients of $\p_r^3$ and $\p_r^2$ of the third order differential operator 
$A_{\lam_j}^*\Lam_{\lam_j}A_{\lam_j}$ is:
\begin{align}
    -\varphi_{\lam_j}\p_r^3-\frac{3}{2}\left(\p_r\varphi_{\lam_j}+\frac{\varphi_{\lam_j}}{r}\right)\p_r^2.
\end{align}
We can calculate the highest order and second-highest order parts of the commutators $[\calH_{\calU},A_{\lam_j}^*\Lam_{\lam_j}A_{\lam_j}]$ and $[ A_{\lam_j}^*\Lam_{\lam_j}A_{\lam_j},K_z   ]$ as
\begin{align}
    [\calH_{\calU},A_{\lam_j}^*\Lam_{\lam_j}A_{\lam_j}]&=2(\p_r\varphi_{\lam_j})\p_r^4+\rm{l.o.t},\\
    [ A_{\lam_j}^*\Lam_{\lam_j}A_{\lam_j},K_z   ]&=-3\varphi_{\lam_j}(\p_rK_z)\p_r^2+\rm{l.o.t}.
\end{align}
We only estimate the terms containing these leading order parts, as the estimates for  the remaining parts are much easier.
\begin{align}
    -  2\La K_z\calH_{\calU}y,(\p_r\varphi_{\lam_j})\p_r^4y \Ra
    &=-2\La K_z\p_r^3y, (\p_r\varphi_{\lam_j})\p_r^3y  \Ra+\rm{l.o.t},\\
    \La \calH_{\calU}y,-3\varphi_{\lam_j}(\p_rK_z)\p_r^2\calH_{\calU}y\Ra&=3\La \p_r^3 y,\varphi_{\lam_j}(\p_rK_z)\p_r^3y\Ra+\rm{l.o.t},
\end{align}
Using ~\eqref{eq:Kz-bounded}, ~\eqref{eq:Kz-rr-bounded}, ~\eqref{eq:Kz-rrr-bounded} and the fact that $|\varphi_{\lam_j}|+|r\p_r\varphi_{\lam_j}|\lesssim 1$, these terms can be estimated as
$\|z\|_{\calH^1}^{\frac{3}{2}}\|z\|_{\calH^3}^{\frac{1}{2}}\|y\|_{\calH^3}^{2}$, which is bounded by $|t|^{-2-\frac{5}{2}\veps_0}$.

The term $ \La A_{\lam_j}[\calE_{\calU},i](K_z\calH_{\calU}y) ,i\Lam_{\lam_j}A_{\lam_j}y \Ra$ can be estimated using ~\eqref{eq:Lambda-estimate-pointwise} and ~\eqref{eq:commutator-pointwise} to obtain
\begin{align}
    \left|\La A_{\lam_j}[\calE_{\calU},i](K_z\calH_{\calU}y) ,i\Lam_{\lam_j}A_{\lam_j}y \Ra\right|\lesssim |t|^{-1}\sum_{1\leq m<\ell\leq N}\frac{\lam_m}{\lam_{\ell}}\|y\|_{\rm{Mor}(\lam_{\ell})}^2.
\end{align}

Then we estimate the terms in $(E)_j$ containing $iK_z\calH_{\calU}\eta$, where we have to use ~\eqref{eq:correction-H4}.
\begin{align}
    &\La A_{\lam_j}\calH_{\calU}(iK_z\calH_{\calU}\eta) ,i\Lam_{\lam_j}A_{\lam_j}y \Ra
      +\La A_{\lam_j}\calH_{\calU}y ,i\Lam_{\lam_j}A_{\lam_j}(iK_z\calH_{\calU}\eta) \Ra\\
     = &\La A_{\lam_j}\calH_{\calU}(K_z\calH_{\calU}\eta) ,\Lam_{\lam_j}A_{\lam_j}y \Ra+\La A_{\lam_j}[\calH_{\calU},i](K_z\calH_{\calU}\eta) ,i\Lam_{\lam_j}A_{\lam_j}y \Ra\\
     &-\La A_{\lam_j}\calH_{\calU}y ,\Lam_{\lam_j}A_{\lam_j}(K_z\calH_{\calU}\eta) \Ra
\end{align}
The term $\La A_{\lam_j}[\calH_{\calU},i](K_z\calH_{\calU}\eta) ,i\Lam_{\lam_j}A_{\lam_j}y \Ra$ can be estimated using ~\eqref{eq:Kz-bounded}, ~\eqref{eq:commutator-pointwise}, ~\eqref{eq:FkFl-pointwise}, ~\eqref{eq:Lambda-estimate-pointwise} to obtain the bound
$|t|^{-1}\|z\|_{\calH^1}^2\|\eta\|_{\calH^2}\|y\|_{\calH^3} $, which can be bounded by $|t|^{-\frac{7}{2}-\veps_0}$.
For the remaining two terms, since $A_{\lam_j}\calH_{\calU}=\p_r^3+\rm{l.o.t}$, $ \Lam_{\lam_j}A_{\lam_j}=-\varphi_{\lam_j}\p_r^2+\rm{l.o.t}$, $A_{\lam_j}^*\Lam_{\lam_j}A_{\lam_j}=-\varphi_{\lam_j}\p_r^3+\rm{l.o.t}$, the leading order terms are 
\begin{align}
\La K_z\p_r^2(\calH_{\calU}\eta) ,\varphi_{\lam_j}\p_r^3y \Ra+\La \p_r^3y ,\varphi_{\lam_j}K_z\p_r^2(\calH_{\calU}\eta) \Ra=2\La K_z\p_r^2(\calH_{\calU}\eta) ,\varphi_{\lam_j}\p_r^3y \Ra.
\end{align}
Using ~\eqref{eq:Kz-bounded}, ~\eqref{eq:correction-H3} and ~\eqref{eq:correction-H4} we obtain
\begin{align}
    &\left|  \La K_z\p_r^2(\calH_{\calU}\eta) ,\varphi_{\lam_j}\p_r^3y \Ra\right|\\
    \lesssim &  \int_0^{\infty}\frac{|z|^2}{r^2}|r^2\p_r^2(\calH_{\calU}\eta)||\p_r^3 y|\udr\\
    \lesssim &\|z\|_{\calH^1}\|z\|_{\calH^3}\|r^2\p_r^2(\calH_{\calU}\eta)\|_{L^2(\udr)}\|y\|_{\calH^3}\lesssim_L |t|^{-\frac{5}{2}-\veps_0}.
\end{align}
The remaining parts can be estimated in similar ways to obtain the bound $C(L)|t|^{-\frac{5}{2}-\veps_0}$.

Combining these estimates, we obtain that there exists a constant $C_{L,E}$ depending on $L$ such that
\begin{align}\label{eq:E-estimate-final}
    |(\rm{E})|&\lesssim |t|^{-1}\sum_{j=1}^N\|y\|_{\rm{Mor}(\lam_j)}^2+|t|^{-1}\sum_{j=1}^N\lam_j\|z\|_{\rm{Mor}(\lam_j)}^2+C_{L,E}o_{T_0\to-\infty}(1)|t|^{-2-2\veps_0}.
\end{align}

\bigskip
\noindent
\textbf{Estimates for $(\rm{F})$.} We first rearrange the expression of $(\rm{F})$ in ~\eqref{eq:F-expression} as follows:
\begin{align}
    (\rm{F})&=\ti \de\sum_{j=1}^N \de_0^{N-j}\La A_{\lam_j}\calH_{\calU}(-i\calH_{\calU} y) ,i\Lam_{\lam_j}A_{\lam_j}y \Ra
      +\ti \de\sum_{j=1}^N \de_0^{N-j}\La A_{\lam_j}\calH_{\calU}y ,i\Lam_{\lam_j}A_{\lam_j}(-i\calH_{\calU} y) \Ra\\
      &=-\ti \de\sum_{j=1}^N \de_0^{N-j}\La A_{\lam_j}\calH_{\calU}(\calH_{\calU} y) ,\Lam_{\lam_j}A_{\lam_j}y \Ra-\ti \de\sum_{j=1}^N \de_0^{N-j}\La A_{\lam_j}[\calH_{\calU},i](\calH_{\calU} y) ,i\Lam_{\lam_j}A_{\lam_j}y \Ra.
\end{align}
Since
\begin{align}
    \calH_{\calU}=H_{\lam_j}+\sum_{\ell \neq j}V_{\ell}+\calE_{\calU},\ A_{\lam_j}H_{\lam_j}=\ti H_{\lam_j}A_{\lam_j},\ [\calH_{\calU},i]=[\calE_{\calU},i],
\end{align}
we have
\begin{align}
    A_{\lam_j}\calH_{\calU}^2&=A_{\lam_j}H_{\lam_j}^2+A_{\lam_j}H_{\lam_j}\left(\sum_{\ell\neq j}V_{\ell}+\calE_{\calU}\right)+A_{\lam_j}\left(\sum_{\ell\neq j}V_{\ell}+\calE_{\calU}\right)H_{\lam_j}+A_{\lam_j}\left(\sum_{\ell\neq j}V_{\ell}+\calE_{\calU}\right)^2\\
    &=\ti H_{\lam_j}^2A_{\lam_j}+\ti H_{\lam_j}A_{\lam_j}\left(\sum_{\ell\neq j}V_{\ell}+\calE_{\calU}\right)+A_{\lam_j}\left(\sum_{\ell\neq j}V_{\ell}+\calE_{\calU}\right)\calH_{\calU}
\end{align}
Therefore
\begin{equation}\label{eq:F-all}
    \begin{aligned}
         (\rm{F})&=-\ti \de\sum_{j=1}^N \de_0^{N-j}\La \ti H_{\lam_j}^2A_{\lam_j}y ,\Lam_{\lam_j}A_{\lam_j}y \Ra-\ti \de\sum_{j=1}^N \de_0^{N-j}\La \ti H_{\lam_j}A_{\lam_j}\left(\sum_{\ell\neq j}V_{\ell}+\calE_{\calU}\right)y ,\Lam_{\lam_j}A_{\lam_j}y \Ra\\
    &\quad -\ti \de\sum_{j=1}^N \de_0^{N-j}\La A_{\lam_j}\left(\sum_{\ell\neq j}V_{\ell}+\calE_{\calU}\right)\calH_{\calU} y,\Lam_{\lam_j}A_{\lam_j}y \Ra\\
    &\quad -\ti \de\sum_{j=1}^N \de_0^{N-j}\La A_{\lam_j}[\calE_{\calU},i](\calH_{\calU} y) ,i\Lam_{\lam_j}A_{\lam_j}y \Ra.
    \end{aligned}
\end{equation}

By Lemma ~\ref{lem:Morawetz-one}, for some $c_1>0$ only depending on $k,\de,R$, there holds
\begin{align}
    -\ti \de\sum_{j=1}^N \de_0^{N-j}\La \ti H_{\lam_j}^2A_{\lam_j}y ,\Lam_{\lam_j}A_{\lam_j}y \Ra\leq-c_1\ti\de\sum_{j=1}^N \de_0^{N-j}\|A_{\lam_j}y\|_{\rm{Mor}_2(\lam_j)}^2.
\end{align}

Next we estimate the remaining terms in ~\eqref{eq:F-all}.
~\eqref{eq:Vj-estimate-pointwise}, ~\eqref{eq:EU-estimate-pointwise} yield that
\begin{align}
    &\left|\La \ti H_{\lam_j}A_{\lam_j}\left(\sum_{\ell\neq j}V_{\ell}+\calE_{\calU}\right)y ,\Lam_{\lam_j}A_{\lam_j}y \Ra\right|\\
    =&\left|\La A^*_{\lam_j}A_{\lam_j}\left(\sum_{\ell\neq j}V_{\ell}+\calE_{\calU}\right)y , A^*_{\lam_j}\Lam_{\lam_j}A_{\lam_j}y \Ra\right|\\
    \lesssim&\sum_{\ell\neq j}\int_0^{\infty} r^{-2}F_{\ell}^2(t,r)r|y|_{-3}   \frac{r}{\lam_j+r}|y|_{-3}\udr \\
    &+\sum_{1\leq m<\ell\leq N}\int_0^{\infty} r^{-2}F_m(t,r)F_{\ell}(t,r)r|y|_{-3}   \frac{r}{\lam_j+r}|y|_{-3}\udr\\
    = &\int_0^{\infty}\frac{\sum_{\ell\neq j}F_{\ell}^2(t,r)+\sum_{1\leq m<n\leq N}F_m(t,r)F_n(t,r)}{\lam_j+r}|y|_{-3}^2\udr
\end{align}

\begin{align}
    &\left|\La A_{\lam_j}\left(\sum_{\ell\neq j}V_{\ell}+\calE_{\calU}\right)\calH_{\calU} y,\Lam_{\lam_j}A_{\lam_j}y \Ra\right|\\
    =&\left|\La\left(\sum_{\ell\neq j}V_{\ell}+\calE_{\calU}\right)\calH_{\calU} y, A_{\lam_j}^*\Lam_{\lam_j}A_{\lam_j}y \Ra\right|\\
    \lesssim &\int_0^{\infty}\frac{\sum_{\ell\neq j}F_{\ell}^2(t,r)+\sum_{1\leq m<n\leq N}F_m(t,r)F_n(t,r)}{\lam_j+r}|y|_{-3}^2\udr.
\end{align}

\begin{align}
    &\left|\La A_{\lam_j}[\calE_{\calU},i](\calH_{\calU} y) ,i\Lam_{\lam_j}A_{\lam_j}y \Ra\right|\\
    =  &\left|\La [\calE_{\calU},i](\calH_{\calU} y) ,iA_{\lam_j}^*\Lam_{\lam_j}A_{\lam_j}y \Ra\right|\\
    \lesssim &\int_0^{\infty}\frac{\sum_{\ell\neq j}F_{\ell}^2(t,r)+\sum_{1\leq m<n\leq N}F_m(t,r)F_n(t,r)}{\lam_j+r}|y|_{-3}^2\udr.
\end{align}
We have to estimate
\begin{align}
    \frac{\sum_{\ell\neq j}F_{\ell}^2(t,r)+\sum_{1\leq m<n\leq N}F_m(t,r)F_n(t,r)}{\lam_j+r}.
\end{align}
For $\ell\neq j$, if $\ell<j$, for $r\leq\lam_{\ell}$, we have
\begin{align}
    \frac{F_{\ell}^2(t,r)}{\lam_j+r}\cdot\frac{(\lam_{\ell}+r)^2}{\lam_{\ell}}=\left(\frac{r}{\lam_{\ell}}\right)^{2k}\frac{\lam_{\ell}+r}{\lam_{\ell}}\frac{\lam_{\ell}+r}{\lam_j+r}\leq 4\frac{\lam_{\ell}}{\lam_j}
\end{align}
For $\lam_{\ell}  <r\leq\lam_j $, we have
\begin{align}
    \frac{F_{\ell}^2(t,r)}{\lam_j+r}\cdot\frac{(\lam_{\ell}+r)^2}{\lam_{\ell}}=\left(\frac{\lam_{\ell}}{r}\right)^{2k}\frac{\lam_{\ell}+r}{\lam_j+r}\frac{\lam_{\ell}+r}{\lam_{\ell}}\leq 4 \left(\frac{\lam_{\ell}}{r}\right)^{2k}\frac{r}{\lam_j}\frac{r}{\lam_{\ell}}\leq 4\frac{\lam_{\ell}}{\lam_j}
\end{align}
For $r>\lam_j $, we have
\begin{align}
    \frac{F_{\ell}^2(t,r)}{\lam_j+r}\cdot\frac{(\lam_{\ell}+r)^2}{\lam_{\ell}}=\left(\frac{\lam_{\ell}}{r}\right)^{2k}\frac{\lam_{\ell}+r}{\lam_j+r}\frac{\lam_{\ell}+r}{\lam_{\ell}}\leq 4 \left(\frac{\lam_{\ell}}{r}\right)^{2k}\frac{r}{\lam_{\ell}}\leq  4\frac{\lam_{\ell}}{\lam_j}
\end{align}
If $\ell>j$, for $r>\lam_{\ell}$, we have
\begin{align}
    \frac{F_{\ell}^2(t,r)}{\lam_j+r}\cdot\frac{(\lam_{\ell}+r)^2}{\lam_{\ell}}=\left(\frac{\lam_{\ell}}{r}\right)^{2k}\frac{\lam_{\ell}+r}{\lam_{\ell}}\frac{\lam_{\ell}+r}{\lam_j+r}\leq 4  \left(\frac{\lam_{\ell}}{r}\right)^{2k}\frac{r}{\lam_{\ell}}\frac{r}{r}\leq 4.
\end{align}
For $\lam_j<r\leq \lam_{\ell}$, we have
\begin{align}
    \frac{F_{\ell}^2(t,r)}{\lam_j+r}\cdot\frac{(\lam_{\ell}+r)^2}{\lam_{\ell}}=\left(\frac{r}{\lam_{\ell}}\right)^{2k}\frac{\lam_{\ell}+r}{\lam_{\ell}}\frac{\lam_{\ell}+r}{\lam_j+r}\leq 4 \left(\frac{r}{\lam_{\ell}}\right)^{2k}\frac{\lam_{\ell}}{\lam_{\ell}}\frac{\lam_{\ell}}{r}\leq 4.
\end{align}
For $r\leq \lam_j$, we have
\begin{align}
    \frac{F_{\ell}^2(t,r)}{\lam_j+r}\cdot\frac{(\lam_{\ell}+r)^2}{\lam_{\ell}}=\left(\frac{r}{\lam_{\ell}}\right)^{2k}\frac{\lam_{\ell}+r}{\lam_{\ell}}\frac{\lam_{\ell}+r}{\lam_j+r}\leq 4 \left(\frac{r}{\lam_{\ell}}\right)^{2k}\frac{\lam_{\ell}}{\lam_{\ell}}\frac{\lam_{\ell}}{\lam_j}=  4 \left(\frac{r}{\lam_{\ell}}\right)^{2k-1}\frac{r}{\lam_j} \leq 4.
\end{align}
For $m<n$, similarly we have
\begin{align}
    \frac{F_m(t,r)F_n(t,r)}{\lam_j+r}\frac{(\lam_n+r)^2}{\lam_n}\leq 4 \frac{\lam_m}{\lam_n}.
\end{align}

Therefore we have
\begin{align}
    &\sum_{j=1}^N \de_0^{N-j}\Big|\La \ti H_{\lam_j}A_{\lam_j}\left(\sum_{\ell\neq j}V_\ell+\calE_{\calU}\right)y ,\Lam_{\lam_j}A_{\lam_j}y \Ra+\La A_{\lam_j}\left(\sum_{\ell\neq j}V_\ell+\calE_{\calU}\right)\calH_{\calU} y,\Lam_{\lam_j}A_{\lam_j}y \Ra  \\
    &+\La A_{\lam_j}[\calE_{\calU},i](\calH_{\calU} y) ,i\Lam_{\lam_j}A_{\lam_j}y \Ra \Big|\\
    \lesssim &\sum_{j=1}^N \de_0^{N-j}\left[\sum_{\ell< j}\frac{\lam_\ell}{\lam_j}\|y\|_{\rm{Mor}(\lam_\ell)}^2 +\sum_{\ell>j}\|y\|_{\rm{Mor}(\lam_\ell)}^2  +\sum_{1\leq m<n\leq N}\frac{\lam_m}{\lam_n}\|y\|_{\rm{Mor}(\lam_n)}^2\right]\\
    =& \sum_{j=1}^N\sum_{\ell>j}\de_0^{j-\ell}\frac{\lam_j}{\lam_{\ell}}\de_0^{N-j}\|y\|_{\rm{Mor}(\lam_j)}^2+\sum_{j=1}^N\sum_{\ell<j}\de_0^{j-\ell}\de_0^{N-j}\|y\|_{\rm{Mor}(\lam_j)}^2\\
    &+\frac{1-\de_0^N}{1-\de_0}\left[\sum_{1\leq m<n\leq N}\frac{\lam_m}{\lam_n}\de_0^{n-N}\de_0^{N-n}\|y\|_{\rm{Mor}(\lam_n)}^2\right].
\end{align}

Combining these estimates yields that there exists a constant $C_{1,F}$ only depending on $k,N$ such that
\begin{equation}\label{eq:F-estimate-final}
    \begin{aligned}
        (\rm{F})\leq& -c_1\ti\de\sum_{j=1}^N \de_0^{N-j}\|A_{\lam_j}y\|_{\rm{Mor}_2(\lam_j)}^2\\
    &+C_{1,F}\cdot\ti\de\sum_{j=1}^N\left( o_{\de_0\to 0}(1)+\de_0^{-N} o_{T_0\to-\infty}(1) \right) \de_0^{N-j}\left(\|A_{\lam_j}y\|_{\rm{Mor}_2(\lam_j)}^2+\|y\|_{\rm{Mor}(\lam_j)}^2\right).
    \end{aligned}
\end{equation}

\subsection{Completing the proof of Proposition~\ref{prop:I3'}}

In this section we complete the proof of Proposition ~\ref{prop:I3'}.

We first study the relation among $\|z\|_{\rm{Mor}(\lam_j)},\|y\|_{\rm{Mor}(\lam_j)},\|A_{\lam_j}y\|_{\rm{Mor}_2(\lam_j)}$.

First of all using ~\eqref{eq:correction-H3} we have
\begin{align}\label{eq:z-y-Mor}
    \lam_j\|z\|_{\rm{Mor}(\lam_j)}^2\leq \|y\|_{\rm{Mor}(\lam_j)}^2+C(L)|t|^{-2}.
\end{align}

Next using Lemma ~\ref{lem:Alam-coercive}, we obtain
\begin{align}
    \|y\|_{\rm{Mor}(\lam_j)}^2\lesssim  \ti C(L)\|A_{\lam_j}y\|_{\rm{Mor}_2(\lam_j)}^2+\lam_j^{-9}\left|\La y,\chi_{L,j}\sin\phi(r/\lam_j) \Ra\right|^2+\lam_j^{-9}\left|\La y,i\chi_{L,j}\sin\phi(r/\lam_j) \Ra\right|^2\\
\end{align}
Consider the frame defined in Section ~\ref{subsec:Darboux} with $\calV$ replaced by $Q_{\lam_j}$. A direct calculation yields that the related orthogonal frame, denoted by $(f_1^{Q_j},f_2^{Q_j},Q_{\lam_j})$ for convenience, satisfies
\begin{align}
    \Lam Q_{\lam_j}=k\sin\phi(r/\lam_j)f_1^{Q_j},\ RQ_{\lam_j}=\sin \phi(r/\lam_j)f_2^{Q_j}.
\end{align}

\begin{lem}\label{lem:frame-compare}
    For $L\gg1$ and $T_0\ll -1$ sufficiently negative such that 
    \begin{align}
        L^2\lam_{j}\leq \sqrt{\lam_{j}\lam_{j+1}}\leq \frac{\lam_{j+1}}{L^2}, \ 0\leq j\leq N,
    \end{align}
    where we adopt the convention $\lam_0=0,\lam_{N+1}=\infty$, we have the following estimates:\\
    For $1\leq j\leq N$ and $L^{-2}\lam_j\leq r\leq L^2\lam_j$, there holds
    \begin{equation}\label{eq:frame-compare-core}
        \begin{aligned}
            &\left|f_1(\te,r)- (-1)^{[\frac{j}{2}]}e^{\te_j R}\left(\cos \te_jf_1^{Q_j}(\te,r)-\sin\te_jf_2^{Q_j}(\te,r)\right)\right|\\
            &+\left|f_2(\te,r)-(-1)^{[\frac{j-1}{2}]} e^{\te_j R}\left(\sin \te_jf_1^{Q_j}(\te,r)+\cos\te_jf_2^{Q_j}(\te,r)\right)\right|\\
            \lesssim & L^{-2k}+C(L)\sum_{\ell=1}^{N}\left[\left(\frac{\lam_{\ell}}{\lam_{\ell+1}}\right)^k+\left|\te_{\ell}-\te_{\ell}^{\infty}\right|\right].
        \end{aligned}
    \end{equation}
     For $0\leq j\leq N$ and $L^2\lam_j\leq r\leq  L^{-2}\lam_{j+1}$, there holds
    \begin{align}\label{eq:frame-compare-transition}
        &\left|f_1(\te,r)-(-1)^{[\frac{j+1}{2}]}e^{k\te R}e_1\right|+\left|f_2(\te,r)-(-1)^{[\frac{j}{2}]}e^{k\te R}e_2\right|\\
        \lesssim  & L^{-2k}+C(L)\sum_{\ell=1}^{N}\left[\left(\frac{\lam_{\ell}}{\lam_{\ell+1}}\right)^k+\left|\te_{\ell}-\te_{\ell}^{\infty}\right|\right].
    \end{align}
Here $[m]$ denotes the integral part of $m$.
    
\end{lem}

\begin{proof}[Proof of Lemma ~\ref{lem:frame-compare}]
    In the proof we suppress the variable $\te$ for convenience.

    For any $0\leq \ell\leq N$, it follows from ~\eqref{eq:Q-cal-expansion-local} and ~\eqref{eq:Q-cal-norm} that 
    \begin{align}
        \sup_{L^2\lam_{\ell}\leq r\leq  L^{-2}\lam_{\ell+1}}|\calU(r)-(-1)^{\ell}e_3|+\int_{L^2\lam_{\ell}}^{L^{-2}\lam_{\ell+1}}\left|\p_r\calU(r)\right|\ud r\lesssim L^{-2k}+\int_{L^2\lam_{\ell}}^{L^{-2}\lam_{\ell+1}}\sum_{m=1}^N F_m\frac{\ud r}{r}\lesssim L^{-2k}.
    \end{align}
This implies
    \begin{align}\label{eq:derivative-transition}
       \int_{L^2\lam_{\ell}}^{L^{-2}\lam_{\ell+1}}\left|\p_r f_1(r)\right|\ud r+\int_{L^2\lam_{\ell}}^{L^{-2}\lam_{\ell+1}}\left|\p_r f_2(r)\right|\ud r\lesssim \int_{L^2\lam_{\ell}}^{L^{-2}\lam_{\ell+1}}\left|\p_r\calU(r)\right|\ud r\lesssim L^{-2k}.
    \end{align}
 Define 
    \begin{align}
        \ti f_1^{Q_{\ell}}&=(-1)^{[\frac{\ell}{2}]}e^{\te_{\ell} R}\left(\cos \te_{\ell}f_1^{Q_{\ell}}(r)-\sin\te_{\ell}f_2^{Q_{\ell}}(r)\right),\\ \ti f_2^{Q_{\ell}}&=(-1)^{[\frac{{\ell}-1}{2}]}e^{\te_{\ell} R}\left(\sin \te_{\ell}f_1^{Q_{\ell}}(r)+\cos\te_{\ell}f_2^{Q_{\ell}}(r)\right).
    \end{align}
    Then we have
    \begin{align}
        \p_r \ti f_1^{Q_{\ell}}=-\left(\ti f_1^{Q_{\ell}}\cdot \p_r\left((-1)^{{\ell}-1}e^{\te_{\ell} R}Q_{\lam_{\ell}}\right)\right)(-1)^{\ell-1}e^{\te_{\ell} R}Q_{\lam_{\ell}},\\
        \p_r \ti f_2^{Q_{\ell}}=-\left(\ti f_2^{Q_{\ell}}\cdot \p_r\left((-1)^{\ell-1}e^{\te_{\ell} R}Q_{\lam_{\ell}}\right)\right)(-1)^{{\ell}-1}e^{\te_{\ell} R}Q_{\lam_{\ell}},
    \end{align}
    which implies for $m=1,2$,
    \begin{equation}\label{eq:derivative-core}
        \begin{aligned}
             \left|\p_r\left(f_m-\ti f_m^{Q_{\ell}}\right)\right| 
        \lesssim \left|\calU-(-1)^{{\ell}-1}e^{\te_{\ell} R}Q_{\lam_{\ell}}\right|\left|\p_r\calU\right| 
        +\left|\p_r\left(\calU-(-1)^{{\ell}-1}e^{\te_{\ell} R}Q_{\lam_{\ell}}\right)\right|
        +\left|f_m-\ti f_m^{Q_{\ell}}\right| |\p_r\calU|.
        \end{aligned}
    \end{equation}
    The expansion of $\ti f_1^{Q_{\ell}}$, $\ti f_2^{Q_{\ell}}$ at $L^{-2}\lam_{\ell}$ and $L^2\lam_{\ell}$ is given by
    \begin{align}\label{eq:expansion-core-1}
        \ti f_1^{Q_{\ell}}(L^{-2}\lam_{\ell})=(-1)^{[\frac{\ell}{2}]}e^{k\te R}e_1+O(L^{-2k}),\ \ti f_2^{Q_{\ell}}(L^{-2}\lam_{\ell})=(-1)^{[\frac{\ell-1}{2}]}e^{k\te R}e_2+O(L^{-2k}),
    \end{align}
    \begin{equation}\label{eq:expansion-core-2}
        \begin{aligned}
             \ti f_1^{Q_{\ell}}(L^{2}\lam_{\ell})&=(-1)^{[\frac{\ell+1}{2}]}e^{k\te R}e_1+O(L^{-2k}+|\te_{\ell}-\te_{\ell}^{\infty}|),\\
        \ti f_2^{Q_{\ell}}(L^{2}\lam_{\ell})&=(-1)^{[\frac{\ell}{2}]}e^{k\te R}e_2+O(L^{-2k}+|\te_{\ell}-\te_{\ell}^{\infty}|).
        \end{aligned}
    \end{equation}

    We argue by induction on $j$. 

    For $j=0$, using the initial condition in ~\eqref{eq:frame-def}, $\calU(0)=e_3$ and  ~\eqref{eq:derivative-transition} we have
    \begin{align}
        \left|f_1(r)-e^{k\te R}e_1\right|+\left|f_2(r)-e^{k\te R}e_2\right|\lesssim \int_0^{L^{-2}\lam_1}|\p_r\calU(r)|\ud r \lesssim L^{-2k}.
    \end{align}

    For $1\leq j\leq N$, if ~\eqref{eq:frame-compare-transition} holds for $j-1$, then using ~\eqref{eq:expansion-core-1}  we have
    \begin{align}
        \left|f_1(L^{-2}\lam_j)-  \ti f_1^{Q_j}(L^{-2}\lam_j)\right|+\left|f_2(L^{-2}\lam_j)-   \ti f_2^{Q_j}(L^{-2}\lam_j)\right|
            \lesssim L^{-2k}+C(L)\sum_{\ell=1}^{N}\left[\left(\frac{\lam_{\ell}}{\lam_{\ell+1}}\right)^k+\left|\te_{\ell}-\te_{\ell}^{\infty}\right|\right],
    \end{align}
    which combined with ~\eqref{eq:derivative-core} and Gronwall's inequality yields for $m=1,2$,
    \begin{align}
        &\sup_{L^{-2}\lam_j\leq r\leq L^2\lam_j}|f_m(r)-\ti f_m^{Q_{j}}(r)|\\
        \lesssim &e^{\int_{L^{-2}\lam_{j}}^{L^2\lam_{j}}|\p_r\calU|\ud r}\left[ L^{-2k}+C(L)\sum_{\ell=1}^{N}\left[\left(\frac{\lam_{\ell}}{\lam_{\ell+1}}\right)^k+\left|\te_{\ell}-\te_{\ell}^{\infty}\right|\right] \right]\\
        &+e^{\int_{L^{-2}\lam_{j}}^{L^2\lam_{j}}|\p_r\calU|\ud r}\left[ \int_{L^{-2}\lam_{j}}^{L^2\lam_j}\left[\left|\calU-(-1)^{{j}-1}e^{\te_{j} R}Q_{\lam_{j}}\right|\left|\p_r\calU\right| 
        +\left|\p_r\left(\calU-(-1)^{{j}-1}e^{\te_{j} R}Q_{\lam_{j}}\right)\right|\right]\ud r  \right].
    \end{align}
    Using ~\eqref{eq:Q-cal-expansion-local}, ~\eqref{eq:Q-cal-norm}, ~\eqref{eq:frame-estimate-pointwise} we obtain
    \begin{align}
        \sup_{L^{-2}\lam_j\leq r\leq L^2\lam_j}|f_m(r)-\ti f_m^{Q_{j}}(r)|\lesssim L^{-2k}+C(L)\sum_{\ell=1}^{N}\left[\left(\frac{\lam_{\ell}}{\lam_{\ell+1}}\right)^k+\left|\te_{\ell}-\te_{\ell}^{\infty}\right|\right], 
    \end{align}
    which finish the proof of ~\eqref{eq:frame-compare-core} for $j$. 
    
    For $1\leq j\leq N$, if ~\eqref{eq:frame-compare-core} holds for $j$, then we have
    \begin{align}
         \left|f_1(L^{2}\lam_j)-  \ti f_1^{Q_j}(L^{2}\lam_j)\right|+\left|f_2(L^{2}\lam_j)-    \ti f_2^{Q_j}(L^{2}\lam_j)\right| 
            \lesssim  L^{-2k}+C(L)\sum_{\ell=1}^{N}\left[\left(\frac{\lam_{\ell}}{\lam_{\ell+1}}\right)^k+\left|\te_{\ell}-\te_{\ell}^{\infty}\right|\right],
    \end{align}
    which combined with ~\eqref{eq:derivative-transition} and ~\eqref{eq:expansion-core-2} yields ~\eqref{eq:frame-compare-transition} for $j$. 

    Therefore the proof is finished by induction.
    
\end{proof}

Now we turn back to  the estimates of $\lam_j^{-9}\left|\La y,\chi_{L,j}\sin\phi(r/\lam_j) \Ra\right|^2+\lam_j^{-9}\left|\La y,i\chi_{L,j}\sin\phi(r/\lam_j) \Ra\right|^2$.\\
We have
\begin{align}
    \int_{0}^{L^{-2}\lam_j} |y||\sin\phi(r/\lam_j)|\udr\lesssim \|y\|_{\rm{Mor}(\lam_j)}\left(\int_0^{L^{-2}\lam_j} F_j^2r^6\frac{(\lam_j+r)^2}{\lam_j}\udr \right)^{\frac{1}{2}}\lesssim L^{-2k-8} \lam_j^{\frac{9}{2}}\|y\|_{\rm{Mor}(\lam_j)},
\end{align}
\begin{align}
     \int_{L^{-2}\lam_j}^{L^{2}\lam_j} |y||\sin\phi(r/\lam_j)|\udr\lesssim \|y\|_{\rm{Mor}(\lam_j)}\left(\int_{L^{-2}\lam_j}^{L^{2}\lam_j} F_j^2r^6\frac{(\lam_j+r)^2}{\lam_j}\udr \right)^{\frac{1}{2}}\lesssim L^4 \lam_j^{\frac{9}{2}}\|y\|_{\rm{Mor}(\lam_j)},
\end{align}
which combined with ~\eqref{eq:frame-compare-core} yields
\begin{align}
  &\lam_j^{-9}\left|\La y,\chi_{L,j}\sin\phi(r/\lam_j) \Ra\right|^2+\lam_j^{-9}\left|\La y,i\chi_{L,j}\sin\phi(r/\lam_j) \Ra\right|^2\\
  \lesssim &  \lam_j^{-9}\left|\La \ti T_{\calU} y,\chi_{L,j}e^{\te_j R}\Lam Q_{\lam_j} \Ra\right|^2+\lam_j^{-9}\left|\La \ti T_{\calU} y,\chi_{L,j}e^{\te_j R}R Q_{\lam_j}\Ra\right|^2\\
  &+ \left(L^{4-2k}+C(L)o_{T_0\to -\infty}(1)\right)\|y\|_{\rm{Mor}(\lam_j)}^2.
\end{align}

Recall that $y=z+\eta$ with $\eta$ satisfying ~\eqref{eq:eta-ortho}, $g$ satisfies ~\eqref{eq:g-ortho}, and $g=\sqrt{1-\frac{|z|^2}{4}}\ti T_{\calU}z-\frac{|z|^2}{2}\calU$ by definition,  we have
\begin{align}
    &\left|\La \ti T_{\calU} y,\chi_{L,j}e^{\te_j R}\Lam Q_{\lam_j} \Ra\right|+\left|\La \ti T_{\calU} y,\chi_{L,j}e^{\te_j R}R Q_{\lam_j}\Ra\right|\\
    =& \left|\Bigg< \left(1-\sqrt{1-\frac{|z|^2}{4}}\right)\ti T_{\calU} z+\frac{|z|^2}{2}\calU,\chi_{L,j}e^{\te_j R}\Lam Q_{\lam_j} \Bigg>\right|\\
    &+\left|\Bigg< \left(1-\sqrt{1-\frac{|z|^2}{4}}\right)\ti T_{\calU} z+\frac{|z|^2}{2}\calU,\chi_{L,j}e^{\te_j R}R Q_{\lam_j} \Bigg>\right|\\
    \lesssim & \int_0^{L^2\lam_j} |z|^2F_j\udr\lesssim_L \lam_j^6\|z\|_{\calH^3}^2.
\end{align}
Therefore
\begin{align}
    \lam_j^{-9}\left|\La \ti T_{\calU} y,\chi_{L,j}e^{\te_j R}\Lam Q_{\lam_j} \Ra\right|^2+\lam_j^{-9}\left|\La \ti T_{\calU} y,\chi_{L,j}e^{\te_j R}R Q_{\lam_j}\Ra\right|^2\lesssim_L \lam_j^3\|z\|_{\calH^3}^4\lesssim |t|^{-2-4\veps_0}.
\end{align}

Combining these estimates we obtain that for $L\gg 1$ sufficiently large and for $T_0\ll-1$ sufficiently negative after fixing $L$, we have
\begin{align}\label{eq:y-Ay-control}
    \|y\|_{\rm{Mor}(\lam_j)}^2\lesssim_L \|A_{\lam_j}y\|_{\rm{Mor}_2(\lam_j)}^2+o_{T_0\to-\infty}(1)|t|^{-2-2\veps_0}.
\end{align}

\begin{proof}[Proof of Proposition~\ref{prop:I3'}]

    The terms containing $\|z\|_{\rm{Mor}(\lam_j)}^2$ appearing in ~\eqref{eq:A-estimate-final}, ~\eqref{eq:B-estimate-final}, ~\eqref{eq:C-estimate-final}, ~\eqref{eq:E-estimate-final}  can be estimated into the form $\Ga(t)|\lam_j|\|z\|_{\rm{Mor}(\lam_j)}^2$ or $|t|^{-\frac{1}{2}-\veps_0}\lam_j|\|z\|_{\rm{Mor}(\lam_j)}^2$, which can be estimated using ~\eqref{eq:z-y-Mor} to obtain the bound $o_{T_0\to -\infty}(1)\|y\|_{\rm{Mor}(\lam_j)}^2+C(L)o_{T_0\to -\infty}(1)|t|^{-2-2\veps_0}$ where we used the choice $\veps_0<\frac{1}{2}$ in Proposition ~\ref{prop:bootstrap} .

    Combining the estimates ~\eqref{eq:A-estimate-final}, ~\eqref{eq:B-estimate-final}, ~\eqref{eq:C-estimate-final}, ~\eqref{eq:D-estimate-final}, ~\eqref{eq:E-estimate-final} ~\eqref{eq:F-estimate-final} and the coercivity estimate ~\eqref{eq:y-Ay-control}, we obtain
    \begin{align}
        \frac{\ud}{\ud t}\calI_3(t)\lesssim & -c_1\ti\de\sum_{j=1}^N \de_0^{N-j}\|A_{\lam_j}y\|_{\rm{Mor}_2(\lam_j)}^2\\
        &+\ti\de\left( o_{L\to\infty}(1)+C(L)o_{\de_0\to 0}(1)+C(L,\de_0)o_{T_0\to-\infty}(1)\right) \sum_{j=1}^N\de_0^{N-j}\|A_{\lam_j}y\|_{\rm{Mor}_2(\lam_j)}^2\\
        &+C(L)\left(\ti \de+C(\de_0)o_{T_0\to-\infty}(1)\right)|t|^{-2-2\veps_0}.
    \end{align}

    We first choose $L\gg 1$ sufficiently large and fix it, then choose $0<\ti \de,\de_0\ll 1$ sufficiently small depending on $L$, and finally take $T_0\ll -1$ sufficiently negative depending on $L,\de_0,\ti \de$. In particular, we obtain
    \begin{align}
        \frac{\ud}{\ud t}\calI_3(t)\leq \frac{c_L}{8\hat C}|t|^{-2-2\veps_0},
    \end{align}
   which finishes the proof.
   
\end{proof}

\section{Closing the bootstrap}\label{sec:completing-bootstrap}
In this section we prove Proposition ~\ref{prop:bootstrap} using the results obtained in Section ~\ref{sec:energy-estimate}~\ref{sec:mod} ~\ref{sec:high-Sobolev}. We fix the choice of $L$ and the choice of $\ti \de, \de_0$ in Proposition ~\ref{prop:I3'}.

\subsection{Improvement of $g$}

Using Corollary ~\ref{cor:g-improved}, we obtain that for $T_0<0$ sufficiently negative,
\begin{align}
    \|g\|_{\dot \calH^1}\lesssim |t|^{-\frac{1}{2}-\frac{1}{k-2}-\frac{\veps_1}{2}}\ll|t|^{-\frac{1}{2}}.
\end{align}

For $\eta$ constructed in Lemma~\ref{lem:correction-z-eta}, a same argument as in that of ~\eqref{eq:Mor-bounded} yields $|\cM(\eta(t),\eta(t))|\leq C(L)|t|^{-2}$

For $\|g\|_{\dot \calH^3}$, using ~\eqref{eq:C-hat-def}, Lemma ~\ref{lem:E3-coercive}, Proposition ~\ref{prop:I3'} and the  assumption $\veps_0<\min\{\frac{1}{2},\frac{1}{k-2}\}$ in Proposition ~\ref{prop:bootstrap},  we have
\begin{align}
    \|g(t)\|_{\dot\calH^3}^2\leq& \hat C\|z(t)\|_{\calH^3}^2\\
    \leq& \frac{\hat C}{c_L}\calI_3(t)+\frac{\hat C}{c_L|t|}\|z\|_{\calH^3}+\frac{\hat C}{c_L}|t|^{-1-\frac{3}{2}\veps_0-\frac{1}{2(k-2)}-\frac{\veps_1}{4}}\\
    \leq & \frac{\hat C}{c_L}\calI_3(T)+\frac{\hat C}{c_L}\int_T^{t} \frac{c_L}{8\hat C}|s|^{-2-2\veps_0}\ud s+\frac{\hat C}{c_L}\left(|t|^{-\frac{3}{2}-\veps_0}+ |t|^{-1-\frac{3}{2}\veps_0-\frac{1}{2(k-2)}-\frac{\veps_1}{4}} \right)\\
    \leq &\ \frac{1}{8}|t|^{-1-2\veps_0}+C(L)\hat C\left(|T|^{-2}+|t|^{-\frac{3}{2}-\veps_0}+ |t|^{-1-\frac{3}{2}\veps_0-\frac{1}{2(k-2)}-\frac{\veps_1}{4}}\right)\\
    \leq &\ \frac{1}{4}|t|^{-1-2\veps_0}
\end{align}
provided that $T_0<0$ is sufficiently negative.

\subsection{Improvement of the stable parameters}\label{subsec:stable}

We will use ~\eqref{eq:lam-hatlam-te-hatte-close}, ~\eqref{eq:lam-hat-modulation} and ~\eqref{eq:te-hat-modulation}.  

Under modulation assumption, we have
\begin{align}
    \frac{1}{\lam_j^2(t)}\left(\frac{\lam_j(t)}{\lam_{j+1}(t)}\right)^k&=\frac{\be_j}{\ka_k |t|}\left(1+O\left(|t|^{-\veps_j}\right)\right),\ 1\leq j\leq N-1\\
    \frac{1}{\lam_j^2(t)}\left(\frac{\lam_{j-1}(t)}{\lam_{j}(t)}\right)^k&\lesssim |t|^{-1-\frac{2}{k-2}}, \ 1\leq j\leq N.
\end{align}
The above estimates and ~\eqref{eq:lam-hat-modulation}  implies
\begin{align}
    |\hat \lam_N(t)-1|\lesssim o_{T_0\to-\infty}(1)\int_{T}^t |s|^{-1-\veps_N}\ud s\lesssim o_{T_0\to-\infty}(1)|t|^{-\veps_N}.
\end{align}
Therefore 
\begin{align}
    \left|\frac{\lam_N(t)}{\lam_{N}^{\infty}(t)}-1\right|&=\left|\lam_N(t)-1\right|\leq|\lam_N(t)-\hat\lam_N(t)|+|\hat \lam_N(t)-1|\lesssim |t|^{-\frac{1}{2}}+ o_{T_0\to-\infty}(1)|t|^{-\veps_N},
\end{align}
which proves ~\eqref{eq:lamN-improved} provided that $T_0<0$ is sufficiently negative.

~\eqref{eq:te-hat-modulation} yields 
\begin{align}
    \left|\hat \te_N'(t)\right|\lesssim o_{T_0\to -\infty}(1)|t|^{-1-\veps_N},
\end{align}
\begin{align}
    \hat \te_j'(t)=-\frac{k\be_j}{|t|}\left(\hat \te_j-\te_j^{\infty}\right)+o_{T_0\to-\infty}(1)|t|^{-1-\veps_j}.
\end{align}
Therefore for $1\leq j\leq N$, we have
\begin{align}
    \frac{\ud}{\ud t}\left[ |t|^{-k\be_j}\left(\hat\te_j-\te_j^{\infty}\right) \right]=o_{T_0\to-\infty}(1)|t|^{-1-\veps_j-k\be_j}.
\end{align}
Then we have
\begin{align}
    \left|\te_j-\te_j^{\infty}\right|\leq |\te_j-\hat \te_j|+|\hat \te_j-\te_j^{\infty}|\lesssim |t|^{-\frac{1}{2}}+o_{T_0\to-\infty}(1)|t|^{-\veps_j},
\end{align}
which proves ~\eqref{eq:te-improved} provided that $T_0<0$ is sufficiently negative.

\subsection{Topological argument for the unstable parameters}\label{subsec:unstable}
In this section we prove that we can choose $\lam_j^0$ such that the related solution exists on $[T,T_0]$ with ~\eqref{eq:lamN-bootstrap}, ~\eqref{eq:te-bootstrap}, ~\eqref{eq:g-H1-bootstrap}, ~\eqref{eq:g-H3-bootstrap} and ~\eqref{eq:lam-improved}     hold for $t\in [T,T_0]$. It suffices to prove the following lemma:

\begin{lem}\label{lem:topological-shooting}
    Under the assumption of Proposition ~\ref{prop:bootstrap}, there exist initial parameters $\lam_j^0$ with $1\leq j\leq N-1$ such that the related solution exists on $[T,T_0]$ with ~\eqref{eq:lamN-bootstrap}, ~\eqref{eq:te-bootstrap}, ~\eqref{eq:g-H1-bootstrap}, ~\eqref{eq:g-H3-bootstrap} and
    \begin{align}\label{eq:lam-improved-1}
        \left|\frac{\hat\lam_j(t)}{\lam_j^{\infty}(t)}-1\right|\leq \frac{1}{4}|t|^{-\veps_j}
    \end{align}
    hold for $t\in [T,T_0]$. 
\end{lem}

Once Lemma ~\ref{lem:topological-shooting} is proved,  the proof of Proposition ~\ref{prop:bootstrap} is completed with ~\eqref{eq:lam-hatlam-te-hatte-close} if we take $T_0\ll-1$ sufficiently negative.

\begin{proof}[Proof of Lemma ~\ref{lem:topological-shooting}]
     For  parameters $\lam_j^0$ with $1\leq j\leq N-1$, we consider the related solution.  Each group of  parameters $(p_1,\dots,p_{N-1})\in [-1,1]^{N-1}$ uniquely determines a group of parameters $(\lam_1(T),\dots,\lam_{N-1}(T))$ via
    \begin{align}
         \lam_j(T,p_1,\dots,p_{N-1})=\lam_j^{\infty}(T)\left(1+\frac{p_j}{4|T|^{\veps_j}}\right),\ 1\leq j\leq N-1.
    \end{align}
    To highlight the role $\{p_j\}_{1\leq j\leq N-1}$, we write the refined parameters as $\hat \lam_j(t,p_1,\dots,p_{N-1} )$ for $1\leq j\leq N-1$ if the case $j=N$ for ~\eqref{eq:lamN-bootstrap}, ~\eqref{eq:te-bootstrap}, ~\eqref{eq:g-H1-bootstrap}, ~\eqref{eq:g-H3-bootstrap} and ~\eqref{eq:lam-improved-1} hold on $[T,t]$.

    Suppose that Lemma ~\ref{lem:topological-shooting} fails. For each $\{\lam_j^0\}_{1\leq j\leq N-1}$, we consider the related solution $u$ defined on $[T,T_+)$. Let $T_2$ be the infimum of the values of $t\in[T,T_+)$ such that the case $j=N$ for ~\eqref{eq:lamN-bootstrap}, ~\eqref{eq:te-bootstrap}, ~\eqref{eq:g-H1-bootstrap},  ~\eqref{eq:g-H3-bootstrap} or ~\eqref{eq:lam-improved-1} fails. Then Lemma ~\ref{lem:Cauchy} yields $T_2$ exists and satisfied $T_2<T_0$. Otherwise there are two cases: one is that $T_2=T_+$, and another is $T_2\geq T_0$, both are impossible.   ~\eqref{eq:lamN-improved}, ~\eqref{eq:te-improved}, ~\eqref{eq:g-H1-improved} and ~\eqref{eq:g-H3-improved} yield that the case $j=N$ for  ~\eqref{eq:lamN-bootstrap}, ~\eqref{eq:te-bootstrap}, ~\eqref{eq:g-H1-bootstrap} and  ~\eqref{eq:g-H3-bootstrap} still hold at $T_2$.   Since $u$ is continuous in $\dot H^1$ at $T_2$, there exists $\rho_1,\rho_2>0$ such that $|u_3(t,r)|>\frac{1}{2}$ for all $r\in (0,\rho_1]\cup[\rho_2,\infty)$ and $t$ in a small neighborhood of $T_2$. Since $\|\p_t u\|_{L^2}=\|u\wedge\De u\|_{L^2}\lesssim\|u\|_{\dot H^1}^{\frac{1}{2}}\|u\|_{\dot H^3}^{\frac{1}{2}}$,  we have $u(t)\to u(T_2)$ in $L^2_{\rm{loc}}$ as $t\to T_2$, therefore $u_3(t)$ have the same sign near $T_2$ in $(0,\rho_1]$ or in $[\rho_2,\infty)$. Therefore using $|u|=1$ we obtain
    \begin{equation}\label{eq:H1-L2}
        \begin{aligned}
             &\int_0^{\infty}|u_3(T_2,r)-u_3(t,r)|^2r^{-1}\ud r\\
        \lesssim & \rho_1^{-2}\int_{\rho_1}^{\rho_2}|u_3(T_2,r)-u_3(t,r)|^2 \udr\\
        &+\sum_{\ell=1}^2\int_{(0,\rho_1]\cup[\rho_2,\infty) }\frac{\left(u_{\ell}\left(T_2,r\right)-u_{\ell}(t,r)\right)^2\left(u_{\ell}\left(T_2,r\right)+u_{\ell}(t,r)\right)^2}{\left(u_3(T_2,r)+u_3(t,r)\right)^2}r^{-1}\ud r\\
        \lesssim & \rho_1^{-2}\int_{\rho_1}^{\rho_2}|u_3(T_2,r)-u_3(t,r)|^2 \udr+\sum_{\ell=1}^2\int_{(0,\rho_1]\cup[\rho_2,\infty) }\frac{\left(u_{\ell}\left(T_2,r\right)-u_{\ell}(t,r)\right)^2 }{r^2}\udr,
        \end{aligned}
    \end{equation}
    which tends to $0$ as $t\to T_2$. Then the decomposition in Lemma ~\ref{lem:static-mod} still holds in a small neighborhood of $T_2$, which we denote by $I_2$. Since $u$ is continuous in $\dot H^1$, we can shrink $I_2$ if necessary to obtain that ~\eqref{eq:lamN-bootstrap} and  ~\eqref{eq:te-bootstrap} hold in $I_2$.  Since $u(t)=\calU_{\vec \lam(t),\vec \te(t)}+g(t)$ and the parameters $\vec \lam(t), \vec \te(t)$ are continuous, after shrinking $I_2$, we have ~\eqref{eq:g-H1-bootstrap} holds in $I_2$.  By interpolation we have $u(t)\to u(T_2)$ in $H_{\rm{loc}}^2$ as $t\to T_2$, then similarly ~\eqref{eq:g-H3-bootstrap} holds in $I_2$. Therefore, for some $1\leq j\leq N-1$,  ~\eqref{eq:lam-improved-1} holds with equality, which leads to a  map 
    \begin{align}
        \Phi: \left[-1,1\right]^{N-1}&\to \p\left[-1,1\right]^{N-1}\\
          (p_1,\dots,p_{N-1})&\mapsto \left(4|T_2|^{\veps_j}\left(\frac{\hat\lam_j(T_2,p_1,\dots,p_{N-1})}{\lam_j^{\infty}(T_2)}-1\right)\right)_{1\leq j\leq N-1}.
    \end{align}
 
        The map $\Phi$ is continuous due to Lemma~\ref{lem:Cauchy} and the same argument as in ~\eqref{eq:H1-L2}. The equations ~\eqref{eq:lam-hat-modulation} yields if $j\in\{1,\dots,N-1\}$ such that $4|T_2|^{\veps_j}\left(\frac{\hat\lam_j(T_2,p_1,\dots,p_{N-1})}{\lam_j^{\infty}(T_2)}-1\right)=\pm1$, then ~\eqref{eq:lam-hat-modulation} yields
    \begin{align}
        &4|T_2|^{\veps_j}\left(\frac{\hat\lam_j(T_2,p_1,\dots,p_{N-1})}{\lam_j^{\infty}(T_2)}-1\right)\frac{\ud }{\ud t}\Bigg|_{t=T_2}\left(4|t|^{\veps_j}\left(\frac{\hat\lam_j(t,p_1,\dots,p_{N-1})}{\lam_j^{\infty}(t)}-1\right)\right)\\
        =&\frac{(k-2)\be_j-\veps_j+o_{T_0\to-\infty}(1)}{|T_2|}\left[4|T_2|^{\veps_j}\left(\frac{\hat\lam_j(T_2,p_1,\dots,p_{N-1})}{\lam_j^{\infty}(T_2)}-1\right)\right]^2>0,
    \end{align}
    which combined with  $\lam_j(T)=\hat\lam_j(T)$ implies $\Phi|_{ \p\left[-1,1\right]^{N-1}}=\rm{Id}$. This contradicts with Brouwer fixed point theorem.

\end{proof}

\section{Construction by compactness}

\begin{proof}[Proof of Theorem ~\ref{thm:main}]

Let $T_0\ll-1$ given in  Proposition ~\ref{prop:bootstrap}. Take $T_0>T_1>T_2>\cdots$ such that $\lim_{n\to \infty} T_n=-\infty$. For each $n\geq 1$, Proposition ~\ref{prop:bootstrap} yields a solution $u_n$ satisfies
\begin{align}
    \left\|u_n(t)-\calU_{\vec\lam^{\infty}(t),\vec \te^{\infty}}\right\|_{\dot\calH^1}\lesssim |t|^{-\veps_1}.
\end{align}
For any $\ti T<T_0$, there exists a constant $C_{\ti T}$ such that the solutions $\{u_n\}_{n\geq 1}$ satisfy
\begin{align}\label{eq:un-uniform-bounded}
    \sup_{\left\{n:T_n<\ti T\right\}}\left\|u_n\right\|_{L_t^{\infty}\left([\ti T,T_0], \dot H_x^1\cap \dot H_x^3 \right)}\leq C_{\ti T}<\infty.
\end{align}
Therefore for $t\in [\ti T,T_0]$ and $n$ such that $T_n<\ti T$, we have
\begin{align}
    \left\|\p_t u_n(t)\right\|_{L^2}=\left\|u_n(t)\wedge\De u_n(t)\right\|_{L^2}\leq \left\|\De u_n(t)\right\|_{L^2}\leq\|u_n(t)\|_{\dot H^1}^{\frac{1}{2}}\|u_n(t)\|_{\dot H^3}^{\frac{1}{2}}\leq C_{\ti T}.
\end{align}
Fix $R>0$, for $\ti T\leq t_1<t_2\leq T_0$ and $n$ such that $T_n<\ti T$, we have 
\begin{align}
    \|u_n(t_1)-u_n(t_2)\|_{L^2_x(B(0,R))}&\lesssim \int_{t_1}^{t_2}\|\p_t u_n(t)\|_{L_x^2(B(0,R))}\ud t\lesssim_R  C_{\ti T}|t_1-t_2|.
\end{align}

Arzel\`a-Ascoli theorem yields that $\{u_n\}_{n\geq 1}$ has a subsequence converging in $ C_t\left([\ti T,T_0];L^2_x(B(0,R))\right) $. Standard diagonal argument yields a subsequence, still denoted by $\{u_n\}_{n\geq 1}$, converges in\\
$C_{t,\rm{loc}}\left((-\infty,T_0];L_{x,\rm{loc}}^2\right)$. ~\eqref{eq:un-uniform-bounded} yields that for any fixed $t<T_0$, $\{u_n\}_{n\geq 1}$ has a sequence converge weakly in $\dot H_x^1\cap \dot H_x^3$, therefore  there exists a $k$-equivariant function $u\in C_{t,\rm{loc}}\left((-\infty,T_0];L_{x,\rm{loc}}^2\right)\cap L_{t,\rm{loc}}^{\infty}\left((-\infty,T_0];\dot H^1_x\cap \dot H^3_x\right) $ such that 
\begin{align}
    u_n\to u\quad \rm{in}\quad C_{t,\rm{loc}}\left((-\infty,T_0];L_{x,\rm{loc}}^2\right).
\end{align}
Then $u$ satisfies
\begin{align}
    \p_t u=u\wedge \De u
\end{align}
in the sense of distribution.  For any $R>0$, and $t\leq T_0$, 
\begin{align}
    \||u(t)|-1\|_{L^2(B(0,R))}\leq \|u(t)-u_n(t)\|_{L^2(B(0,R))}\to 0,
\end{align}
therefore $|u(t,x)|=1$ for a.e. $x$. Lemma ~\ref{lem:Cauchy} yields that $u\in C_{t,\rm{loc}}\left((-\infty,T_0];\dot H_x^1\cap\dot H_x^3\right)$.   Since $u_n\to u$ in $C_{t,\rm{loc}}\left((-\infty,T_0];L_{x,\rm{loc}}^2\right)$,  for any $R>0$,
\begin{align}
    \left\|u(t)-\calU_{\vec \lam^{\infty}(t),\vec \te^{\infty}}\right\|_{\dot \calH^1\left(R^{-1}<r<R\right)}&\leq\liminf_{n\to \infty}\left\|u_n(t)-\calU_{\vec \lam^{\infty}(t),\vec \te^{\infty}}\right\|_{\dot \calH^1\left(R^{-1}<r<R\right)}\\
    &\leq \liminf_{n\to \infty}\left\|u_n(t)-\calU_{\vec \lam^{\infty}(t),\vec \te^{\infty}}\right\|_{\dot \calH^1\left(\R^2\right)}\\
    &\lesssim |t|^{-\veps_1}.
\end{align}
Since this bound is uniform in $R$, let $R\to \infty$, the monotone convergence theorem yields
\begin{align}
     \left\|u(t)-\calU_{\vec \lam^{\infty}(t),\vec \te^{\infty}}\right\|_{\dot \calH^1}\lesssim |t|^{-\veps_1}\to 0 \  \mas \  t\to-\infty.
\end{align}

We also have $\calU_{\vec \lam^{\infty}(t),\vec \te^{\infty}}-\calQ_{\vec \lam^{\infty}(t),\vec \te^{\infty}}\to 0 $ in $\dot \calH^1$, which finishes the proof.   
    
\end{proof}

\appendix

\section{Some technical estimates}\label{Appendix:estimates}
In this section, we collect all the estimates for quantities with coherent structure that are used in the main text.  Recall the definition of $F_j(t,r)$ in ~\eqref{eq:def-Fj}.

\begin{lem}
    \label{lem:frame-estimate}
    For the orthogonal frame $(f_1,f_2,\calU)$ we used in Section~\ref{sec:high-Sobolev}, there holds
    \begin{align}
        \sum_{j=1}^4 \left|(r\p_r)^j\calU(t,\te,r)\right|+\sum_{j=1}^3 \left|(r\p_r)^jf_1(t,\te,r)\right|+\sum_{j=1}^3 \left|(r\p_r)^jf_2(t,\te,r)\right|&\lesssim \sum_{j=1}^NF_j(t,r), \label{eq:frame-estimate-pointwise}\\
        \sum_{j=0}^3 \left|(r\p_r)^j\p_t\calU(t,\te,r)\right|+\sum_{j=1}^3 \left|(r\p_r)^j\p_tf_1(t,\te,r)\right|+\sum_{j=1}^3 \left|(r\p_r)^j\p_tf_2(t,\te,r)\right|&\lesssim \Ga(t)\sum_{j=1}^NF_j(t,r),\label{eq:frame-t-estimate-pointwise}\\
        |\p_t f_1(t,\te,r)|+|\p_t f_2(t,\te,r)|&\lesssim \Ga(t).\label{eq:frame-t-estimate-pointwise-2}
    \end{align}
\end{lem}
\begin{proof}
    The estimates for $\calU$ follows from ~\eqref{eq:Q-rescaled-expansion} and ~\eqref{eq:Q-cal-norm}. For $f_1$, recall the definition ~\eqref{eq:frame-def}, we have
    \begin{align}
        r\p_r f_1=&-(f_1\cdot r\p_r\calU)\calU, \\
        (r\p_r)^2 f_1=&-(f_1\cdot (r\p_r)^2\calU)\calU-(f_1\cdot r\p_r\calU)r\p_r\calU-((r\p_r f_1)\cdot r\p_r\calU)\calU,
    \end{align} 
    which yields the estimates for $j=1,2$. The cases $j=3$ follows in the same way.
    
    The estimates for $\p_t\calU$ follow from direct calculation. For $\p_t f_1$, we have
    \begin{align}
        \p_r \p_tf_1=-(\p_tf_1\cdot \p_r\calU)\calU-(f_1\cdot \p_t\p_r\calU)\calU-(f_1\cdot \p_r\calU)\p_t\calU,
    \end{align}
    which yields
    \begin{align}
        |\p_r\p_t f_1|\lesssim \frac{1}{r}\sum_{j=1}^NF_j|\p_t f_1|+\frac{\Ga(t)}{r}\sum_{j=1}^NF_j+\frac{\Ga(t)}{r}\left(\sum_{j=1}^NF_j\right)^2.
    \end{align}
    Note that $\p_t f_1(t,0)=0$, the estimates for $\p_t f_1$ follows from Gronwall's inequality.
    For $(r\p_r)^j \p_tf_1$, we have
    \begin{align}
        r\p_r \p_tf_1=-(\p_tf_1\cdot r\p_r\calU)\calU-(f_1\cdot \p_t(r\p_r\calU))\calU-(f_1\cdot r\p_r\calU)\p_t\calU,
    \end{align}
    which implies the estimates for $j=1$. The cases $j=2,3$ follow in the same way. 

     For $f_2$, since $f_2=\calU\wedge f_1$, we have
    \begin{align}
        \p_r f_2=-(f_2\cdot \p_r\calU)\calU,
    \end{align}
    therefore the proof for $f_2$ is analogous and we omit the details.

\end{proof}

\begin{lem}\label{lem:F-estimate}
    For $T_{\calU}(\calU\wedge\De\calU)$ we have the point-wise bounds
    \begin{align}
        \sum_{j=0}^3 |(r\p_r)^j(r^2T_{\calU}(\calU\wedge\De\calU))|&\lesssim \sum_{1\leq m< \ell\leq N}F_m(t,r)F_{\ell}(t,r),\label{eq:F-pointwise}\\
        \sum_{j=0}^3 |(r\p_r)^j(r^2\p_t(T_{\calU}(\calU\wedge\De\calU)))|&\lesssim \Ga(t)\sum_{1\leq m< \ell\leq N}F_m(t,r)F_{\ell}(t,r).\label{eq:Ft-pointwise}
    \end{align}
    There holds
    \begin{align}
        \|T_{\calU}(\calU\wedge\De\calU)\|_{\calH^1}&\lesssim \left(\sum_{j=1}^N\frac{1}{\lam_j^4}\left(\left(\frac{\lam_{j-1}}{\lam_{j}}\right)^{2k}+\left(\frac{\lam_{j}}{\lam_{j+1}}\right)^{2k}\right)\right)^{\frac{1}{2}},\label{eq:F-H1}\\
         \|T_{\calU}(\calU\wedge\De\calU)\|_{(\calH^1)^*}&\lesssim \sum_{j=1}^N\left(\frac{\lam_{j}}{\lam_{j+1}}\right)^{k},\label{eq:F-H-1}\\
         \|\p_t\left(T_{\calU}(\calU\wedge\De\calU)\right)\|_{\calH^1}&\lesssim \Ga(t)\left(\sum_{j=1}^N\frac{1}{\lam_j^4}\left(\left(\frac{\lam_{j-1}}{\lam_{j}}\right)^{2k}+\left(\frac{\lam_{j}}{\lam_{j+1}}\right)^{2k}\right)\right)^{\frac{1}{2}},\label{eq:Ft-H1}\\
         \|\p_t\left(T_{\calU}(\calU\wedge\De\calU)\right)\|_{(\calH^1)^*}&\lesssim \Ga(t)\sum_{j=1}^N\left(\frac{\lam_{j}}{\lam_{j+1}}\right)^{k}.\label{eq:Ft-H-1}
    \end{align}
\end{lem}
\begin{proof}
    The estimates ~\eqref{eq:F-pointwise} and ~\eqref{eq:Ft-pointwise} follow directly from ~\eqref{eq:U-wedge-De-U-expansion-local}, ~\eqref{eq:frame-estimate-pointwise}, ~\eqref{eq:frame-t-estimate-pointwise} and ~\eqref{eq:frame-t-estimate-pointwise-2}. Then the estimates ~\eqref{eq:F-H1} and ~\eqref{eq:Ft-H1} follow from these two pointwise bounds. For ~\eqref{eq:F-H-1}, using ~\eqref{eq:U-wedge-De-U-expansion-local} we have
    \begin{align}
        &\left|\La  T_{\calU}(\calU\wedge\De\calU),h   \Ra\right|\\
        \lesssim & \sum_{j=1}^N\int_{\sqrt{\lam_{j-1}\lam_j}}^{\sqrt{\lam_j\lam_{j+1}}}\left(F_j^2\sum_{\ell\neq j}F_{\ell}+\sum_{\ell\neq j}F_{\ell}^2\right)|h|\frac{\ud r}{r}\\
        \lesssim & \sum_{j=1}^N\left[\sum_{\ell\neq j}\left(\int_{\sqrt{\lam_{j-1}\lam_j}}^{\sqrt{\lam_j\lam_{j+1}}}\left(F_j^2F_{\ell}+F_{\ell}^2\right)^2\frac{\ud r}{r}\right)^{\frac{1}{2}}\right]\left\|\frac{h}{r}\right\|_{L^2(\udr)}\\
        \lesssim & \left(\sum_{j=1}^{N-1} \left(\frac{\lam_j}{\lam_{j+1}}\right)^k\right)\|h\|_{\dot\calH^1}.
    \end{align}
    The estimate ~\eqref{eq:Ft-H-1} follows in the same way.
\end{proof}

\begin{lem}
    \label{lem:F-tangent-estimate}
    There holds
    \begin{align}
        \sum_{\ell=0}^3\left|(r\p_r)^{\ell}\left(\frac{1}{\lam_j}T_{\calU}\left(\p_{\lam_j}\calU\right)\right)\right|+\sum_{\ell=0}^3\left|(r\p_r)^{\ell}\left(\frac{1}{\lam_j^2}T_{\calU}\left(\p_{\te_j}\calU\right)\right)\right|&\lesssim \frac{1}{\lam_j^2}F_j(t,r)\label{eq:U-tangent-pointwise}\\
        \sum_{\ell=0}^3\left|(r\p_r)^{\ell}\p_t\left(\frac{1}{\lam_j}T_{\calU}\left(\p_{\lam_j}\calU\right)\right)\right|+\sum_{\ell=0}^3\left|(r\p_r)^{\ell}\p_t\left(\frac{1}{\lam_j^2}T_{\calU}\left(\p_{\te_j}\calU\right)\right)\right|&\lesssim \Ga(t)\frac{1}{\lam_j^2}F_j(t,r).\label{eq:U-t-tangent-pointwise}
    \end{align}

    There holds
    \begin{align}
        \|\lam_j^{-1}T_{\calU}(\p_{\lam_j}\calU)\|_{X^*}+\|\lam_j^{-2}T_{\calU}(\p_{\te_j}\calU)\|_{X^*}&\lesssim o_{L\to\infty}(1)+\left(\frac{\lam_{j-1}}{\lam_{j}}\right)^k+\left(\frac{\lam_{j}}{\lam_{j+1}}\right)^k.\label{eq:F-p-X}
    \end{align}
    Here $\|w\|_{X^*}:=\sup_{v\in X,\|v\|_{\calH^1}=1}|\La w,v\Ra|$.
    
\end{lem}
\begin{proof}
    The estimates ~\eqref{eq:U-tangent-pointwise} and ~\eqref{eq:U-t-tangent-pointwise} follow directly from ~\eqref{eq:partialU-expression}, ~\eqref{eq:frame-estimate-pointwise}, ~\eqref{eq:frame-t-estimate-pointwise} and ~\eqref{eq:frame-t-estimate-pointwise-2}.

    For ~\eqref{eq:F-p-X}, since $v\in X$, we have
    \begin{align}
        \La \lam_j^{-1}T_{\calU}(\p_{\lam_j}\calU),v    \Ra&=\frac{(-1)^j}{\lam_j^2}\La \frac{1}{|\calQ|} e^{\te_j R}\Lam Q_{\lam_j}, \ti T_{\calU}v \Ra=\frac{(-1)^j}{\lam_j^2}\La \left(\frac{1}{|\calQ|} -\chi_{L,j} \right)e^{\te_j R}\Lam Q_{\lam_j}, \ti T_{\calU}v \Ra,\\
        \La \lam_j^{-2}T_{\calU}(\p_{\te_j}\calU),v    \Ra&=\frac{(-1)^{j-1}}{\lam_j^2}\La \frac{1}{|\calQ|} e^{\te_j R}R Q_{\lam_j}, \ti T_{\calU}v \Ra=\frac{(-1)^{j-1}}{\lam_j^2}\La \left(\frac{1}{|\calQ|} -\chi_{L,j} \right)e^{\te_j R}R Q_{\lam_j}, \ti T_{\calU}v \Ra.
    \end{align}
    Write $\frac{1}{|\calQ|} -\chi_{L,j}= \frac{1-\chi_{L,j}}{|\calQ|} +\chi_{L,j}(\frac{1}{|\calQ|}-1)$, using ~\eqref{eq:Q-cal-norm}, the expressions of $\Lam Q,RQ$, Cauchy's inequality and Hardy's inequality yields the estimate ~\eqref{eq:F-p-X}.
\end{proof}

\begin{lem}
    \label{lem:operator-estimate}
    For $1\leq j\leq N$ and $w(r)\in \C$, there holds
    \begin{align}
        |\p_tA_{\lam_j}|&\lesssim \frac{|\lam_j'|}{\lam_j}F_j^2(t,r)\frac{1}{r},\label{eq:A-t-estimate-pointwise}\\
        |\p_r\p_tA_{\lam_j}|&\lesssim \frac{|\lam_j'|}{\lam_j}F_j^2(t,r)\frac{1}{r^2},\label{eq:A-rt-estimate-pointwise}
    \end{align}
    For $1\leq j\leq N$ and $w(r)\in \C$, there holds
\begin{align}
    |\Lam_{\lam_j}w|\lesssim \frac{r}{\lam_j+r}|\p_r w|+\frac{|w|}{\lam_j+r},\label{eq:Lambda-estimate-pointwise}\\
    |\p_t\Lam_{\lam_j}w|\lesssim\frac{|\lam_j'|}{\lam_j}\frac{\lam_j}{(\lam_j+r)^2}(|r\p_r w|+|w|). \label{eq:Lambda-t-estimate-pointwise}
\end{align}

    For $0\leq a\leq 3$, there holds
    \begin{align}
         |((r\p_r)^a(r^2V_{j})) |\lesssim F_j^2(t,r),\ |((r\p_r)^a(r^2\p_tV_{j})) |\lesssim \frac{|\lam_j'|}{\lam_j}F_j^2(t,r). \label{eq:Vj-estimate-pointwise}
    \end{align}

    For $0\leq a\leq 3$ and $w\in \C$, there holds
    \begin{align}
        |((r\p_r)^a(r^2V_{\calU})) (w)|\lesssim |w|,\ |((r\p_r)^a(r^2\p_tV_{\calU})) (w)|\lesssim \Ga(t)|w|\label{eq:VU-bounded-pointwise}
    \end{align}

    For $0\leq a\leq 3$ and $w\in \C$, there holds
    \begin{align}
        |((r\p_r)^a(r^2\E_{\calU})) (w)|&\lesssim \sum_{1\leq m< \ell\leq N}F_m(t,r)F_{\ell}(t,r) |w|,\label{eq:EU-estimate-pointwise}\\
        |((r\p_r)^a(r^2\p_t\E_{\calU})) (w)|&\lesssim \Ga(t)\sum_{1\leq m< \ell\leq N}F_m(t,r)F_{\ell}(t,r) |w|,\label{eq:EU-t-estimate-pointwise}\\
        |[\calH_{\calU},i]w|=|[\calE_{\calU},i]w|&\lesssim \frac{1}{r^2}\sum_{1\leq m< \ell \leq N}F_m(t,r)F_{\ell}(t,r) |w|.\label{eq:commutator-pointwise}
    \end{align}
\end{lem}
\begin{proof}
    The proof of these estimates follows from explicit expressions of the operator and the expansion of $\calQ$ in the region $\{\sqrt{\lam_{j-1}\lam_j}\leq r\leq \sqrt{\lam_j\lam_{j+1}}\}$, we omit the details.
\end{proof}

\section{Proof of Morawetz estimates}\label{Appendix:Mor}
In this section we will show the proof of Lemma ~\ref{lem:Morawetz-one}. We will first give a proof of non-localized version.

\begin{lem}\label{lem:Morawetz-nonlocal}
    There holds
    \begin{align}\label{eq:Morawetz-one-nonlocal-Appendix}
       \La \ti H_{1}^2y,\ULam y\Ra\gtrsim 2\pi\int_0^{\infty}\left(|\p_r^2 y|^2+\frac{|\p_r y|^2}{r^2}+\frac{|y|^2}{r^4}\right)\udr.
    \end{align}
\end{lem}
\begin{proof}
    We omit the subscript $1$ for convenience. The notation used in this proof is independent of other sections. We have
    \begin{align}
        \La \ti H^2y,\ULam y\Ra&=\La \ti Hy,\ULam\ti H y\Ra+\La \ti H y,[\ti H,\ULam]y \Ra\\
        &=2\|\ti H y\|_{L^2}^2+2\La \ti H y,\frac{k^2\sin^2\phi(r)}{r^2}y \Ra.
    \end{align}
    Let $x=\log r$, $ v(x)=\frac{y(e^x)}{e^x}$, $a(x)=k\cos\phi(e^x)$, then $\p_x a=a^2-k^2$, and
    \begin{align}
        \frac{1}{4\pi} \La \ti H^2y,\ULam y\Ra&=\int_{\R}\left|-\p_x^2v-2\p_xv+(k^2+2a)v\right|^2\ud x\\
        &\quad+\Re\int_{\R}\left(-\p_x^2v-2\p_xv+(k^2+2a)v\right)(k^2-a^2)\bar v\ud x\\
        &=\int_{\R}\left|-\p_x^2v-2\p_xv+\left(k^2+2a+\frac{k^2-a^2}{2}\right)v\right|^2\ud x\\
        &\quad -\frac{1}{4}\int_{\R}(k^2-a^2)^2|v|^2\ud x.
    \end{align}
    We have
    \begin{align}\label{eq:appendix-B-1}
        \Re \int_{\R}\left(-\p_x^2v-2\p_xv+\left(k^2+2a+\frac{k^2-a^2}{2}\right)v\right)\bar v\ud x&=\int_{\R}\left(|\p_x v|^2+\left(k^2+2a+\frac{k^2-a^2}{2}\right)|v|^2\right)\ud x.
    \end{align}
    Since $-k\leq a\leq k$ and $k\geq 3$, we have $k^2+2a+\frac{k^2-a^2}{2}\geq k^2-2k>0 $, which yields
    \begin{align}
       \Re \int_{\R}\left(-\p_x^2v-2\p_xv+\left(k^2+2a+\frac{k^2-a^2}{2}\right)v\right)\bar v\ud x\geq k(k-2)\int_{\R}|v|^2\ud x,
    \end{align}
    then
    \begin{equation}\label{eq:appendix-B-2}
        \begin{aligned}
            k^2(k-2)^2\int_{\R}|v|^2\ud x\leq \int_{\R}\left|-\p_x^2v-2\p_xv+\left(k^2+2a+\frac{k^2-a^2}{2}\right)v\right|^2\ud x.
        \end{aligned}
    \end{equation}   
    Since $k\geq 3$, we have
    \begin{align}
         k^2+2a+\frac{k^2-a^2}{2}-(k^2-a^2)=\frac{k^2}{2}+\frac{(a+2)^2}{2}-2\geq 0,
    \end{align}
    which yields
    \begin{align}
        \frac{1}{4}\int_{\R}(k^2-a^2)^2|v|^2\ud x &\leq \frac{k^2}{4}\int_{\R}(k^2-a^2)|v|^2\ud x\\
        &\leq \frac{k^2}{4}\int_{\R}\left( k^2+2a+\frac{k^2-a^2}{2}\right)|v|^2\ud x\\
        &\leq \frac{k^2}{4}\Re \int_{\R}\left(-\p_x^2v-2\p_xv+\left(k^2+2a+\frac{k^2-a^2}{2}\right)v\right)\bar v\ud x\\
        &\leq \frac{k^2}{4k(k-2)}\int_{\R}\left|-\p_x^2v-2\p_xv+\left(k^2+2a+\frac{k^2-a^2}{2}\right)v\right|^2\ud x\\
        &\leq \frac{3}{4}\int_{\R}\left|-\p_x^2v-2\p_xv+\left(k^2+2a+\frac{k^2-a^2}{2}\right)v\right|^2\ud x.
    \end{align}
    Then we only have to prove that
    \begin{align}
        \int_{\R}\left|-\p_x^2v-2\p_xv+\left(k^2+2a+\frac{k^2-a^2}{2}\right)v\right|^2\ud x\gtrsim \int_0^{\infty}\left(|\p_r^2 y|^2+\frac{|\p_r y|^2}{r^2}+\frac{|y|^2}{r^4}\right)\udr.
    \end{align}
    By definition of $v$ we have $\p_r y=\p_x v+v$, $\p_r^2 y=\frac{1}{r}(\p_x^2v+\p_x v)$, which yields
    \begin{align}
        \int_0^{\infty}\left(|\p_r^2 y|^2+\frac{|\p_r y|^2}{r^2}+\frac{|y|^2}{r^4}\right)\udr\lesssim\int_{\R}\left(|\p_x^2 v|^2+|\p_x v|^2+|v|^2\right)\ud x.
    \end{align}
    It follows from ~\eqref{eq:appendix-B-1} and ~\eqref{eq:appendix-B-2} that
    \begin{align}
        \int_{\R}\left(|\p_x v|^2+|v|^2\right)\ud x\lesssim\int_{\R}\left|-\p_x^2v-2\p_xv+\left(k^2+2a+\frac{k^2-a^2}{2}\right)v\right|^2\ud x,
    \end{align}
    which yields
    \begin{align}
        \int_{\R}|\p_x^2 v|^2\ud x\lesssim & \int_{\R}\left(|\p_x v|^2+|v|^2\right)\ud x+\int_{\R}\left|-\p_x^2v-2\p_xv+\left(k^2+2a+\frac{k^2-a^2}{2}\right)v\right|^2\ud x\\
        \lesssim & \int_{\R}\left|-\p_x^2v-2\p_xv+\left(k^2+2a+\frac{k^2-a^2}{2}\right)v\right|^2\ud x.
    \end{align}

\end{proof}

\begin{proof}[Proof of Lemma ~\ref{lem:Morawetz-one}]
    By scaling we may assume $\lam=1$. The notation used in this proof is independent of other sections. We still define
    \begin{align}
        x=\log r, \ v(x)=\frac{y(e^x)}{e^x},\ a(x)=k\cos \phi(e^x),\ b(x)=\frac{\varphi(e^x)}{e^x}=\frac{1}{\sqrt{1+(e^x/R)^2}}-\de\frac{1}{1+(e^x/R)^2}.
    \end{align}
    We have $\p_x a=a^2-k^2$. 
    By definition we have
     \begin{align}
         \frac{1}{2\pi}\La \ti H_1^2y,\Lam_1 y\Ra =\int_{\R}\left(\be_2(a,b)|\p_x^2v|^2+\be_1(a,b)|\p_x v|^2+\be_0(a,b)|v|^2\right)\ud x,
     \end{align}
    where 
    \begin{align}
        \be_2(a,b)&=2b+2\p_x b,\\
        \be_1(a,b)&=bB_k(a)+(2k^2+4a+4)\p_x b-4\p_x^2 b-\frac{3}{2}\p_x^3 b,\\
        \be_0(a,b)&=bC_k(a)+4(k^2-a^2)\p_x b-(k^2+a^2+4a+4)\p_x^2 b-\left(\frac{k^2}{2}+a+1\right)\p_x^3 b+\p_x^4 b+\frac{1}{4}\p_x^5 b
    \end{align}
    with  
     \begin{equation}
        \begin{aligned}
            B_k(a)&=2\left(-a^2+4a+3k^2+4\right),\\
       C_k(a)&=2\left(3a^4-8a^3-5k^2a^2+8a^2+12k^2a+3k^4-4k^2\right).
        \end{aligned}
    \end{equation}
    Note that when we replace $b$ by $1$, this identity is just
    \begin{align}
        \frac{1}{2\pi}\La \ti H^2y,\ULam y\Ra  =\int_{\R}\left(2|\p_x^2v|^2+B_k(a)|\p_xv|^2+C_k(a)|v|^2 \right)\ud x
    \end{align}
Note that $|a(x)+k|\lesssim_ke^{-2kx}$, for $j=0,1,2$. We define $\ti \be_j\coloneqq \be_j\left(-k,\frac{1}{\sqrt{1+(e^x/R)^2}}\right)$. Let $s=\frac{e^x}{R}$, then
    \begin{align}
        \ti\be_2=\frac{2}{(1+s^2)^{\frac{3}{2}}},
    \end{align}
    \begin{align}
        \ti\be_1=\frac{(4k^2-8k+3)s^6+(16k^2-32k+10)s^4+(20k^2-40k+68)s^2+(8k^2-16k+16)}{2(1+s^2)^{\frac{7}{2}}}.
    \end{align}
    For $k\geq 3$, it can be verified that the coefficients of $s$ in the numerator of $\be_1$ are strictly larger than $0$, which yields
    \begin{align}
        \ti\be_1\gtrsim \frac{1}{(1+s^2)^{\frac{1}{2}}}.
    \end{align}
    A direct calculation yields
    \begin{align}
        \ti\be_0=\frac{1}{4(1+s^2)^{11/2}}\Big[c_5 s^{10}+c_4 s^8+c_3s ^6+c_2 s^4+c_1 s^2+c_0\Big],
    \end{align}
    where
    \begin{align}
        c_5&=8k^4-32k^3+26k^2+12k-9,\\
        c_4&=8(5k^4-20k^3+17k^2+6k-9),\\
        c_3&=2(40k^4-160k^3+157k^2+6k-216),\\
        c_2&=4(20k^4-80k^3+89k^2-18k+144),\\
        c_1&=8k(5k^3-20k^2+23k-6),\\
        c_0&=8k^2(k-2)^2.
    \end{align}
    For $k\geq 3$, it can be verified that $c_j,\ j=0,\dots,5$ are strictly larger than $0$, which yields
    \begin{align}
        \ti\be_0\gtrsim \frac{1}{(1+s^2)^{\frac{1}{2}}}.
    \end{align}
Since for $j\geq 0$, $(s\p_s)^j(\frac{1}{1+s^2})\lesssim\frac{1}{1+s^2}$, there exists $\de(k)>0$ such that for $0<\de<\de(k)$ we still have
    \begin{align}
        \be_1(-k,b),\be_0(-k,b)\gtrsim_{\de} \frac{1}{(1+s^2)^{\frac{1}{2}}}.
    \end{align}
    For $\be_2$, we have
    \begin{align}
        \be_2(-k,b)=2\left(\frac{1}{(1+s^2)^{\frac{3}{2}}}-\de\frac{1-s^2}{(1+s^2)^2}\right)\gtrsim_{\de} \frac{2}{(1+s)^2}.
    \end{align}
    Since $|a(x)+k|\lesssim_ke^{-2kx}$, there exists $M\gg 1$ depending only on $k,\de$ such that for $x\geq M$, 
    \begin{align}\label{eq:be-j-outer}
        \be_j(x)\gtrsim_{\de} \frac{1}{(1+e^x/R)^2},\ j=0,1,2.
    \end{align}
    Fix such $M$, take $\ti M\gg 1$ to be determined later, and take $R\gg e^{M+\ti M}$, then we have uniformly in $\{-\infty<x<M+\ti M\}$,
    \begin{align}
        b=1-\de+O\left(\frac{e^{2x}}{R^2}\right),\ |\p_x^jb|
\lesssim \frac{e^{2x}}{R^2},\ j=1,\dots,5,
    \end{align}
    which yields
    \begin{equation}\label{eq:be-j-inner}
        \begin{aligned}
             \be_2&=2(1-\de)+o_{R\to\infty}(1),\\
    \be_1&=(1-\de)B_k(a)+o_{R\to\infty}(1),\\
    \be_0&=(1-\de)C_k(a)+o_{R\to\infty}(1)
        \end{aligned}
    \end{equation}
    uniformly in $\{-\infty<x<M+\ti M\}$.
    
    Take smooth cutoff functions $\chi_{\pm}$ such that $0\leq \chi_{\pm}\leq 1$ and
    \begin{align}
        \chi_+^2(x)+\chi_-^2(x)\equiv1,\ \chi_-(x)=\begin{cases}
            1,& -\infty<x<M,\\
            0,& M+\ti M<x<\infty.
        \end{cases},\ 
        |\p_x^j\chi_{\pm}(x)|\lesssim \ti M^{-j}.
    \end{align}
    Lemma ~\ref{lem:Morawetz-nonlocal} yields for suitable $w$, there holds
    \begin{align}
        \int_{\R}\left(2|\p_x^2w|^2+B_k(a)|\p_xw|^2+C_k(a)|w|^2 \right)\ud x \gtrsim\int_{\R}\left(|\p_x^2 w|^2+|\p_xw|^2+|w|^2\right)\ud x.
    \end{align}
    Take $w=\chi_-v$, ~\eqref{eq:be-j-inner} yields for $R$ sufficiently large,
    \begin{align}
        \int_{\R}\left(\be_2|\p_x^2(\chi_-v)|^2+\be_1|\p_x(\chi_-v)|^2+\be_0|\chi_-v|^2 \right)\ud x \gtrsim\int_{\R}\left(|\p_x^2 (\chi_-v)|^2+|\p_x(\chi_-v)|^2+|\chi_-v|^2\right)\ud x.
    \end{align}

    On the other hand, ~\eqref{eq:be-j-outer} yields
    \begin{align}
        &\int_{\R}\left(\be_2|\p_x^2(\chi_+v)|^2+\be_1|\p_x(\chi_+v)|^2+\be_0|\chi_+v|^2 \right)\ud x\\ \gtrsim&\int_{\R}\frac{1}{(1+e^x/R)^2}\left(|\p_x^2 (\chi_+v)|^2+|\p_x(\chi_+v)|^2+|\chi_+v|^2\right)\ud x.
    \end{align}
Therefore using $|\p_x^j\chi_{\pm}(x)|\lesssim \ti M^{-j}$ we have
    \begin{align}
        &\int_{\R}\left(\be_2|\p_x^2v|^2+\be_1|\p_x v|^2+\be_0|v|^2\right)\ud x\\
        \gtrsim&  \sum_{\iota=\pm}\int_{\R}\left(\be_2|\p_x^2(\chi_\iota v)|^2+\be_1|\p_x(\chi_{\iota} v)|^2+\be_0|\chi_{\iota}v|^2\right)\ud x\\
        &- o_{\ti M\to \infty}(1)\int_{M}^{M+\ti M}\left(|\p_x^2 v|^2+|\p_x v|^2+|v|^2\right)\ud x.
    \end{align}
    In the region $M\leq x\leq M+\ti M$, since $R\gg e^{M+\ti M}$, we have $(1+e^x/R)^{-2}= 1+o_{R\to\infty}(1)$, therefore we first take $\ti M$ sufficiently large and then take $R\gg e^{M+\ti M}$ to obtain
    \begin{align}
        \frac{1}{2\pi}\La \ti H_1^2y,\Lam_1 y\Ra\gtrsim & \int_{\R}\frac{1}{(1+e^x/R)^2}\left(|\p_x^2 v|^2+|\p_x v|^2+|v|^2\right)\ud x\\
        \gtrsim &\int_{\R}\frac{1}{(1+e^x)^2}\left(|\p_x^2 v|^2+|\p_x v|^2+|v|^2\right)\ud x\\
        \gtrsim&\|y\|_{\rm{Mor}_2(1)}^2,
    \end{align}
    where the last line follows from $\p_r y=\p_x v+v$, $\p_r^2 y=\frac{1}{r}(\p_x^2v+\p_x v)$

\end{proof}

\bibliographystyle{plain}
\bibliography{SM}

\end{document}